\documentclass[11pt,reqno]{amsart}

\usepackage{amsmath,amsfonts,amssymb,amscd,amsthm,amsbsy,bm,latexsym}
\usepackage{mathrsfs}
\usepackage{gensymb}
\usepackage[pdftex,bookmarksnumbered=true,bookmarksopen=true,          
            colorlinks=true,pdfborder=001,citecolor=blue,
            linkcolor=red,anchorcolor=green,urlcolor=blue]{hyperref}

\allowdisplaybreaks

\usepackage{color}

\newtheorem{lemma}{Lemma}[section]
\newtheorem{theorem}{Theorem}[section]

\newtheorem{proposition}{Proposition}[section]
\newtheorem{remark}{Remark}[section]

\numberwithin{equation}{section}

\newcommand{\R}{\mathbb{R}}

\newcommand{\T}{\mathbb{T}}

\newcommand{\FI}{\mathbf{I}}

\newcommand{\na}{\nabla}

\newcommand{\al}{\alpha}

\newcommand{\ga}{\gamma}
\newcommand{\om}{\omega}

\newcommand{\lag}{\langle}
\newcommand{\rag}{\rangle}

\newcommand{\md}{\mathrm{d}}

\begin{document}

\title[VPB system in the weakly collisonal regime]{The non-cutoff Vlasov-Poisson-Boltzmann system with weak collisions}

\author[Y.-J. Lei]{Yuanjie Lei}
\address[YJL]{School of Mathematics and Statistics, Huazhong University of Science and Technology, Wuhan 430074, China; Hubei Key Laboratory of Engineering Modeling and Scientific Computing, Huazhong University of Science and Technology, Wuhan 430074, China}
\email{leiyuanjie@hust.edu.cn}

\author[S.-Q. Liu]{Shuangqian Liu}
\address[SQL]{School of Mathematics and Statistics, Central China Normal University, Wuhan 430079, China; Hubei Key Laboratory of Mathematical Sciences, Central China Normal University, Wuhan 430079, China}
\email{sqliu@ccnu.edu.cn}

\author[Q.-H. Xiao]{Qinghua Xiao}
\address[QHX]{Innovation Academy for Precision Measurement Science and Technology, Chinese Academy of Sciences, Wuhan 430071, China}
\email{xiaoqh@apm.ac.cn}

\author[H.-J. Zhao]{Huijiang Zhao}
\address[HJZ]{School of Mathematics and Statistics, Wuhan University, Wuhan 430072, China; Computational Science Hubei Key Laboratory, Wuhan University, Wuhan 430072, China}
\email{hhjjzhao@whu.edu.cn}

\begin{abstract}
We prove global existence of smooth solutions near Maxwellians for the non-cutoff Vlasov-Poisson-Boltzmann system in the weakly collisional regime. To address the weak dissipation of the non-cutoff linearized Boltzmann operator, we develop a refined velocity-weighted energy framework combined with vector-field techniques to control the transport term, nonlinear collisions, and the self-consistent electric field. This approach yields uniform-in-time bounds, captures enhanced dissipation of the solution, and establishes Landau damping for both the density and electric field, providing the first global-in-time result of this type for the non-cutoff Vlasov-Poisson-Boltzmann system. Our approach is inspired by the recent work of Chaturvedi-Luk-Nguyen ({\it J. Amer. Math. Soc.} {\bf 36} (2023), no. 4, 1103--1189.)

\end{abstract}

\maketitle
\thispagestyle{empty}
\tableofcontents

\section{Introduction and main results}

\setcounter{equation}{0}

The Vlasov-Poisson-Boltzmann (VPB) system is a fundamental kinetic model describing the dynamics of dilute charged particles subject to self-consistent electrostatic forces and binary collisions. It provides a rigorous framework to capture the interplay between long-range mean-field interactions, represented by the Vlasov-Poisson coupling, and short-range collisional effects, modeled by the Boltzmann collision operator. In plasma physics and rarefied gas dynamics, such a system is particularly relevant in regimes where collisions are infrequent but still play a non-negligible role in the evolution of the distribution function.

In this paper, we consider the following Cauchy problem of the non-cutoff VPB system
\begin{align}\label{VPB}
\begin{cases}
 \displaystyle \partial_tF+ v  \cdot\nabla_xF-\nabla_x\phi\cdot\nabla_{ v  }F=\nu Q(F,F),\quad (x,v)\in \mathbb{T}^3\times\mathbb{R}^3,\\[2mm]
 -\Delta_x\phi= \displaystyle{\int_{\mathbb{R}^3}}F\,\mathrm{d}v-\frac{1}{(2\pi)^3}\displaystyle{\int_{\T^3\times\mathbb{R}^3}}F\,\mathrm{d}v\mathrm{d}x,\ \int_{\T^3}\phi \md x=0,\\[2mm]
\displaystyle F(0,x,v)=F_{0}(x,v).
 \end{cases}
\end{align}
Here, $\nabla_x=\left[\partial_{x_1}, \partial_{x_2},\partial_{x_3}\right], \nabla_v=\left[\partial_{v_1}, \partial_{v_2},\partial_{v_3}\right]$.  The unknown function $F= F(t,x, v) \geq  0$ is the number density functions with position $x = (x_1, x_2, x_3)\in \mathbb{T}^3=[0,2\pi]^3$ and velocity $ v=( v_1,  v_2,  v_3) \in {\mathbb{R}}^3$ at time $t\geq 0$. The constant $\nu>0$ represents the strength of the collisional interaction.

The Boltzmann collision operator $Q$ in \eqref{VPB} is given by
\begin{equation*}
  Q(F,G)(v)
  =\int_{\mathbb{R}^3\times \mathbb{S}^2}{\bf B}(v-u,\sigma)\{F(u')G(v')-F(u)G(v)\}
  \,\mathrm{d}\sigma\mathrm{ d}u,
\end{equation*}
where in terms of velocities $u$ and $v$ before the collision, velocities  $v'$ and $u'$ after the collision are defined by
\begin{equation*}
v'=\frac{v+u}{2}+\frac{|v-u|}{2}\sigma,\ \ \ \ u'=\frac{v+u}{2}-\frac{|v-u|}{2}\sigma.
\end{equation*}

The Boltzmann collision kernel ${\bf B}(v-u,\sigma)$ depends only on the relative velocity $|v-u|$ and on the
deviation angle $\theta$ given by $\cos\theta =\sigma\cdot\ (v-u)/{|v-u|}$. 
Throughout the paper, the collision
kernel is further supposed to satisfy the following assumptions:

\begin{itemize}
\item[(A1).] ${\bf B}(v-u,\sigma)$ takes the product form in its argument as
$$
{\bf B}(v-u,\sigma)=\Phi(|v-u|){\bf b}(\cos\theta)
$$
with $\Phi$ and ${\bf b}$ being non-negative functions;

\item[(A2).] The angular function $\sigma\rightarrow {\bf b}(\sigma\cdot\ (v-u)/{|v-u|})$ is not integrable on ${\mathbb{S}}^2$, i.e.
\[
\int_{{\mathbb{S}}^2}{\bf b}(\cos\theta)\,\mathrm{d}\sigma=2\pi\int_0^{\pi/2}\sin\theta\ {\bf b}(\cos\theta)\,\mathrm{d}\theta=\infty.
\]
Moreover, there are two positive constants $c_b>0, 0<s<1$ such that
\[
\frac{c_b}{\theta^{1+2s}}\leq\sin\theta {\bf b}(\cos\theta)\leq \frac{1}{c_b\theta^{1+2s}};
\]

\item[(A3).] The kinetic function $z\rightarrow \Phi(|z|)$ satisfies
\[
\Phi(|z|)=C_\Phi|z|^\gamma
\]
for some positive constant $C_\Phi> 0.$ The exponent $\gamma>-3$ is determined by the intermolecular interaction law.
\end{itemize}

It is convenient to call hard potentials when $\gamma+2s\geq0$ and
soft potentials when $-3<\gamma<-2s$ with $0<s<1$. In this paper, we restrict our attention to the following range:
\begin{equation}\label{case}
\max\left\{-3,-\frac32-2s\right\}<\gamma<-2s, \ \  0< s<1.
\end{equation}
From the physical viewpoint, this regime corresponds to soft long-range interactions between particles, such as those arising from inverse power law forces.

\subsection{The problem}
Set the normalized global Maxwellian
$
  \mu=\mu(v)=(2\pi)^{-{3}/{2}}e^{-| v  |^2/2},
$
and introduce the perturbation $f=f(t,x, v)$ by
$
F(t,x,v) = \mu+\mu^{\frac12}f(t, x,  v).
$
With this formulation, the Cauchy problem for the VPB system \eqref{VPB} can be rewritten as
\begin{equation} \label{f}
\begin{cases}
  \displaystyle\partial_tf+ v  \cdot\nabla_xf-\nabla_x\phi\cdot\nabla_{ v  }f+\nabla_x\phi\cdot v \mu^{\frac12}+ \frac12\nabla_x\phi\cdot v f+\nu\mathcal{L} f=\nu {\Gamma}(f,f), \\[2mm]
\displaystyle -\Delta_x\phi=\int_{\mathbb{R}^3}\mu^{\frac12} f\mathrm{d}v-\frac{1}{(2\pi)^3}\int_{\T^3\times\R^3}\mu^{\frac12}f_0(x,v)\mathrm{d}x\mathrm{d}v,\ \int_{\T^3}\phi \md x=0,
\quad (x,v)\in \mathbb{T}^3\times\mathbb{R}^3,\\
\end{cases}
\end{equation}
with initial datum
\begin{equation}\label{f-initial}
f(0,x,v)=f_{0}(x,v).
\end{equation}
Here, the linear collision operator $\mathcal{L} f$ and nonlinear collision operator ${ \Gamma}(f,g)$ are defined as
\begin{equation*}
\begin{aligned}
\mathcal{L} f =& -{\bf \mu}^{-\frac12}
\big\{{Q( \mu,{\bf \mu}^{\frac12}f)+ Q( \mu^{\frac12}f, \mu)}\big\},\\[2mm]
{ \Gamma}(f,g) =&{\bf \mu}^{-\frac12}Q\big({\bf \mu}^{\frac12}f,{\bf \mu}^{\frac12}g\big)+{\bf \mu}^{-\frac12}Q\big({\bf \mu}^{\frac12}g,{\bf \mu}^{\frac12}f\big).
\end{aligned}
\end{equation*}

For the linearized Boltzmann collision operator $\mathcal{L}$, it is well known \cite{Gressman_Strain-JAMS-2011} that it is non-negative
  and the null space of $\mathcal{L}$ is given by
\begin{equation*}
{\mathcal{ N}}={\textrm{span}}\left\{\mu^{\frac12}, v_i{\mu}^{\frac12} (1\leq i\leq3),|v|^2{\bf \mu}^{\frac12}\right\}.
\end{equation*}
Define ${\bf P}$ as the orthogonal projection from $ L^2({\mathbb{R}}^3_ v)$ to $\mathcal{N}$. For any given function $f(t, x, v )\in L^2({\mathbb{R}}^3_ v)$, one has
\begin{equation*}
  {\bf P}f ={a(t, x)\mu^{\frac12}+\sum_{i=1}^{3}b_i(t, x) v_i{\mu}^{\frac12}+c(t, x)(| v|^2-3)}{\bf \mu}^{\frac12}
\end{equation*}
with
\begin{equation*}
  a=\int_{{\mathbb{R}}^3}{\bf \mu}^{\frac12}f\,\mathrm{d} v,\quad
  b_i=\int_{{\mathbb{R}}^3} v  _i {\bf \mu}^{\frac12}f\,\mathrm{d} v,\quad
  c=\frac{1}{6}\int_{{\mathbb{R}}^3}(| v|^2-3){\bf \mu}^{\frac12}f \,\mathrm{d} v.
\end{equation*}

Therefore, for the given global Maxwellian $\mu$, we have the following macro-micro decomposition  \cite{Guo-IUMJ-04}:
\begin{equation}
 f(t,x, v)={\bf P}f(t,x, v)+{\bf P}^{\bot}f(t, x, v),\notag
\end{equation}
where  ${\bf P}f$ and ${\bf P}^{\bot}f$ correspond to the macroscopic and the microscopic component of $f(t,x,v)$, respectively.

Notice that
\[\int_{\mathbb{R}^3}\psi(v)\cdot{\bf P}^{\bot}f\,\mathrm{d}v=0,
\ \ \forall \psi(v)\in\mathcal{ N}.\]
For any fixed $(t,x)\in \R^+\times\mathbb{T}^3$, $\mathcal{L} f$ is self-adjoint and is known to be locally coercive
in the following sense \cite{AMUXY-JFA-2012,Gressman_Strain-JAMS-2011}:
\[\langle \mathcal{L} f,f\rangle\gtrsim \left|{\bf P}^{\bot} f\right|^2_{D},\]
where $|\cdot|_{D}$ is defined by
\[|f|^2_{D}=|f|^2_{L^2_{\gamma+2s}}+\int_{\mathbb{R}^3\times\mathbb{R}^3}(\langle v\rangle\langle v'\rangle)^{\frac{\gamma+2s+1}2}\frac{(f(v)-f(v'))^2}{d(v,v')^{3+2s}}\chi_{\{(v,v')|d(v,v')\leq1\}}
\, \mathrm{d}v'\mathrm{d}v\]
with
\[\langle v\rangle=\sqrt{1+|v|^2},\ \
d(v,v')=\sqrt{|v-v'|+\frac14(|v|^2-|v'|^2)^2},\]
and $\chi_{\Omega}$ is the standard indicator function of the set $\Omega$. Note that \begin{equation*}
|f|_{L^2_{\gamma+2s}}+|f|_{H^s_\gamma} \lesssim |f|_{D}\lesssim|f|_{H^s_{\gamma+2s}}.
\end{equation*}

Our main goal of this manuscript is to construct a global in time classical solution to the Cauchy problem \eqref{f} and \eqref{f-initial} for any $\nu>0$ under the condition \eqref{case}.
 \subsection{The notations}
Before stating our main results, we first introduce some notations used throughout the paper.\begin{itemize}
	\item $Y=t\nabla_{x}+\nabla_{v}$ is a commuting vector field: $[Y,\partial_t+v\cdot \nabla_{x}]=0$.
  \item
  $C$ denotes some positive constant (generally large) and $\eta,~\delta$ denote
some positive constant (generally small), where  $C$, $\eta$ and $\delta$ may take different values in different places.
\item
$A\lesssim B$ means that there is a generic constant $C> 0$ such that $A \leq   CB$. $A \sim B$ means $A\lesssim B$ and $B\lesssim A$.
\item The multi-indices $ \alpha= [\alpha_1,\alpha_2, \alpha_3]$, $\beta = [\beta_1, \beta_2, \beta_3]$, and $\omega = [\omega_1, \omega_2, \omega_3]$ will be used to record spatial, velocity derivatives, and vector field derivatives respectively. And $ \partial^\alpha_x\partial^\beta_v Y^\omega=\partial^{\alpha_1}_{x_1}\partial^{\alpha_2}_{x_2}\partial^{\alpha_3}_{x_3}
\partial^{\beta_1}_{ v_1}\partial^{\beta_2}_{ v_2}\partial^{\beta_3}_{ v_3}Y^{\omega_1}Y^{\omega_2}Y^{\omega_3}$. Similarly, the notation $\partial^{\alpha}$ will be used when $\beta=0$ and likewise for $\partial_{\beta}$. The length
of $\alpha$ is denoted by $|\alpha|=\alpha_1 +\alpha_2 +\alpha_3$. $\alpha'\leq  \alpha$ means that no component of $\alpha'$ is greater than the corresponding component of $\alpha$, and $\alpha'<\alpha$ means that $\alpha'\leq  \alpha$ and $|\alpha'|<|\alpha|$. And it is convenient to write $\nabla_x^k=\partial^{\alpha}$ with $|\alpha|=k$.
\item $\langle\cdot,\cdot\rangle$ denotes the ${L^2_{ v}}$ inner product in ${\mathbb{ R}}^3_{ v}$, with the ${L^2}$ norm $|\cdot|_{L^2}$. For notational simplicity, $(\cdot, \cdot)$ denotes the ${L^2}$ inner product either in ${\mathbb{ T}}^3_{x}\times{\mathbb{ R}}^3_{ v }$ or in ${\mathbb{ T}}^3_{x}$ with the $L^2$ norm $\|\cdot\|$.
    \item $B_C \subset \mathbb{R}^3$
denotes the ball of radius $C$ centered at the origin, and $L^2
(B_C)$ stands for the space $L^2$ over $B_C$ and likewise for other spaces.
\item
For $l\in \mathbb{R}$ , $\langle v\rangle=\sqrt{1+| v|^2}$, $L_l^2$  denotes the weighted space with norm
$
|f|_{L^2_{l}}^2\equiv\int_{\mathbb{R}^3_v}|f(v)|^2\langle v\rangle^{2l}\,\mathrm{d}v.
$
The weighted frictional Sobolev norm $|f(v)|_{H^s_l}^2=|\langle v\rangle^lf(v)|_{H^s}^2$ is given by
\begin{equation*}
|f(v)|_{H^s_l}^2=|f|^2_{L^2_l}+\int_{\mathbb{R}^3\times\mathbb{R}^3}
\frac{[\langle v\rangle^lf(v)-\langle v'\rangle^lf(v')]^2}{|v-v'|^{3+2s}}\chi_{\{(v,v')||v-v'|\leq1\}}\,\mathrm{d}v\mathrm{d}v'.
\end{equation*}

\item For functions in $\mathbb{T}^3_x\times\mathbb{R}^3_v$, short forms $\|\cdot\|_{H^s_{l}}=\||\cdot|_{H^s_l}\|_{L^2_x}$, $\|\cdot\|_{H^s}=\|\cdot\|_{H^s_0}$ when $l=0$, and $\|\cdot\|_{D}=\||\cdot|_{D}\|_{L^2_x}$ are used.

\item We use the notation $(\cdot)_{=0}$ to denote {the part corresponding to $0$ mode of a function  after spatial Fourier transformation, i.e. $(\cdot)_{=0}=\frac{1}{|\T^3|}\int_{\mathbb{T}^3}(\cdot)\,\mathrm{d}x$.} In addition, we write \[ f_{\neq 0} := f - \frac{1}{|\T^3|}\int_{\mathbb{T}^3} f \, \mathrm{d}x, \]
for the fluctuation part of
$f$ with zero mean.
\end{itemize}
\subsection{The weights and the corresponding functionals}
For given constants $q\geq0$ and $\iota_0$, we introduce the following weight function $w^{(q)}_{\iota_0,N}$:
\begin{equation}\label{def-weight-1}
w^{(q)}_{\iota_0,N}=w^{(q)}_{\iota_0,N}(\alpha, \beta,\omega)=\langle v\rangle^{\frac{-\gamma}{2s}(\iota_0-|\alpha|-|\beta|-
	|\omega|)}e^{{q}\langle v\rangle},\quad |\alpha|+|\beta|+
	|\omega|=N.
\end{equation}
Here and below we choose $q>0$  when $|\omega|=0$ and $q=0$ when $|\omega|>0$.

To capture the structure of the VPB system \eqref{f}, as in \cite{CLN-JAMS-2023}, we define the energy norm $\mathcal{E}^{(q,\iota_0)}_{\alpha,\beta,\omega}[f](t)$ as
\begin{equation}
	\begin{split}
	\mathcal{E}^{(q,\iota_0)}_{\alpha,\beta,\omega}[f](t)
	\equiv&
	A_0\sum_{|\alpha'|\leq 1}\|e^{\frac12\phi}\nu^{\kappa |\beta|} w^{(q)}_{\iota_0,N+|\alpha'|}\partial^{\alpha+\alpha'}_x\partial^\beta_v Y^\omega f\|^2\nonumber\\
	&+\nu^\kappa\int_{\mathbb{R}^3\times \mathbb{R}^3}e^{\phi}\nu^{2\kappa |\beta|} w^{(q)}_{\iota_0,N+1}\nabla_x\partial^{\alpha}_x\partial^\beta_v Y^\omega f\cdot w^{(q)}_{\iota_0,N+1}\nabla_v\partial^{\alpha}_x\partial^\beta_v Y^\omega f\,\mathrm{d}v\mathrm{d}x\\
	&
+\nu^{2\kappa}\sum_{|\beta'|= 1}\|e^{\frac12\phi}\nu^{\kappa |\beta|} w^{(q)}_{\iota_0,N+1}\partial^{\alpha}_x\partial^{\beta+\beta'}_v Y^\omega f\|^2,
	\end{split}
\end{equation}
and the corresponding dissipation functionals $\mathcal{D}^{(q,\iota_0)}_{\alpha,\beta,\omega}[f](t)$ as
\begin{equation}
	\begin{split}
		\mathcal{D}^{(q,\iota_0)}_{\alpha,\beta,\omega}[f](t)
		\equiv&
		A_0\nu^{1-\kappa}\sum_{|\alpha'|\leq 1}\|e^{\frac12\phi}\nu^{\kappa |\beta|} w^{(q)}_{\iota_0,N+|\alpha'|}\partial^{\alpha+\alpha'}_x\partial^\beta_v Y^\omega f\|_D^2\nonumber\\
		&+\|e^{\frac12\phi}\nu^{\kappa |\beta|} w^{(q)}_{\iota_0,N+1}\partial^{\alpha+e_i}_x\partial^\beta_v Y^\omega f\|^2\\
		&
		+\nu^{1+\kappa}\sum_{|\beta'|= 1}\|e^{\frac12\phi}\nu^{\kappa |\beta|} w^{(q)}_{\iota_0,N+1}\partial^{\alpha}_x\partial^{\beta+\beta'}_v Y^\omega f\|^2_D,
	\end{split}
\end{equation}
respectively. Here, $A_0>0$ and $\kappa=\frac{1}{1+2s}.$

In addition, to capture Landau damping and estimate the density function $\rho:=\int_{\mathbb{R}^3}\mu^{\frac12} f$, we introduce  more complete version of energy norms $\widetilde{\mathcal{E}}^{(q,\iota_0)}_{\alpha,\beta,\omega}[f](t)$ as
\begin{equation}\label{def-E}
\widetilde{\mathcal{E}}^{(q,\iota_0)}_{\alpha,\beta,\omega}[f](t)
=\mathcal{E}^{(q,\iota_0)}_{\alpha,\beta,\omega}[f](t)+A_0^{-1}\sum_{|\beta'|=2}
\nu^{4\kappa}\|e^{\frac12\phi}\nu^{\kappa |\beta|} w^{(q)}_{\iota_0,N+|\beta'|}\partial^{\alpha}_x\partial^{\beta+\beta'}_v Y^\omega f\|^2,
\end{equation}
and the corresponding dissipation  $\widetilde{\mathcal{D}}^{(q,\iota_0)}_{\alpha,\beta,\omega}[f](t)$ as
\begin{equation}\label{def-D}
	\widetilde{\mathcal{D}}^{(q,\iota_0)}_{\alpha,\beta,\omega}[f](t)
=\mathcal{D}^{(q,\iota_0)}_{\alpha,\beta,\omega}[f](t)+A_0^{-1}\sum_{|\beta'|=2}
\nu^{1+3\kappa}\|e^{\frac12\phi}\nu^{\kappa |\beta|} w^{(q)}_{\iota_0,N+|\beta'|}\partial^{\alpha}_x\partial^{\beta+\beta'}_v Y^\omega f\|^2_D.
\end{equation}
We also introduce the following combined energy norms, classified with respect to $\om$ and $q$, for the index $(\alpha, \beta)$
\begin{equation*}
	\begin{split}
		\mathbb{E}^{(q,\iota_0)}_{N_\alpha^{low},N_{\alpha,\beta},N_\beta,N_\omega}[f](t)\equiv\sum_{|\alpha|\geq N_\alpha^{low},|\beta|\leq N_\beta,\atop |\alpha|+|\beta|\leq N_{\alpha,\beta}}\mathcal{E}^{(q,\iota_0)}_{\alpha,\beta,0}[f](t)+\sum_{|\alpha|\geq N_\alpha^{low},|\beta|\leq N_\beta,\atop |\alpha|+|\beta|\leq N_{\alpha,\beta},1\leq |\omega|\leq N_\omega}\mathcal{E}^{(0,\iota_0)}_{\alpha,\beta,\omega}[f](t),
		\end{split}
\end{equation*}
and the case-wise combined dissipation functionals
\begin{equation*}
	\begin{split}
		\mathbb{D}^{(q,\iota_0)}_{N_\alpha^{low},N_{\alpha,\beta},N_\beta,N_\omega}[f](t)\equiv\sum_{|\alpha|\geq N_\alpha^{low},|\beta|\leq N_\beta,\atop |\alpha|+|\beta|\leq N_{\alpha,\beta}}\mathcal{D}^{(q,\iota_0)}_{\alpha,\beta,0}[f](t)+\sum_{|\alpha|\geq N_\alpha^{low},|\beta|\leq N_\beta,\atop |\alpha|+|\beta|\leq N_{\alpha,\beta},1\leq |\omega|\leq N_\omega}\mathcal{D}^{(0,\iota_0)}_{\alpha,\beta,\omega}[f](t),
	\end{split}
\end{equation*}
where $0\leq N_\alpha^{low}, ~N_\beta\leq N_{\al,\beta}$.

Similarly, the case-wise combined version of \eqref{def-E} and \eqref{def-D} is
\begin{equation*}
	\begin{split}
		\widetilde{\mathbb{E}}^{(q,\iota_0)}_{N_\alpha^{low},N_{\alpha,\beta},N_\beta,N_\omega}[f](t)\equiv\sum_{|\alpha|\geq N_\alpha^{low},|\beta|\leq N_\beta,\atop |\alpha|+|\beta|\leq N_{\alpha,\beta}}	\widetilde{\mathcal{E}}^{(q,\iota_0)}_{\alpha,\beta,0}[f](t)+\sum_{|\alpha|\geq N_\alpha^{low},|\beta|\leq N_\beta,\atop |\alpha|+|\beta|\leq N_{\alpha,\beta},1\leq |\omega|\leq N_\omega}	\widetilde{\mathcal{E}}^{(0,\iota_0)}_{\alpha,\beta,\omega}[f](t),
	\end{split}
\end{equation*}
and
\begin{equation*}
	\begin{split}
			\widetilde{\mathbb{D}}^{(q,\iota_0)}_{N_\alpha^{low},N_{\alpha,\beta},N_\beta,N_\omega}[f](t)\equiv\sum_{|\alpha|\geq N_\alpha^{low},|\beta|\leq N_\beta,\atop |\alpha|+|\beta|\leq N_{\alpha,\beta}}	\widetilde{\mathcal{D}}^{(q,\iota_0)}_{\alpha,\beta,0}[f](t)+\sum_{|\alpha|\geq N_\alpha^{low},|\beta|\leq N_\beta,\atop |\alpha|+|\beta|\leq N_{\alpha,\beta},1\leq |\omega|\leq N_\omega}	\widetilde{\mathcal{D}}^{(0,\iota_0)}_{\alpha,\beta,\omega}[f](t).
	\end{split}
\end{equation*}
For convenience, we define the following total energy norms for all indexes satisfying $|\alpha|+|\beta|+|\omega|\leq N$:
\begin{equation}\label{Energy-N}
	\widetilde{\mathbb{E}}^{(q,\iota_0)}_N[f](t)\equiv\sum_{N_{\alpha,\beta}+N_\omega\leq N}\widetilde{\mathbb{E}}^{(q,\iota_0)}_{0,N_{\alpha,\beta},N_\beta,N_\omega}[f](t),
\end{equation}
and the total dissipation functionals
\begin{equation}\label{Dissi-N}
	\widetilde{\mathbb{D}}^{(q,\iota_0)}_N[f](t)\equiv\sum_{N_{\alpha,\beta}+N_\omega\leq N}\widetilde{\mathbb{D}}^{(q,\iota_0)}_{0,N_{\alpha,\beta},N_\beta,N_\omega}[f](t).
\end{equation}
\subsection{The main result}
With the above preparation in hand, our main result concerning the global solvability of the Cauchy problem \eqref{f} and \eqref{f-initial} can be stated as follows.
\begin{theorem}\label{Main-Th.}
Let $0<q\ll 1$, \( \iota_0 \geq N_{max}+9 \), and $N_{max}\geq 9$. Denote
		\[\kappa=\frac{1}{1+2s}\]
under the condition \eqref{case}. Assume the initial datum \( f_0 \) satisfies the following conservation laws:
\begin{align}\label{cons-laws}
  &\int_{\mathbb{T}^3\times\mathbb{R}^3}\mu^{\frac12} f_0\,\mathrm{d}v\mathrm{d}x=0,\quad \int_{\mathbb{T}^3\times\mathbb{R}^3}v_j\mu^{\frac12} f_0\,\mathrm{d}v\mathrm{d}x=0,\quad 1\leq j\leq 3,\nonumber\\
  &\int_{\mathbb{T}^3\times\mathbb{R}^3}|v|^2\mu^{\frac12} f_0\,\mathrm{d}v\mathrm{d}x+\int_{\mathbb{T}^3}|\nabla_x\phi(0,x)|^2\,\mathrm{d}v\mathrm{d}x=0.
\end{align}
	There exist positive constants $\mathcal{M}\in (0,\mathcal{M}_0]$, and $\nu\in(0,\nu_0]$ such that if:
	\begin{equation}
	 \sum_{|\alpha|+|\beta| \leq N_{\text{max}}+2} \left\| e^{q(v)} \langle v \rangle^{-\frac{\iota_0\gamma}{2s}} \partial_x^\alpha \partial_v^\beta f_0 \right\| \leq \mathcal{M} \nu^{\kappa},\notag
	\end{equation}
	then the Cauchy problem \eqref{f} and \eqref{f-initial} admits a unique global smooth solution $ f(t,x,v) $. Moreover, The solution $ f(t,x,v) $ satisfies:
	
	\begin{itemize}
		\item[\textit{(i)}] \textbf{Bounded weighted energy:}
		\begin{equation}\label{main1}
		\sum_{|\alpha|+|\beta|\leq N_{max}-1} \nu^{\kappa|\beta|} \left\| e^{q(v)} \langle v \rangle^{\frac{-\gamma}{2s}(\iota_0-|\alpha|-|\beta|)} \partial_x^\alpha \partial_v^\beta f(t)\right\| \leq \mathcal{M}\nu^{\kappa},
			\end{equation}
		and
			\begin{equation}\label{main2}
			\sum_{|\alpha|+|\beta|+|\omega|\leq N_{max}-2} \nu^{\kappa|\beta|} \left\|  \langle v \rangle^{\frac{-\gamma}{2s}(\iota_0-|\alpha|-|\beta|-|\omega|)} \partial_x^\alpha \partial_v^\beta Y^\omega f(t)\right\| \leq \mathcal{M}\nu^{\kappa}.
		\end{equation}
			
			\item[\textit{(ii)}] \textbf{Energy decay:}
			\begin{equation}\label{main3}
			\sum_{|\alpha|+|\beta|+|\omega| \leq N_{\text{max}}-2} \nu^{\kappa|\beta|} \left\| \langle v \rangle^{\frac{-\gamma}{2s}(\iota_0-|\alpha|-|\beta|)} \partial_x^\alpha \partial_v^\beta Y^\omega f(t) \right\| \leq C\mathcal{M}\nu^\kappa e^{-\delta {(\nu t)}^{\frac{1}{1-\gamma-2s}}}.
				\end{equation}
				
				\item[\textit{(iii)}] \textbf{Enhanced dissipation:} For \( f_{\neq 0} := f - \frac{1}{|\T^3|}\int_{\mathbb{T}^3} f \, \mathrm{d}x \),
				\begin{equation}\label{main10}
				\sum_{|\alpha|+|\beta|+|\omega| \leq N_{\text{max}}-2} \nu^{\kappa|\beta|} \left\| \langle v \rangle^{\frac{-\gamma}{2s}(\iota_0-|\alpha|-|\beta|-|\omega|)}\partial_x^\alpha \partial_v^\beta Y^\omega f_{\neq 0} \right\| \leq C\mathcal{M}\nu^{\kappa} \min\left\{e^{-\delta_N(\nu^\kappa t)^{\frac{s}{s-\gamma}}},e^{-\delta_N(\nu t)^{\frac{1}{1-\gamma-2s}}}\right\}.
					\end{equation}
					
					\item[\textit{(iv)}] \textbf{Uniform polynomial decay:} For \( \rho_{\neq 0}(t,x) = \sum_{k \neq0} \rho_k(t) e^{ik \cdot x}\),
						\begin{equation}\label{main11}
					|\rho_k(t)| \leq C\mathcal{M}\nu^{\kappa} (1 + |k| + |kt|)^{-N_{\text{max}}+1} \min\left\{e^{-\delta_N(\nu^\kappa t)^{\frac{s}{s-\gamma}}},e^{-\delta_N(\nu t)^{\frac{1}{1-\gamma-2s}}}\right\}.
					\end{equation}
				\end{itemize}
			\end{theorem}
We make some remarks on Theorem \ref{Main-Th.}.
\begin{remark} 
Our results remain valid for collision kernels of the hard-potential type.
\end{remark}
\begin{remark} 
  For the non-cutoff collision operator, all existing estimates \cite{DLYZ-KRM2013, FLLZ-SCM-2018} require a first-order exponential weight in velocity, $e^{q\langle v\rangle}$, and it remains an open question whether the large-time behavior and Landau damping for the VPB system in the weak-collision regime can be established under this weight. Therefore, we choose the exponential velocity weight $e^{q\langle v\rangle}$ in the whole article.
\end{remark}
\begin{remark}  The sub-exponential decay in time that we establish is likely optimal. Indeed, if one employs a second-order exponential weight $e^{q\langle v\rangle^2}$ and lets $s\to 1$, then applying our  to the Vlasov-Poisson-Landau system yields a sub-exponential rate of the form $e^{-\delta_N(\nu^{1/3} t)^{2/5}}$. This improves upon the decay rate $e^{-\delta_N(\nu^{1/3} t)^{1/3}}$ obtained in \cite{CLN-JAMS-2023}. We attribute this improvement to the choice of a different algebraic weight function, namely
    \[
    \langle v\rangle^{-\frac{\gamma}{2s}(\iota_0-|\alpha|-|\beta|-|\omega|)}.
    \]
 This help us derive the  dissipation estimate \eqref{equation-decay}, which immediately is the basis of the stretched exponential decay estimate \eqref{Pro-linear-decay-2}, see the proof of Proposition \ref{linear-decay} for more details.  
\end{remark}
\begin{remark}
For the index $\kappa = \tfrac{1}{1+2s}$, we note that $\kappa = \tfrac{1}{3}$ when $s=1$, which corresponds to the Landau collision operator, and $\kappa = 1$ when $s=0$, formally corresponding to the cutoff Boltzmann collision operator. In fact, it is known from \cite{DL-M3AS-1992,Des-TTSP-1992} that the Landau operator can be obtained from the Boltzmann operator through the so-called grazing collision limit. Therefore, under the weak collision framework, our result provides a unified connection: on one hand, it relates to the recent work \cite{CLN-JAMS-2023}, which captures Landau damping and enhanced dissipation effects for the Vlasov–Poisson–Landau system; on the other hand, it also links to earlier works such as \cite{Guo-CPAM-2002,DYZ-JDE-2012,DYZ-MMMAS-2013,XXZ-JDE-2013,XXZ-JFA-2017,YZ-CMP-2006,YYZ-ARMA-2006}, which established the existence theory for the cutoff VPB system but without exhibiting these dissipative mechanisms.

\end{remark}


\subsection{Literatures}

The study of kinetic equations describing dilute charged particles has attracted sustained interest. The spatially inhomogeneous Boltzmann equation without angular cutoff serves as a fundamental model, capturing the regularization effect of grazing collisions and the nonlinear coupling between transport and dissipation. The global existence of classical solutions near Maxwellians for the non-cutoff Boltzmann equation was independently established by Gressman-Strain \cite{Gressman_Strain-JAMS-2011} and by Alexandre-Morimoto-Ukai-Xu-Yang \cite{AMUXY-JFA-2012}, who developed anisotropic Littlewood–Paley analysis and obtained sharp coercivity estimates revealing the fractional derivative structure of the collision operator. These works settled the global well-posedness problem in the perturbative setting and provided a framework for further analysis of regularity and decay.


When the Boltzmann equation is coupled with the Poisson equation, one obtains the VPB system, which accounts for both collisional and collective plasma effects. In the strongly collisional regime $\nu=1$,
a relatively complete mathematical theory is available. Lions \cite{Lions-JMKU-1994} constructed global renormalized solutions for the spatially inhomogeneous VPB system, and Mischler \cite{M-CMP-2000} extended the analysis to the case with boundaries. In the perturbative setting near a global Maxwellian, Guo \cite{Guo-CPAM-2002} introduced the macro–micro decomposition method and established global existence for the hard-sphere model in a periodic domain. Yang-Zhao \cite{YZ-CMP-2006} and Yang-Yu-Zhao \cite{YYZ-ARMA-2006} extended the result to the whole space, while Duan-Yang-Zhao \cite{DYZ-JDE-2012,DYZ-MMMAS-2013} derived global classical solutions for hard and moderately soft potentials by introducing velocity weights. Employing refined time-decay estimates, Xiao-Xiong-Zhao \cite{XXZ-JDE-2013,XXZ-SCM-2014,XXZ-JFA-2017} further treated the non-cutoff hard potentials and the whole range of cutoff soft potentials with relaxed initial conditions. More recent results include the work of Cao-Kim-Lee \cite{CKL-ARMA-2019} on global solutions in convex domains with diffuse reflection and of Li-Wang \cite{LW-JDE-2021} on solutions with in-flow boundary conditions.

The large-time behavior of solutions has also been investigated. Duan-Strain \cite{DS-ARMA-2011} obtained optimal decay rates by combining linear decay with the Duhamel principle, and Yang-Yu \cite{YY-CMP-2011} constructed compensating functions to handle the multi-species case. Li-Zhong-Yang \cite{LYZ-IUMJ-2016,LYZ-ARMA-2021} used spectral analysis to derive optimal decay rates and constructed corresponding global classical solutions. In the non-cutoff setting, Duan-Liu \cite{DL-CMP-2013} proved global classical solutions near Maxwellians for the VPB system without angular cutoff and established decay estimates, extending the non-cutoff Boltzmann theory to the VPB framework. Their work has been further developed in subsequent studies: Duan-Liu-Yang-Zhao \cite{DLYZ-KRM2013} analyzed the nonrelativistic Vlasov-Maxwell-Boltzmann system with non-cutoff potentials, and Fan-Lei-Liu-Zhao \cite{FLLZ-SCM-2018} considered the case with weak angular singularity. These results show that the non-cutoff approach can be systematically applied to collisional kinetic–field models.

In comparison, the VPB system in the weakly collisional regime remains much less understood. Formally, as
$\nu\rightarrow0$, the linearized collision operator loses its dissipativity and the VPB system approaches the collisionless Vlasov-Poisson equations, where nonlinear Landau damping describes the main relaxation mechanism. It is therefore natural to ask whether global classical solutions still exist for small but positive collision strength and how the solution behavior is related to Landau damping. While a complete theory is not yet available, several related results are informative. 
For the Vlasov-Fokker-Planck equation, Bedrossian \cite{B-AP-2017} established nonlinear Landau damping, while Bedrossian-Zhao-Zi \cite{BZZ-CMP-2025} investigated the stability threshold problem for initial data with Gevrey regularity. Recently, 
Chaturvedi-Luk-Nguyen \cite{CLN-JAMS-2023} established global existence and decay estimates for the Vlasov-Poisson-Landau (VPL) system in the weakly collisional periodic setting, Bedrossian-Zhao-Zi \cite{BZZ-arxiv-2025} studied the collisionless limit of the VPL system. 
In parallel, Bedrossian-Zelati-Dolce \cite{BCD-PLMS-2024} analyzed the long-time behavior of the non-cutoff Boltzmann equation with soft potentials, clarifying how weak collisions interact with transport to accelerate relaxation toward equilibrium. Taken together, these works demonstrate a common feature of the weakly collisional regime: transport effects such as phase mixing and kinetic Taylor dispersion become dominant over collisions in determining the long-time dynamics.
Their analysis combines phase mixing with enhanced dissipation and provides a general approach to control nonlinear effects when collisional dissipation is small. These developments indicate that transport effects together with the residual collision operator may still suffice to yield global solutions and asymptotic stability in the weakly collisional regime.

\subsection{Difficulties and strategies}
We now outline the principal ideas of our proof. The essential difficulty in establishing the global solution to the VPB system in the weak-collision regime, amplifying the dissipative effect, and analyzing the large-time behavior lies in the  {\bf intrinsic weakness of the diffusion furnished by the non-cutoff Boltzmann operator}. This weakness manifests concretely through the following difficulties:
\begin{itemize}
\item[1)] {\it Transport term at top order.}
In the VPB system, the density constraint requires that the highest-order derivatives in the energy space be taken with respect to the velocity variable. As a result, the transport term $v\cdot\nabla_x f$ does not vanish in top-order estimates, in contrast to the force-free Boltzmann equation studied in \cite{BCD-PLMS-2024}. Controlling this contribution is therefore more delicate.

    \item [2)] {\it Nonliear density estimate and weighted energy estimates under weak dissipation.}
Near global equilibrium, the linearized Boltzmann operator provides only weak $H_v^s$ dissipation $(s<1)$ compared with the full $H^1_v$ dissipation of the Landau operator. This weaker coercivity creates significant difficulties in controlling both the nonlinear density component and the weighted energy estimates, especially for terms involving higher-order velocity derivatives and the electric-field interaction.

\item[3)]{\it Insufficiency of fractional diffusion for energy closure.}
The dissipation provided by the non-cutoff Boltzmann operator corresponds to fractional diffusion, which is too weak to close the weighted energy estimates in the VPB setting. For example, the analog of the linearized estimate in \cite{CLN-JAMS-2023} yields
\begin{equation*}
    A_0\nu\|\mu(Y^{\omega'}f)_{=0} \|^2_{L^2_{x,v}}\lesssim A_0\nu\|\partial_{v_j}\mu(Y^{\omega''}f)_{=0} \|^2_{L^2_{x,v}}\lesssim A_0\nu^{\frac13}\|\mu(Y^{\omega''}f)_{=0} \|^2_{\mathcal{D}^{(\vartheta)}_{\alpha,0,\omega''}},
\end{equation*}
where $Y^{\omega'}=Y_jY^{\omega''}$. For the VPB system, however, the analogous estimate takes the form
\begin{equation}\label{diff-2-1}
    A_0\nu\|\mu(Y^{\omega'}f)_{=0} \|^2_{L^2_{x,v}}\lesssim A_0\nu\|\partial_{v_j}\mu(Y^{\omega''}f)_{=0} \|^2_{L^2_{x,v}}\lesssim A_0\nu^{1-2\kappa}\|\mu(Y^{\omega''}f)_{=0} \|^2_{\mathcal{D}^{(\vartheta)}_{\alpha,0,\omega''}}.
\end{equation}
This in turn gives
\begin{equation}\label{diff-2-2}
    A_0\nu^{1-2\kappa}\|\mu(Y^{\omega''}f)_{=0} \|^2_{\mathcal{D}^{(\vartheta)}_{\alpha,0,\omega''}}\lesssim A_0\nu^{\kappa}\|\mu(Y^{\omega''}f)_{=0} \|^2_{\mathcal{D}^{(\vartheta)}_{\alpha,0,\omega''}}.
\end{equation}
This forces $\kappa \leq \tfrac{1}{3}$, contradicting the natural choice $\kappa = \tfrac{1}{1+2s} > \tfrac{1}{3}$ for $0<s<1$. The contradiction highlights the intrinsic weakness of fractional diffusion in $\nu$, making it impossible to close the estimates without a refined approach.

\end{itemize}
Together, these three obstacles reflect the fundamental difficulty caused by the weak diffusion of the non-cutoff Boltzmann operator.

To overcome the first difficulty arising from the persistence of the transport term in top-order estimates, we introduce a velocity-weighted energy functional
\[w^{(q)}_{\iota_0,N}=w^{(q)}_{\iota_0,N}(\alpha, \beta,\omega)=\langle v\rangle^{\frac{-\gamma}{2s}(\iota_0-|\alpha|-|\beta|-
	|\omega|)}e^{{q}\langle v\rangle},\quad |\alpha|+|\beta|+
	|\omega|=N,\]
where the algebraic weight is carefully tuned to the dissipative structure of the non-cutoff Boltzmann operator. This design allows us to control the transport term via the crucial inequality
 \begin{align*}
  &\left\|\nu^{\kappa|\beta|}w_{\iota_0-N}\partial^{\beta'}_{v}
  \partial^\alpha_x\partial^{\beta-\beta'}_v Y^\omega f\right\|^2\nonumber\\
  \lesssim& A_0^{-s}\Big(
 A_0
  \nu^{1-\kappa}\left\|\nu^{\kappa(|\beta|-1)}w_{\iota_0-N+1} \partial^\alpha_x\partial^{\beta-\beta'}_v Y^\omega  f\right\|_D^2+\nu^{1+\kappa}\left\| \nu^{\kappa(|\beta|-1)}w_{\iota_0-N}\nabla_v \partial^\alpha_x\partial^{\beta-\beta'}_v Y^\omega f\right\|^2_D\Big)
\end{align*}
valid for $|\beta'|=1$ with the choice $\kappa=\frac{1}{1+2s}$, $s\in(0,1)$. This $\kappa$ coincides with the one used in the Landau setting of \cite{CLN-JAMS-2023} as $s\to 1$, but here its role is even more delicate since the top-order derivative is taken in $v$, not $x$, and hence the transport term does not drop out as in the force-free Boltzmann case \cite{BCD-PLMS-2024}.

The second difficulty arises from the weaker coercivity of the linearized Boltzmann operator compared with the Landau operator. While the Landau operator controls the full $H^1_v$-norm, the Boltzmann operator only generates an $H^s_v$-norm with $s<1$. This loss of dissipation creates additional difficulties in both the nonlinear density estimates and the weighted energy estimates.

\begin{itemize}
\item For the nonlinear density estimates, to compensate for the weaker dissipation, we derive a Volterra-type bound for the density component, adapted from \cite{CLN-JAMS-2023}, but now carefully tracking the $H^s_v$-regularity of $f$ up to order $\max\{1+2s,2\}$:
\[
\nu^{(|\beta|+1+2s)\kappa}\sum_{\substack{|\alpha|+|\beta|+|\omega|=N_{\max},\atop |\beta'|=1}}
\|\partial^\alpha_x\partial^{\beta+\beta'}v Y^\omega f\|_{\dot{H}^{2s}_v}.
\]
The dependence of this estimate on higher-order velocity derivatives leads to a delicate interplay between the $\nu$-weights and the nonlinear estimates. To address this, we introduce a more refined bookkeeping of the $\nu$–exponents in the nonlinear collision term than in the Landau case (see \eqref{dens-colli} for the detailed bounds), which ultimately enables us to close the uniform energy estimate.

\item For the weighted energy estimates, we need to control the term
\[
\sum_{|\alpha_1|+|\omega_1|\leq |\alpha|+|\omega|}
\Big(\partial^{\alpha_1+e_i}_x Y^{\omega_1}\phi
\partial^{\alpha-\alpha_1}_x\partial^{\beta+e_i}_vY^{\omega-\omega_1}f,
e^{\phi}\nu^{2\kappa|\beta|}w^2_{\iota_0-N}\partial^\alpha_x\partial^\beta_v Y^\omega f\Big),
\]
with $|\alpha|+|\beta|+|\omega|=N\leq N_{\max}$. In \cite{CLN-JAMS-2023}, when $|\alpha_1|+|\omega_1|\geq 5$, this term could be dominated by
$$\sum_{|\alpha|+|\omega|\leq N}\|\partial^\alpha_xY^\omega \rho_{\neq0}\| \min\left\{\left(\widetilde{\mathbb{E}}^{(q,\iota_0)}_N[f](t)\right)^{1/2},
\left(\widetilde{\mathbb{D}}^{(q,\iota_0)}_N[f](t)\right)^{1/2}\right\}$$ 
via integration by parts, thereby absorbing the extra velocity derivative into the dissipation norm $\widetilde{\mathbb{D}}^{(q,\iota_0)}_N[f](t)$. However, due to the weaker $H^s_v$ coercivity of the Boltzmann operator, this approach fails in our setting, particularly for terms of the form $\partial^\alpha_x\partial^{\beta+\beta'}_vY^\omega$ with $|\beta'|=2$.

To overcome this difficulty, instead of following \cite{CLN-JAMS-2023}, we exploit the sharp time decay of $\nabla_v\phi$ from \eqref{coro-elec} together with the favorable decay of the norms $\widetilde{\mathbb{E}}^{(q,\iota_0)}_{\bar{N}}f$ for sufficiently small $\bar{N}$, using the {\it a priori} assumption \eqref{priori-assu} in the cases $|\alpha_1|+|\omega_1|\leq N_{\max}-2$ and $|\alpha_1|+|\omega_1|\geq N_{\max}-1$, respectively (see \eqref{new1} and the discussion above it for details).
\end{itemize}

To address the third difficulty, posed by \eqref{diff-2-1} and \eqref{diff-2-2}, we additionally perform energy estimates on the derivatives of the vector field
$Y$, specifically estimates of the form
   \begin{eqnarray}
\frac{\mathrm{d}}{\mathrm{d}t}\left\|Y^\omega f(t)\right\|^2+\delta\nu\left\|{\bf P}^{\bot} Y^\omega f\right\|^2_{D}\nonumber
\lesssim C_\eta\nu\sum_{|\omega'|\leq |\omega|-1}\|Y^{\omega'} f\|_{D}^2+\cdots.
\end{eqnarray}

To control the first term on the right-hand side of the above inequality, we employ an interpolation inequality for the $Y-$derivative,i.e.
\[C_\eta\nu\sum_{|\omega'|\leq |\omega|-1}\|Y^{\omega'} f\|_{D}^2\lesssim C_{\tilde{\eta}}\nu\|f\|_{D}^2+\tilde{\eta}\nu\sum_{|\omega'|=N}\|Y^{\omega'} f\|_{D}^2,\]
which necessitates controlling the following two macroscopic quantities i.e. $$C_{\tilde{\eta}}\nu\|{\bf P}f\|_D^2+\tilde{\eta}\nu\sum_{|\omega'|=N}\|{\bf P}Y^{\omega'} f\|_{D}^2.$$To this end, we decompose the expression into two parts: $(\cdot)_{\neq0}+(\cdot)_{=0}$. By adopting the same strategy as in the estimates \eqref{nozero-Y} and \eqref{zero-Y}, we obtain
\begin{align*}
&C_{\tilde{\eta}}\nu\|{\bf P}f\|^2+\tilde{\eta}\nu\sum_{|\omega|=N}\|{\bf P}Y^{\omega} f\|^2\\
&\hspace{1cm}\lesssim C_{\tilde{\eta}}\nu\|\nabla_x{\bf P}f\|^2+\tilde{\eta}\nu\sum_{|\omega|=N}\|\nabla_x{\bf P}Y^{\omega} f\|^2+C_{\tilde{\eta}}\nu\left(\|\bar{c}\|^2+\|{\bf P}^{\bot}f\|^2_D\right)\\
&\hspace{1cm}\lesssim A^{-1}_0\nu\left(C_{\tilde{\eta}}\widetilde{\mathbb{D}}^{(q,\iota_0)}_N[f](t)
+\tilde{\eta}\widetilde{\mathbb{D}}^{(q,\iota_0)}_0[f](t)\right)+C_{\tilde{\eta}}\nu\left(\mathcal{M}^2\nu^{1+2\kappa}\langle t\rangle^{-4}+\|{\bf P}^{\bot}f\|^2_D\right).
\end{align*}

By taking suitable linear combinations, the simultaneous appearance of both \eqref{diff-2-1} and \eqref{diff-2-2} can be avoided.
The detailed derivation is given in Lemma \ref{lemma-nonweight-Y}. It should be pointed out that our approach is equally applicable to the treatment of the VPL system. Specifically, this method can serve as a replacement for the treatment in $(5.35)$ of \cite{CLN-JAMS-2023}.

\subsection{Outline of the paper} The remainder of the paper is structured as follows.
\begin{itemize}
    \item In Section \ref{Linear Boltzmann equation}, we focus on time decay estimates of the density function of the solution to  the following linear Boltzmann equation without angular cutoff.
    \item In Section \ref{density-estimates}, we are devoted to the density estimates of the VPB system \eqref{f}.
    \item In Section \ref{Nonlinear energy estimates}, we prove the global existence of a smooth solution $f(t,x,v)$ to the Cauchy problem \eqref{f} and \eqref{f-initial} based on the {\it a priori} estimates via the continuation argument.
    \item In Section \ref{The stretched exponential decay}, we establish the stretched exponential decay of the solution $f(t,x,v)$ to the  Cauchy problem \eqref{f} and \eqref{f-initial} and its density function $\rho(t,x)$.
    \item Finally, in Section \ref{Appendix}, we provide the relevant estimates for the linear and nonlinear Boltzmann operators, together with Strain-Guo type lemmas tailored to our framework.
\end{itemize}
\section{Linearized Boltzmann equation}\label{Linear Boltzmann equation}
In this section, we focus on time decay estimates of the density function of the solution to  the linear Boltzmann equation, including the phase mixing decay with respect to $\langle kt\rangle$, which is independent of $\nu$, and enhanced dissipation decay with respect to $ \nu^{\kappa} t$. These estimates will serve as the basis for the subsequent time decay estimates of the density function of \eqref{f} in Section \ref{density-estimates}.

Now we consider the following linear Boltzmann equation
\begin{equation}\label{Linear-BE}
	\partial_t f+v\cdot \nabla_x f+\nu \mathcal{L} f=0, \quad (x,v)\in \mathbb{T}^3\times \mathbb{R}^3,
\end{equation}
with initial datum $f(0,x,v)=f_0(x,v)$. For the given initial datum $f_0(x,v)$, denote
\begin{equation*}
S(t)[f_0](x,v):=f(t,x,v),
\end{equation*}
where $S(t)$ is the semigroup associated to \eqref{Linear-BE}, and $f(t,x,v)$ is the unique solution to \eqref{Linear-BE}.

Applying the Fourier transform to \eqref{Linear-BE}, we can write
\begin{eqnarray*}
	S(t)[f_0](x,v)=\sum_{k\in \mathbb{Z}^3}e^{ik\cdot x}S_k(t)[\hat{f}_{0k}](v),
\end{eqnarray*}
where $S_k(t)[\hat{f}_{0k}](v)$ denotes the corresponding semigroup to the Fourier transform of \eqref{Linear-BE} with intial datum $\hat{f}_{0k}(v)$ for each $k\in \mathbb{Z}^3$.
Then $h(t)=S_k(t)[h_0]$ solves the following linear Boltzmann equation
\begin{equation}\label{linear-h}
	\partial_t h+i k\cdot vh+\nu \mathcal{L}h=0
	\end{equation}
with intial datum $h_0(k,v)=h(0,k,v)$.

For later use, we denote
\[Y_{k,\eta}=\nabla_v+i(\eta+kt),\ \eta\in \mathbb{R}^3,\,k\in\mathbb{Z}^3.\]
\begin{proposition}\label{linear-decay}
	For any given $\eta\in \mathbb{R}^3, k\in\mathbb{Z}^3$ with $|k|\neq0$, the solution to \eqref{linear-h} satisfies the following time decay estimates:
	\begin{itemize}
		\item \textbf{Polynomial decay:}\begin{eqnarray}\label{Pro-linear-decay-1}
			\left|\int_{\mathbb{R}^3}S_k(t)[h_0]\mu^{\frac{1}{2}}\,\mathrm{d}v\right|\leq C_N\langle kt+\eta\rangle^{-N}\langle \nu^\kappa t\rangle^{-\frac32}\left\{\mathbb{E}_{k,\eta,N}^{(0,5)}[h_0]\right\}^{\frac12};
		\end{eqnarray}
		\item \textbf{Stretched exponential decay:}\begin{eqnarray}\label{Pro-linear-decay-2}
		\left|\int_{\mathbb{R}^3}S_k(t)[h_0]\mu^{\frac{1}{2}}\,\mathrm{d}v\right|\leq C_0\min\left\{e^{-\delta_0(\nu^\kappa t)^{\frac{s}{s-\gamma}}},e^{-\delta_0(\nu t)^{\frac{1}{1-\gamma-2s}}}\right\}\left\{\mathbb{E}_{k,\eta,0}^{(\bar{q},1)}[h_0]\right\}^{\frac12};
	\end{eqnarray}
	\item \textbf{Mixed decay:}\begin{eqnarray}\label{Pro-linear-decay-3}
	&&\left|\int_{\mathbb{R}^3}S_k(t)[h_0]\mu^{\frac{1}{2}}\,\mathrm{d}v\right|\nonumber\\ 
    &\lesssim&\langle kt+\eta\rangle^{-N}\min\left\{e^{-\delta_N(\nu^\kappa t)^{\frac{s}{s-\gamma}}},e^{-\delta_N(\nu t)^{\frac{1}{1-\gamma-2s}}}\right\}\left\{\mathbb{E}_{k,\eta,N+1}^{(0,5)}[h_0]
    +\mathbb{E}_{k,\eta,0}^{(\bar{q},1)}[h_0]\right\}^{\frac12},
\end{eqnarray}
	\end{itemize}
where $\delta$ is a small positive constant, and  $\mathbb{E}^{(\bar{q},\bar{\ell})}_{k,\eta,N}[h_0]$ with $\bar{q}=\frac{q}{2}$ for $|\omega|=0$ and $\bar{q}=0$ for $|\omega|>0$ is defined as
\begin{align*}
\mathbb{E}^{(\bar{q},\bar{\ell})}_{k,\eta,N}[h_0]=&\left|\langle v\rangle^{-\frac{\gamma\bar{\ell}}{2s}}e^{\bar{q}\langle v\rangle}h_0\right|^2_{L^2}+\sum_{1\leq|\omega|\leq N} \left |\langle v\rangle^{-\frac{\gamma\bar{\ell}}{2s}}Y^{\omega}_{0,\eta}h_0\right|^2_{L^2}  \nonumber\\
&+\nu^{2\kappa}|k|^{-2}\sum_{|\beta'|=1}\left(\left|\langle v\rangle^{\frac{-\gamma}{2s}(\bar{\ell}-1)}e^{\bar{q}\langle v\rangle}\partial^{\beta'}h_0\right|^2_{L^2}+\sum_{1\leq|\omega|\leq N}\left|\langle v\rangle^{\frac{-\gamma}{2s}(\bar{\ell}-1)}\partial^{\beta'}Y^{\omega}_{0,\eta}h_0\right|^2_{L^2}\right).
\end{align*}
\end{proposition}
\begin{proof} For clarity, we divide our proof into the following four steps.

\noindent\underline{{\it Step 1. 
Basic energy estimates:}} In order to derive the time decay estimates \eqref{Pro-linear-decay-1}, \eqref{Pro-linear-decay-2}, and \eqref{Pro-linear-decay-3}, we first need to establish basic energy estimates. To this aim, for any proper velocity weight $w(v)$ and $j=1,2,3$, we apply $\partial_{v_j}$ to \eqref{linear-h} and multiply the resulting equation by $\nu^\kappa w(v)\langle v\rangle^{\frac{\gamma}{2s}}$ to have
\begin{align}\label{h-v}
    (\partial_t+ik\cdot v)[\nu^\kappa w(v)\langle v\rangle^{\frac{\gamma}{2s}}\partial_{v_j}h]+\nu^\kappa w(v)\langle v\rangle^{\frac{\gamma}{2s}} ik_jh+\nu^{1+\kappa} w(v)\langle v\rangle^{\frac{\gamma}{2s}}\partial_{v_j}\mathcal{L} h=0.
\end{align}
Multiplying \eqref{linear-h} by $w(v)\langle v\rangle^{\frac{\gamma}{2s}} k_j$ yields
\begin{align}\label{h-k}
    (\partial_t+ik\cdot v)[k_jw(v)\langle v\rangle^{\frac{\gamma}{2s}} h]+\nu w(v)\langle v\rangle^{\frac{\gamma}{2s}}\mathcal{L} [k_jh]=0.
\end{align}

Noting that  $w(v)\langle v\rangle^{\frac{\gamma}{s}}\overline{ik_jh}=-w(v)\langle v\rangle^{\frac{\gamma}{2s}} ik_j\overline{h}$, we take the inner product of  \eqref{h-v} with $w(v)\langle v\rangle^{\frac{\gamma}{2s}} ik_j\overline{h}$ in $L^2(\mathbb{R}^3_v)$ to obtain
\begin{align*}
   & -\mathfrak{R}\int_{\mathbb{R}^3}(\partial_t+ik\cdot v)[\nu^\kappa w(v)\langle v\rangle^{\frac{\gamma}{2s}}\partial_{v_j}h] [w(v)\langle v\rangle^{\frac{\gamma}{2s}}ik_j\overline{h}]\,\mathrm{d}v\\
   &+\nu^\kappa\int_{\mathbb{R}^3}|w(v)\langle v\rangle^{\frac{\gamma}{s}} k_j h|^2\,\mathrm{d}v-\nu^{1+\kappa} \int_{\mathbb{R}^3}w^2(v)\langle v\rangle^{\frac{\gamma}{s}}\partial_{v_j}\mathcal{L}hik_j\overline{h}\,\mathrm{d}v=0.
\end{align*}
Then we have
\begin{eqnarray}\label{R-h-1}
&&-\mathfrak{R}\partial_t\int_{\mathbb{R}^3}[\nu^\kappa w(v)\langle v\rangle^{\frac{\gamma}{2s}}\partial_{v_j}h] [w(v)\langle v\rangle^{\frac{\gamma}{2s}}ik_j\overline{h}]\,\mathrm{d}v
+\mathfrak{R}\int_{\mathbb{R}^3}[\nu^\kappa w(v)\langle v\rangle^{\frac{\gamma}{2s}}\partial_{v_j}h] \partial_t[w(v)\langle v\rangle^{\frac{\gamma}{2s}}ik_j\overline{h}]\,\mathrm{d}v\nonumber\\
&&-\mathfrak{R}\int_{\mathbb{R}^3}ik\cdot v[\nu^\kappa w(v)\langle v\rangle^{\frac{\gamma}{2s}}\partial_{v_j}h] [w(v)\langle v\rangle^{\frac{\gamma}{2s}}ik_j\overline{h}]\,\mathrm{d}v+\nu^{\kappa}\int_{\mathbb{R}^3}|w(v)\langle v\rangle^{\frac{\gamma}{s}} k_j h|^2\,\mathrm{d}v\nonumber\\
&&-\nu^{1+\kappa} \int_{\mathbb{R}^3}w^2(v)\langle v\rangle^{\frac{\gamma}{s}}\partial_{v_j}\mathcal{L}hik_j\overline{h}\,\mathrm{d}v=0.
\end{eqnarray}
We take the conjugate of  \eqref{h-k} to have
\begin{eqnarray*}
    \partial_t( w(v)\langle v\rangle^{\frac{\gamma}{2s}} k_j \overline{h})=-\nu w(v)\langle v\rangle^{\frac{\gamma}{2s}} \mathcal{L}[k_j \overline{h}]+ik\cdot v [ w(v)\langle v\rangle^{\frac{\gamma}{2s}} k_j\overline{h}],
\end{eqnarray*}
which together with \eqref{R-h-1} gives
\begin{eqnarray}\label{h-linear-end-1}
&&-\mathfrak{R}\partial_t\int_{\mathbb{R}^3}[\nu^\kappa w(v)\langle v\rangle^{\frac{\gamma}{2s}}\nabla_{v}h]\cdot [w(v)\langle v\rangle^{\frac{\gamma}{2s}}ik\overline{h}]\,\mathrm{d}v
+\nu^\kappa\int_{\mathbb{R}^3}|w(v)\langle v\rangle^{\frac{\gamma}{2s}} k h|^2\,\mathrm{d}v\nonumber\\
&&-\nu^{1+\kappa}\mathfrak{R}\int_{\mathbb{R}^3}[w(v)\langle v\rangle^{\frac{\gamma}{2s}}\nabla_{v}h]\cdot[w(v)\langle v\rangle^{\frac{\gamma}{2s}}\mathcal{L}[i k \overline{h}]]\,\mathrm{d}v\nonumber\\
&&-\nu^{1+\kappa}\mathfrak{R}\int_{\mathbb{R}^3}[w(v)\langle v\rangle^{\frac{\gamma}{2s}}i kh]\cdot[w(v)\langle v\rangle^{\frac{\gamma}{2s}}\nabla_{v}\mathcal{L}[ \overline{h}]]\,\mathrm{d}v=0.
\end{eqnarray}
We take the inner product of \eqref{linear-h} with $\overline{h}$ in $L^2(\mathbb{R}^3_v)$ and use \eqref{Lemma L_1} to have
\begin{eqnarray}\label{h-linear-noweight-0}
\frac{\mathrm{d}}{\mathrm{d}t}\mathfrak{R}\int_{\mathbb{R}^3}|h|^2\,\mathrm{d}v+   \nu \delta|{\bf P}^{\bot}h|_D^2\leq 0,
\end{eqnarray}
where $\delta>0$ is a constant.
Multiplying \eqref{linear-h} by $w^2(v)\overline{h}$, integrating the result equation over $\mathbb{R}^3$ and taking the real part, we use \eqref{Lemma L_2} to have
\begin{eqnarray}\label{h-linear-weight-0}
&&\frac{\mathrm{d}}{\mathrm{d}t}\mathfrak{R}\int_{\mathbb{R}^3}|w(v)h|^2\,\mathrm{d}v+   \delta\nu\left|w(v)h\right|_{D}^2\leq C\nu\left|\chi_{\{| v|\leq C\}}h\right|^2_{L^2}\lesssim \nu \left|{\bf P}^{\bot}h\right|_D^2+\nu |k|^2 \left|\bf Ph\right|_{L^2}^2,
\end{eqnarray}
where we used $|k|\geq1$. 

As the derivation of \eqref{h-linear-weight-0}, we can also obtain
\begin{eqnarray}\label{h-linear-weight-1}
\frac{\mathrm{d}}{\mathrm{d}t}\mathfrak{R}\int_{\mathbb{R}^3}w^2(v)\langle v\rangle^{\frac{\gamma}{s}}|k|^2|h|^2\,\mathrm{d}v+   \nu \delta |k|^2\left|w(v)h\right|_{D}^2\lesssim\nu |k|^2 \left|{\bf P}^{\bot}h\right|_D^2+\nu |k|^2 \left|{\bf P}h\right|_{L^2}^2 .
\end{eqnarray}
Applying $\partial_{v_j}$ to \eqref{linear-h} and taking the inner product of the resulting equation with $\nu^{2\kappa}w^2(v)\langle v\rangle^{\frac{\gamma}{s}}\partial_{v_j}h$ in $L^2(\mathbb{R}^3_v)$, we further take its real part and  use \eqref{Lemma L_3} to get
\begin{eqnarray}\label{h-linear-noweight-v}
&&\frac{\mathrm{d}}{\mathrm{d}t}\mathfrak{R}\int_{\mathbb{R}^3}\nu^{2\kappa}w^2(v)\langle v\rangle^{\frac{\gamma}{s}}|\nabla_vh|^2\,\mathrm{d}v
+\mathfrak{R}\int_{\mathbb{R}^3}\nu^{2\kappa}w^2(v)\langle v\rangle^{\frac{\gamma}{s}}ik h\cdot \nabla_v\overline{h}\,\mathrm{d}v+\nu^{1+2\kappa}\left|w(v)\langle v\rangle^{\frac{\gamma}{2s}} \partial_{v_j} h\right|^2_D\nonumber\\
&&\qquad\lesssim  
\eta \nu^{1+2\kappa}\left|w(v)\langle v\rangle^{\frac{\gamma}{2s}} \nabla_v h\right|^2_D+ \nu^{1+2\kappa} \left|{\bf P}^{\bot}h\right|_D^2+\nu^{1+2\kappa} |k|^2 \left|{\bf P}h\right|_{L^2}^2.
\end{eqnarray}
Note that
\begin{align*}
\mathfrak{R}\int_{\mathbb{R}^3}&\nu^{2\kappa}w^2(v)\langle v\rangle^{\frac{\gamma}{s}}ikh\cdot\nabla_v\overline{h}\,\mathrm{d}v\nonumber\\
\lesssim&\nu^\kappa\int_{\mathbb{R}^3}w^2(v)\langle v\rangle^{\frac{\gamma}{s}}|k h|^2\,\mathrm{d}v+\eta\nu^{3\kappa}\int_{\mathbb{R}^3}w^2(v)\langle v\rangle^{\frac{\gamma}{s}}|\nabla_v h|^2\,\mathrm{d}v\nonumber\\
  \lesssim&\nu^\kappa\int_{\mathbb{R}^3} w^2(v)\langle v\rangle^{\frac{\gamma}{s}}|k h|^2\,\mathrm{d}v+\eta\nu \left|w(v)h\right|^2_{D}+\eta\nu^{1+2\kappa} \left|\langle v\rangle^{\frac{\gamma}{2s}}w(v)\nabla_vh\right|^2_{D},
\end{align*}
where we have used the weighted Gagliardo–Nirenberg interpolation inequality \eqref{ws-ine} in Lemma \ref{ws-ine-lem}, and the relation $3\kappa \geq s+(1+2\kappa)(1-s)$ holds since 
$$\frac{\gamma}{s}=s\gamma +(1-s)\frac{1+s}{s}\gamma,\ \textrm{and}\ \kappa=\frac1{1+2s}.$$

We make a proper linear combination of \eqref{h-linear-end-1}, \eqref{h-linear-noweight-0}, 
\eqref{h-linear-weight-0}, \eqref{h-linear-weight-1} and \eqref{h-linear-noweight-v} to obtain
\begin{eqnarray}\label{h-liear-end-weight}
   && \frac{\mathrm{d}}{\mathrm{d}t}\left\{\int_{\mathbb{R}^3}A_{0}\{w^2(v)|h|^2+w^2(v)\langle v\rangle^{\frac{\gamma}{s}}|kh|^2\}\,\mathrm{d}v
   +\int_{\mathbb{R}^3}C_{A_{0}}
   |h|^2\,\mathrm{d}v\right\}\nonumber\\
   && + \frac{\mathrm{d}}{\mathrm{d}t}\left\{ \int_{\mathbb{R}^3}\nu^{2\kappa}w^2(v)\langle v\rangle^{\frac{\gamma}{s}}|\nabla_{v}h|^2\,\mathrm{d}v\right\}
   -\frac{\mathrm{d}}{\mathrm{d}t}\mathfrak{R}\int_{\mathbb{R}^3}[\nu^\kappa w(v)\langle v\rangle^{\frac{\gamma}{2s}}\nabla_{v}h] \cdot[w(v)\langle v\rangle^{\frac{\gamma}{2s}}ik\overline{h}]\,\mathrm{d}v\\
   &&
   +\nu^\kappa\int_{\mathbb{R}^3}w^2(v)\langle v\rangle^{\frac{\gamma}{s}}|k h|^2\,\mathrm{d}v
   +\nu\left|w(v) h\right|_D^2+\nu\left|w(v)\langle v\rangle^{\frac{\gamma}{2s}}k h\right|_D^2+\nu^{1+2\kappa}\left|w(v)\langle v\rangle^{\frac{\gamma}{2s}} \nabla_v h\right|^2_D\lesssim 0,\nonumber
\end{eqnarray}
where $C_{A_{0}}$ is a constant dependent on $A_{0}$.
Noting the fact
\[
[\partial_t+ik\cdot v, Y_{k,\eta}]=0
\]
We apply $Y_{k,\eta}$ to the equation \eqref{linear-h} and use \eqref{Lemma L_1} to obtain
\begin{eqnarray*}
    \frac12\frac d{dt}\left|Y_{k,\eta}h\right|_{L^2}^2+\delta\nu \left|Y_{k,\eta}h\right|_D^2\lesssim \nu |k|^2\left|{\bf P}[Y_{k,\eta}h]\right|_{L^2}^2.
\end{eqnarray*}
As the derivation of \eqref{h-liear-end-weight}, we can also similarly get
\begin{eqnarray}
   && \frac{\mathrm{d}}{\mathrm{d}t}\left\{
   \int_{\mathbb{R}^3}A_{0}\{w^2(v)|Y_{k,\eta}^\omega h|^2+w ^2(v)\langle v\rangle^{\frac{\gamma}{s}}|k Y_{k,\eta}^\omega h|^2\}+C_{A_{0}}
   |Y_{k,\eta}^\omega h|^2+\nu^{2\kappa}w^2(v)\langle v\rangle^{\frac{\gamma}{s}}|\nabla_{v}Y_{k,\eta}^\omega h|^2\,\mathrm{d}v\right\}\nonumber\\
   &&\qquad-\frac{\mathrm{d}}{\mathrm{d}t}\mathfrak{R}\int_{\mathbb{R}^3}[\nu^\kappa w(v)\langle v\rangle^{\frac{\gamma}{2s}}\nabla_{v}Y_{k,\eta}^\omega h] \cdot[w(v)\langle v\rangle^{\frac{\gamma}{2s}}ik Y_{k,\eta}^\omega\overline{h}]\,\mathrm{d}v+\nu^\kappa\int_{\mathbb{R}^3}\{w^2(v)\langle v\rangle^{\frac{\gamma}{s}}|k Y_{k,\eta}^\omega h|^2\}\,\mathrm{d}v\nonumber\\
   &&
  \qquad\qquad +\nu|w(v) Y_{k,\eta}^\omega h|_D^2+\nu|w(v)\langle v\rangle^{\frac{\gamma}{2s}}k Y_{k,\eta}^\omega h|_D^2+\nu^{1+2\kappa}|w(v)\langle v\rangle^{\frac{\gamma}{2s}} \nabla_v Y_{k,\eta}^\omega h|^2_D\lesssim 0.\notag
\end{eqnarray}
Consequently, we conclude that
\begin{align}\label{h-linear-energy}
    \frac{\mathrm{d}}{\mathrm{d}t}&\sum_{|\omega|\leq N}\left\{\int_{\mathbb{R}^3}A_{0}\{w^2(v)|Y_{k,\eta}^\omega h|^2+w ^2(v)\langle v\rangle^{\frac{\gamma}{s}}|k Y_{k,\eta}^\omega h|^2\}\,\mathrm{d}v\right\}\nonumber\\
   &+ \frac{\mathrm{d}}{\mathrm{d}t}\sum_{|\omega|\leq N}\left\{
   \int_{\mathbb{R}^3}C_{A_{0}}
   |Y_{k,\eta}^\omega h|^2\,\mathrm{d}v+\int_{\mathbb{R}^3}\nu^{2\kappa}w^2(v)\langle v\rangle^{\frac{\gamma}{s}}|\nabla_{v}Y_{k,\eta}^\omega h|^2\,\mathrm{d}v\right\}\nonumber\\
   &-\frac{\mathrm{d}}{\mathrm{d}t}\sum_{|\omega|\leq N}\mathfrak{R}\int_{\mathbb{R}^3}[\nu^\kappa w(v)\langle v\rangle^{\frac{\gamma}{2s}}\nabla_{v}Y_{k,\eta}^\omega h]\cdot [w(v)\langle v\rangle^{\frac{\gamma}{2s}}ik Y_{k,\eta}^\omega\overline{h}]\,\mathrm{d}v\nonumber\\
   &+\nu^\kappa\sum_{|\omega|\leq N}\int_{\mathbb{R}^3}\{w^2(v)\langle v\rangle^{\frac{\gamma}{s}}|k Y_{k,\eta}^\omega h|^2\}\,\mathrm{d}v
   +\nu\sum_{|\omega|\leq N}|w(v) Y_{k,\eta}^\omega h|_D^2\nonumber\\
   &+\nu\sum_{|\omega|\leq N}|w(v)\langle v\rangle^{\frac{\gamma}{2s}}k Y_{k,\eta}^\omega h|_D^2+\nu^{1+2\kappa}\sum_{|\omega|\leq N}|w(v)\langle v\rangle^{\frac{\gamma}{2s}} \nabla_v Y_{k,\eta}^\omega h|^2_D\lesssim 0.
\end{align}
\noindent\underline{{\it Step 2. Polynomial decay:}} With the energy estimate \eqref{h-linear-energy} in hand, we come to prove the desired time decay estimates. For convenience, we define the following energy functional
\begin{align}\label{g-kh}
g^2(t,v)=&\sum_{|\omega|\leq N}\left\{A_{0}\left\{\langle v\rangle^{\frac{-\gamma}{s}}|Y_{k,\eta}^\omega h|^2+|k Y_{k,\eta}^\omega h|^2\right\}+C_{A_{0}}|Y_{k,\eta}^\omega h|^2\right\}\nonumber\\
&+\sum_{|\omega|\leq N}\nu^{2\kappa}|\nabla_{v}Y_{k,\eta}^\omega h|^2-\mathfrak{R}\sum_{|\omega|\leq N}[\nu^\kappa\nabla_{v}Y_{k,\eta}^\omega h] \cdot[ ik Y_{k,\eta}^\omega\overline{h}].
\end{align}
Taking $w(v)=\langle v\rangle^{\frac{-\gamma}{2s}}$ in \eqref{h-linear-energy}, we can deduce that {there is a small positive constant $\delta$ such that}
\begin{align}\label{equation-decay}
    \frac{\mathrm{d}}{\mathrm{d}t}\int_{\mathbb{R}^3}g^2\,\mathrm{d}v+\delta\nu^\kappa\int_{\mathbb{R}^3}\langle v\rangle^{\frac{\gamma}{s}} g^2\,\mathrm{d}v\lesssim 0.
\end{align}
Here we used the following facts:
\begin{align}\label{nabla-Diss}
&\left|\langle v\rangle^{\frac{\gamma}{2s}}(\langle v\rangle^{\frac{-\gamma}{2s}}Y_{k,\eta}^\omega h)\right|^2_{L^2}\lesssim |kY_{k,\eta}^\omega h|^2_{L^2},\quad |k|\geq1,\nonumber\\
  \nu^{2\kappa}\sum_{|\omega|\leq N}|\nabla_{v}Y_{k,\eta}^\omega h|^2_{L^2}\lesssim &A_0^{-s}
  \left(\nu^{1-\kappa}A_0\sum_{|\omega|\leq N}|\langle v\rangle^{\frac{-\gamma}{2s}+\frac{\gamma}{2}} Y_{k,\eta}^\omega h|_s^2\right)^s\left(\nu^{1+\kappa}\sum_{|\omega|\leq N}{|\langle v\rangle^{\frac{\gamma}{2}} }\nabla_v Y_{k,\eta}^\omega h|^2_s\right)^{1-s}\nonumber\\
   \lesssim &A_0^{-s}\left(
  \nu^{1-\kappa}A_0\sum_{|\omega|\leq N}|\langle v\rangle^{\frac{-\gamma}{2s}} Y_{k,\eta}^\omega h|_D^2+\nu^{1+\kappa}\sum_{|\omega|\leq N}| \nabla_v Y_{k,\eta}^\omega h|^2_D\right)
\end{align}
due to $2\kappa\geq (1-\kappa)s+(1+\kappa)(1-s)$ and thus
\begin{align*}
    \sum_{|\omega|\leq N}&\int_{\mathbb{R}^3}w^2(v)\langle v\rangle^{\frac{\gamma}{s}}|k Y_{k,\eta}^\omega h|^2\,\mathrm{d}v
   +\nu^{1-\kappa}\sum_{|\omega|\leq N}|w(v) Y_{k,\eta}^\omega h|_D^2\nonumber\\
   &+\nu^{1-\kappa}\sum_{|\omega|\leq N}|w(v)\langle v\rangle^{\frac{\gamma}{2s}}k Y_{k,\eta}^\omega h|_D^2+\nu^{1+\kappa}\sum_{|\omega|\leq N}|w(v)\langle v\rangle^{\frac{\gamma}{2s}} \nabla_v Y_{k,\eta}^\omega h|^2_D\nonumber\\
  \gtrsim& \int_{\mathbb{R}^3}\langle v\rangle^{\frac{\gamma}{s}} g^2\,\mathrm{d}v.
\end{align*}
On the other hand, taking $w(v)=\langle v\rangle^{-\frac{5\gamma}{s}}$ in  \eqref{h-linear-energy} gives 
\begin{align*}
    &\sup_{t\in [0,\infty)}\int_{\mathbb{R}^3}\langle v\rangle^{-\frac{4\gamma}{s}}g^2(t,v)\,\mathrm{d}v\nonumber\\
    \lesssim&
    \sup_{t\in [0,\infty)}\sum_{|\omega|\leq N}\int_{\mathbb{R}^3}\langle v\rangle^{-\frac{4\gamma}{s}}\left\{A_{0}\left\{\langle v\rangle^{\frac{-\gamma}{s}}|Y_{k,\eta}^\omega h|^2+|k Y_{k,\eta}^\omega h|^2\right\}+\nu^{2\kappa}|\nabla_{v}Y_{k,\eta}^\omega h|^2\right\}\,\mathrm{d}v\\
    \lesssim&\sum_{|\omega|\leq N}\int_{\mathbb{R}^3}\langle v\rangle^{-\frac{4\gamma}{s}}\left\{A_{0}\left\{\langle v\rangle^{\frac{-\gamma}{s}}|Y_{k,\eta}^\omega h(0)|^2+|k Y_{k,\eta}^\omega h(0)|^2\right\}+\nu^{2\kappa}|\nabla_{v}Y_{k,\eta}^\omega h(0)|^2\right\}\,\mathrm{d}v\nonumber\\
    \lesssim& \int_{\mathbb{R}^3}\langle v\rangle^{-\frac{4\gamma}{s}}{g}^2(0)\,\mathrm{d}v\lesssim |k|^2 \mathbb{E}_{k,\eta,N}^{(0,5)}[h_0].
\end{align*}
Then we apply Lemma \ref{decay-algeb} to have
\begin{eqnarray}\label{algeb-Yh}
   |k|^2 \sum_{|\omega|\leq N}|Y_{k,\eta}^\omega h|^2_{L^2} \lesssim \int_{\mathbb{R}^3}g^2(t,v)\,\mathrm{d}v\lesssim\langle \nu^\kappa t\rangle^{-3}\int_{\mathbb{R}^3}\langle v\rangle^{-\frac{4\gamma}{s}}{g}^2(0)\,\mathrm{d}v\lesssim |k|^2\langle \nu^\kappa t\rangle^{-\frac32}\mathbb{E}_{k,\eta,N}^{(0,5)}[h_0].
\end{eqnarray}
Now we let $h(t)=S_k(t)[h_0]$ and use $|\eta+kt|=|Y_{k,\eta}-\nabla_v|$ to have
\begin{align}\label{lin-pha}
&\left(1+|\eta+kt|^N\right)\left| \int_{\mathbb{R}^3}h\mu^{\frac{1}{2}}\,\mathrm{d}v\right|\lesssim \sum_{|\omega|\leq N} \left| \int_{\mathbb{R}^3}\left[\left(Y_{k,\eta}-\nabla_v\right)^{\omega}h\right]
\mu^{\frac{1}{2}}\,\mathrm{d}v\right|\nonumber\\
&\qquad\lesssim \sum_{|\omega_1+\omega_2|\leq N} \left| \int_{\mathbb{R}^3}\left(Y_{k,\eta}^{\omega_1}h\right)
\left(\nabla_v^{\omega_2}\mu^{\frac{1}{2}}\right)\,\mathrm{d}v\right|\lesssim 
\sum_{|\omega|\leq N} \|Y_{k,\eta}^{\omega}h\|_{L^2}.
\end{align}
Combining \eqref{algeb-Yh} and \eqref{lin-pha} yields \eqref{Pro-linear-decay-1}.

\noindent\underline{{\it Step 3. Stretched exponential decay:}}
Corresponding to \eqref{g-kh} with $N=0$, we define
\begin{align*}
\mathbb{g}^2(t,v)=&A_{0}\left\{\langle v\rangle^{-\frac{\gamma}{s}}| h|^2+|k  h|^2\right\}+C_{A_{0}}| h|^2+\nu^{2\kappa}|\nabla_{v}h|^2-\mathfrak{R}\left\{[\nu^\kappa\nabla_{v} h] \cdot[ ik \overline{h}]\right\}.
\end{align*}
Similarly, as the derivation of \eqref{equation-decay}, we can obtain
\begin{align*}
    \frac{\mathrm{d}}{\mathrm{d}t}\int_{\mathbb{R}^3}\mathbb{g}^2(t,v)\,\mathrm{d}v+\delta\nu^\kappa\int_{\mathbb{R}^3}\langle v\rangle^{\frac{\gamma}{s}} \mathbb{g}^2(t,v)\,\mathrm{d}v\lesssim 0,
\end{align*}
and
\begin{align*}
    \int_{\mathbb{R}^3}\mathbb{g}^2(t,v)e^{2\bar{q}\langle v\rangle}\,\mathrm{d}v\lesssim\int_{\mathbb{R}^3}\mathbb{g}^2(0,v)e^{2\bar{q}\langle v\rangle}\,\mathrm{d}v.
\end{align*}
Consequently, we apply Lemma \ref{decay-exponent} to have
\begin{eqnarray*} 
   |k|^2\int_{\mathbb{R}^3}h^2(t,v)\,\mathrm{d}v\lesssim \int_{\mathbb{R}^3}\mathbb{g}^2(t,v)\,\mathrm{d}v\lesssim e^{-\delta\langle\nu^\kappa t\rangle^{-\frac{s}{s-\gamma}}}  \int_{\mathbb{R}^3}\mathbb{g}^2(0,v)e^{\bar{q}\langle v\rangle}\,\mathrm{d}v\lesssim |k|^2e^{-\delta\langle\nu^\kappa t\rangle^{-\frac{s}{s-\gamma}}} \mathbb{E}_{k,\eta,0}^{(\bar{q},1)}[h_0].
\end{eqnarray*}
This together with \eqref{lin-pha} gives 
\begin{align}\label{expone-lin}
  \left|\int_{\mathbb{R}^3}S_k(t)[h_0]\mu^{\frac{1}{2}}\,\mathrm{d}v\right|\lesssim e^{-\delta_0(\nu^\kappa t)^{\frac{s}{s-\gamma}}}\left\{\mathbb{E}_{k,\eta,0}^{(\bar{q},1)}[h_0]\right\}^{\frac12}.
\end{align}
On the other hand,  noting the following fact
\begin{eqnarray*}
    &&|k Y_{k,\eta} h|^2_{L^2}
   +\nu^{1-\kappa}\left|\langle v\rangle^{\frac{-\gamma}{2s}}  Y_{k,\eta} h\right|_D^2+\nu^{1-\kappa}\left|k Y_{k,\eta} h\right|_D^2+\nu^{1+\kappa}\left| \nabla_v Y_{k,\eta} h\right|^2_D\nonumber\\
  &&\qquad \gtrsim \nu^{1-\kappa}\int_{\mathbb{R}^3}\langle v\rangle^{\gamma+2s} \mathbb{g}^2(t,v) \,\mathrm{d}v,
\end{eqnarray*}
we can modify the derivation of \eqref{equation-decay} to have
\begin{align*}
    \frac{\mathrm{d}}{\mathrm{d}t}\int_{\mathbb{R}^3}\mathbb{g}^2(t,v)\,\mathrm{d}v+\delta\nu\int_{\mathbb{R}^3}\langle v\rangle^{\gamma+2s} \mathbb{g}^2(t,v)\,\mathrm{d}v\lesssim 0.
\end{align*}
Then we apply Lemma \ref{decay-exponent} again to obtain
\begin{eqnarray*}
  |k|^2 | h|^2_{L^2} \lesssim  \int_{\mathbb{R}^3}\mathbb{g}^2(t,v)\,\mathrm{d}v\lesssim e^{-\delta\langle\nu t\rangle^{\frac{1}{1-\gamma-2s}}}  \int_{\mathbb{R}^3}\mathbb{g}^2(0,v)e^{\bar{q}\langle v\rangle}\,\mathrm{d}v\lesssim |k|^2e^{-\delta\langle\nu t\rangle^{\frac{1}{1-\gamma-2s}}} \mathbb{E}_{k,\eta,0}^{(\bar{q},1)}[h_0] ,
   \end{eqnarray*}
   which together with \eqref{lin-pha} for the case $N=0$ gives 
\begin{align}\label{expone-lin-0}
  \left|\int_{\mathbb{R}^3}S_k(t)[h_0]\mu^{\frac{1}{2}}\,\mathrm{d}v\right|\lesssim e^{-\delta_0(\nu t)^{\frac{1}{1-\gamma-2s}}}\left\{\mathbb{E}_{k,\eta,0}^{(\bar{q},1)}[h_0]\right\}^{\frac12}.
\end{align}
\eqref{Pro-linear-decay-2} follows from \eqref{expone-lin} and \eqref{expone-lin-0}.
   
\noindent\underline{{\it Step 4. Mixed time decay:}}
For $h(t)=S_k(t)h_0$, we can use \eqref{lin-pha} and \eqref{Pro-linear-decay-2} and employ an interpolation inequality to derive \eqref{Pro-linear-decay-3}:
%
\begin{eqnarray*}
\int_{\mathbb{R}^3}h(t)\mu^{\frac{1}{2}}\,\mathrm{d}v&\lesssim&
\langle kt+\eta\rangle^{-N}\sum_{|\omega|\leq N}|Y^\omega_{k,\eta}h|_{L^2}\lesssim \langle kt+\eta\rangle^{-N}\Big(\sum_{|\omega|\leq N+1}|Y^\omega_{k,\eta}h|_{L^2}\Big)^{\frac{N}{N+1}}\left(|h|_{L^2}\right)^{\frac{1}{N+1}}\nonumber\\
&\lesssim&\langle kt+\eta\rangle^{-N}\min\left\{e^{-\delta_N(\nu^\kappa t)^{\frac{s}{s-\gamma}}},e^{-\delta_N(\nu t)^{\frac{1}{1-\gamma-2s}}}\right\}\left[\mathbb{E}^{(0,5)}_{k,\eta,N+1}[h_0]
+\mathbb{E}^{(\bar{q},1)}_{k,\eta,0}[h_0]\right]^{\frac12}.
\end{eqnarray*}
Thus, the proof of Proposition \ref{linear-decay} is complete.
\end{proof}

\section{Density estimates}\label{density-estimates}
In this section, we are devoted to uniform bound estimates of the density and its derivatives of the VPB system \eqref{f}. As in \cite{CLN-JAMS-2023}, we will first derive Landau damping and enhanced dissipation for the linear version of the system \eqref{f}  based on estimates in Proposition \ref{linear-decay} for the linear Boltzmann equation, then further apply it to the full system.

\subsection{Linear density estimates}
Now we consider the following linear version of the VPB system
\begin{eqnarray}\label{linear-eqn-VPL}
	\partial_t f+v\cdot \nabla_x f+\nabla_x\phi\cdot v \mu^{\frac12}+\nu \mathcal{L}f=\mathcal{R}(t,x,v)
\end{eqnarray}
with the initial datum $f(0,x,v)=f_0(x,v)$, where $\mathcal{R}(t,x,v)$ is a source term and the electric potential $\phi(t,x)$ satisfies the Poisson equation $\phi=(-\Delta_x)^{-1}\big(\int_{\mathbb{R}^3}f(t,x,v)\mu^{\frac{1}{2}}-\frac{1}{(2\pi)^3}\displaystyle{\int_{\T^3\times\mathbb{R}^3}}f_0(x,v)\mu^{\frac{1}{2}}\,\mathrm{d}v\mathrm{d}x\big)$.

For the linear VPB system \eqref{linear-eqn-VPL}, we have the following result related to Landau damping.

\begin{proposition}
	For any initial datum $f_0(x,v)$ and any source term $\mathcal{ N}(t,x,v)$, the unique density solution
	$\rho(t,x)$ to \eqref{linear-eqn-VPL} can be written explicitly as
	\begin{eqnarray}\label{dens-exp}
	\widehat{\rho}(t)=\mathcal{N}_k(t)+\int_{0}^{t}G_k(t-s)\mathcal{N}_k(s)\,\mathrm{d}s
\end{eqnarray}
for each Fourier mode $k\in \mathbb{Z}^3/\{0\}$. Here the source term $\mathcal{ N}_k(t)$ is given by
\begin{eqnarray*}
	\mathcal{ N}_k(t)=\int_{{\mathbb{R}}^3}S_k(t)\hat{f}_{0k}(v)\mu^{\frac{1}{2}}\,\mathrm{d}v
+\int_{0}^{t}\int_{{\mathbb{R}}^3}S_k(t-\tau)
[\widehat{\mathcal{R}}_k(\tau,v)]\mu^{\frac{1}{2}}\,\mathrm{d}v\mathrm{d}\tau,
\end{eqnarray*}
where $S_k(t)$ is the semigroup operator of the linear Boltzmann equation \eqref{linear-h}
For any $N_0\geq 2, t\in[0,\infty)$, there exist constants $C_{N_0}>0$ and $\delta^{'}_{N_0}>0$ such that the kernel $G_k(t)$ in \eqref{dens-exp} satisfies
\begin{eqnarray}\label{kernel-decay}
	|G_k(t)|\leq C_{N_0}|k|^{-1}\langle kt\rangle^{-N_0+2}\min\left\{e^{-\delta'_{N_0}(\nu^\kappa t)^{\frac{s}{s-\gamma}}},e^{-\delta'_{N_0}(\nu t)^{\frac{1}{1-\gamma-2s}}}\right\}.
\end{eqnarray}
	\end{proposition}
\begin{proof}
	This proposition can be proved in the same manner as Proposition 7.1 in \cite{CLN-JAMS-2023}. The only difference is that here we consider the linear non-cutoff Boltzmann collision operator instead of the linear Landau collision operator. Nevertheless, this does not lead to any essential changes, thanks to the time decay estimates in Proposition \ref{linear-decay}, which correspond to Proposition 6.2 in \cite{CLN-JAMS-2023}.
\end{proof}
\subsection{Nonlinear density estimates}
Before deriving the nonlinear density estimates for the solution to the Cauchy problem \eqref{f} and \eqref{f-initial}, we assume the local-in-time existence of a unique solution $f(t,x,v)$ on the time interval $[0, T]$ and impose the following {\it a priori} assumption:
\begin{eqnarray}\label{priori-assu}
	X(T)\equiv&&\sup_{ 0\leq t\leq T}\Bigg\{\max\{\nu^{-\kappa},\langle t\rangle\}^{\min\{0,-N+N_{max}-2\}}\left\{\widetilde{\mathbb{E}}^{(q,\iota_0)}_N[f](t)+\nu^\kappa\int_{0}^{t}\widetilde{\mathbb{D}}^{(q,\iota_0)}_N[f](\tau)\mathrm{d}\tau\right\}\nonumber\\
	&&+\langle t\rangle^{4}\Big\{\|\phi(t)\|_{W^{5,\infty}_x}^2+\sum_{|\alpha|+|\omega|\leq 4}\|\partial^\alpha_xY^\omega \nabla_x\phi(t)\|^2_{L^\infty_x}\Big\}\Bigg\}\leq \mathcal{M}\nu^{2\kappa},
\end{eqnarray}
for any $T\in[0,\infty)$.
Under this {\it a priori} assumption, we are able to deduce the following uniform density estimates.
\begin{theorem}\label{dens-esti}
Under the same assumptions as in Theorem \ref{Main-Th.}, suppose the {\it a priori} assumption \eqref{priori-assu} holds for the unique local smooth solution $ f(t,x,v) $ to the Cauchy problem \eqref{f} and \eqref{f-initial} on the time interval $ [0, T]$. Then, for any $t\in [0, T]$, it holds that
	\begin{eqnarray}\label{Th-rho-1}
		\sum_{|\alpha|+|\omega|\leq N_{max}}\left\{\|\partial^\alpha_xY^\omega \rho_{\neq 0}(t) \|^2+\nu^\kappa\int_{0}^{t}\|\partial^\alpha_xY^\omega \rho_{\neq 0}(\tau) \|^2\mathrm{d}\tau\right\}\lesssim \mathcal{M}^2\nu^{2\kappa}.
	\end{eqnarray}
\end{theorem}
\begin{proof}
We prove this theorem by employing a strategy similar to that of Theorem 8.1 in \cite{CLN-JAMS-2023}. Compared with the Landau collision operator, the linear non-cutoff Boltzmann collision operator exhibits weaker dissipativity, and the estimates for its corresponding nonlinear operator are more involved. To ensure the proof is complete, we provide detailed calculations and estimates highlighting the main differences. To this end, we rewrite the nonlinear equation in \eqref{f} as follows:
\begin{eqnarray}
	\partial_t f+v\cdot \nabla_x f+\nabla_x\phi\cdot v \mu^{\frac12}+\nu \mathcal{L}f=\mathcal{R}(t,x,v),\notag
\end{eqnarray}
where
\[\mathcal{R}(t,x,v)\equiv -\frac12\nabla_x\phi \cdot vf+\nabla_x\phi\cdot \nabla_vf+\nu \Gamma(f,f).\]

For $|\alpha|+|\omega|\leq N_{\max}$, we take $N_0=2N_{\max}+4$ in \eqref{kernel-decay},
and apply the H\"{o}lder inequality to have
\begin{align}\label{rho-infty}
    &\sum_{k\neq 0}|k|^{2|\alpha|}\langle kt\rangle^{2|\omega|}|\widehat{\rho}_k(t)|^2\nonumber\\
       \lesssim&\sum_{k\neq 0}|k|^{2|\alpha|}\langle kt\rangle^{2|\omega|}\left(|\mathcal{N}_k(t)|^2+\int_{0}^{t}|G_k(t-s)||\mathcal{N}_k(s)|^2\,
       \mathrm{d}s\int_{0}^{t}|G_k(t-s)|\,\mathrm{d}s\right)\nonumber\\
              \lesssim&\sum_{k\neq 0}|k|^{2|\alpha|}\langle kt\rangle^{2|\omega|}|\mathcal{N}_k(t)|^2+\sum_{k\neq 0}|k|^{2|\alpha|}\langle kt\rangle^{2|\omega|}\nonumber\\
              &\times\int_{0}^{t}
              |k|^{-1}\langle k(t-s)\rangle^{-2N_{\max}-2}|\mathcal{N}_k(s)|^2\,\mathrm{d}s\int_{0}^{t}|k|^{-1}\langle k(t-s)\rangle^{-2N_{\max}-2}\,\mathrm{d}s\nonumber\\
                     \lesssim&\sum_{k\neq 0}|k|^{2|\alpha|}\langle kt\rangle^{2|\omega|}|\mathcal{N}_k(t)|^2+\sum_{k\neq 0}|k|^{2|\alpha|}\langle kt\rangle^{2|\omega|}\int_{0}^{t}
              |k|^{-3}\langle k(t-s)\rangle^{-2N_{\max}-2}|\mathcal{N}_k(s)|^2\,\mathrm{d}s\nonumber \\
              \lesssim&\sum_{k\neq 0}|k|^{2|\alpha|}\langle kt\rangle^{2|\omega|}|\mathcal{N}_k(t)|^2+\sum_{k\neq 0}\int_{0}^{t}
              |k|^{-3}\langle k(t-s)\rangle^{-2}|k|^{2|\alpha|}\langle ks\rangle^{2|\omega|}|\mathcal{N}_k(s)|^2\,\mathrm{d}s,
\end{align}
where we used the following fact
\begin{align*}
    \langle kt\rangle^{2|\omega|}\lesssim \langle k(t-s)\rangle^{2|\omega|}\langle ks\rangle^{2|\omega|}\leq \langle k(t-s)\rangle^{2N_{max}}\langle ks\rangle^{2|\omega|}.
\end{align*}
To estimate the terms in \eqref{rho-infty}, we first claim that
\begin{align}\label{rho-claim}
 &\sum_{|\alpha|+|\omega|\leq N_{\max}}\Big[\sum_{k\neq 0}|k|^{2|\alpha|}\langle kt\rangle^{2|\omega|}|\mathcal{N}_k(t)|^2\Big]+\nu^\kappa\int_0^t\sum_{k\neq 0}|k|^{2|\alpha|}\langle k\tau\rangle^{2|\omega|}|\mathcal{N}_k(\tau)|^2\mathrm{d}\tau
\lesssim\mathcal{M}^2\nu^{2\kappa}
\nonumber\\&+\mathcal{M}\sum_{|\alpha|+|\omega|\leq N_{\max}}\left\{\sup_{0\leq \tau\leq t}\sum_{k\neq 0}|k|^{2|\alpha|}\langle k\tau\rangle^{2|\omega|}|\widehat{\rho}_k(\tau)|^2+\nu^\kappa\int_0^t\sum_{k\neq 0}|k|^{2|\alpha|}\langle ks\rangle^{2|\omega|}|\widehat{\rho}_k(s)|^2\,\mathrm{d}s\right\}.
\end{align}
Applying this claim in \eqref{rho-infty} yields
\begin{align}\label{rho-infty-1}
 &\sum_{|\alpha|+|\omega|\leq N_{\max}}\Big[\sum_{k\neq 0}|k|^{2|\alpha|}\langle kt\rangle^{2|\omega|}|\widehat{\rho}_k(t)|^2\Big]\nonumber\\
\lesssim&\mathcal{M}^2\nu^{2\kappa}+\mathcal{M}\sum_{|\alpha|+|\omega|\leq N_{\max}}\left\{\sup_{0\leq \tau\leq t}\sum_{k\neq 0}|k|^{2|\alpha|}\langle k\tau\rangle^{2|\omega|}|\widehat{\rho}_k(\tau)|^2+\nu^\kappa\int_0^t\sum_{k\neq 0}|k|^{2|\alpha|}\langle ks\rangle^{2|\omega|}|\widehat{\rho}_k(s)|^2\,\mathrm{d}s\right\}\nonumber\\
&+\sum_{k\neq 0}\int_{0}^{t}
              |k|^{-3}\langle k(t-s)\rangle^{-2}\Bigg[\mathcal{M}^2\nu^{2\kappa}\nonumber\\
&\qquad+\mathcal{M}\sum_{|\alpha|+|\omega|\leq N_{\max}}\left\{\sup_{0\leq \tau\leq t}\sum_{k\neq 0}|k|^{2|\alpha|}\langle k\tau\rangle^{2|\omega|}|\widehat{\rho}_k(\tau)|^2+\nu^\kappa\int_0^s\sum_{k\neq 0}|k|^{2|\alpha|}\langle ks\rangle^{2|\omega|}|\widehat{\rho}_k(\tau)|^2\mathrm{d}\tau\right\}\Bigg]\,\mathrm{d}s\nonumber\\
\lesssim&\mathcal{M}^2\nu^{2\kappa}\nonumber\\
&+\mathcal{M}\sum_{|\alpha|+|\omega|\leq N_{\max}}\left\{\sup_{0\leq \tau\leq t}\sum_{k\neq 0}|k|^{2|\alpha|}\langle k\tau\rangle^{2|\omega|}|\widehat{\rho}_k(\tau)|^2+\nu^\kappa\int_0^t\sum_{k\neq 0}|k|^{2|\alpha|}\langle ks\rangle^{2|\omega|}|\widehat{\rho}_k(s)|^2\,\mathrm{d}s\right\}.
\end{align}
Similarly, we can obtain
\begin{align}\label{rho-integral-1}
    &\nu^\kappa\int_0^t\sum_{k\neq 0}|k|^{2|\alpha|}\langle ks\rangle^{2|\omega|}|\widehat{\rho}_k(s)|^2\,\mathrm{d}s
\lesssim\mathcal{M}^2\nu^{2\kappa}\nonumber\\&+\mathcal{M}\sum_{|\alpha|+|\omega|\leq N_{\max}}\left\{\sup_{0\leq \tau\leq t}\sum_{k\neq 0}|k|^{2|\alpha|}\langle k\tau\rangle^{2|\omega|}|\widehat{\rho}_k(\tau)|^2+\nu^\kappa\int_0^t\sum_{k\neq 0}|k|^{2|\alpha|}\langle ks\rangle^{2|\omega|}|\widehat{\rho}_k(s)|^2\,\mathrm{d}s\right\}.
\end{align}
Therefore, applying the Parseval identity, combing \eqref{rho-infty-1} and \eqref{rho-integral-1}, we derive \eqref{Th-rho-1} and thus complete the whole proof.
Therefore, it suffices to verify the validity of the claim \eqref{rho-claim}.
Note that
\begin{align}\label{def-N-k-int-NL}
\mathcal{N}_k(t)=&\int_{{\mathbb{R}}^3}S_k(t)\hat{f}_{0k}(v)\mu^{\frac{1}{2}}\,\mathrm{d}v+\int_{0}^{t}\int_{{\mathbb{R}}^3}S_k(t-\tau)[\widehat{\mathcal{R}}_k(\tau,v)]\mu^{\frac{1}{2}}\,\mathrm{d}v\nonumber\\
    =&\int_{{\mathbb{R}}^3}S_k(t)\hat{f}_{0k}(v)\mu^{\frac{1}{2}}\,\mathrm{d}v+\int_{0}^{t}\int_{{\mathbb{R}}^3}S_k(t-\tau)\mathcal{F}_x\left[-\frac12\nabla_x\phi\cdot vf+\nabla_x\phi\cdot \nabla_vf\right]\mu^{\frac{1}{2}}\,\mathrm{d}v\mathrm{d}\tau\nonumber\\
    &+\int_{0}^{t}\int_{{\mathbb{R}}^3}S_k(t-\tau)\mathcal{F}_x\left[\nu\Gamma(f,f)\right]\mu^{\frac{1}{2}}\,\mathrm{d}vt\tau\nonumber\\
    =&:\mathcal{N}_{k,\textbf{initial}}(t)+\mathcal{N}_{k,\textbf{NL-electric}}(t)+\mathcal{N}_{k,\textbf{NL-collision}}(t),
\end{align}
where $\mathcal{F}_x$ denotes the Fourier transform with respect to the variable $x$. Now we estimate terms in \eqref{def-N-k-int-NL} one by one.\newline

\noindent $\bullet$~(Estimate of $\mathcal{N}_{k,\textbf{initial}}(t)$):~ We apply \eqref{Pro-linear-decay-1} with $\eta=k$ and $N=N_{\max}+1$ to have
\begin{align*}
    &\sum_{k\neq 0}|k|^{2|\alpha|}\langle kt\rangle^{2|\omega|}|\mathcal{N}_{k,\textbf{initial}}(t)|^2    \lesssim \sum_{k\neq 0}|k|^{2|\alpha|}\langle kt\rangle^{2|\omega|}\left|\int_{{\mathbb{R}}^3}S_k(t)\hat{f}_{0k}(v)\mu^{\frac{1}{2}}\,\mathrm{d}v\right|^2\nonumber\\
   &\qquad \lesssim\sum_{k\neq 0}|k|^{2|\alpha|}\langle kt\rangle^{2|\omega|}\langle kt+k\rangle^{-2N_{\max}-2}
			\langle \nu^\kappa t\rangle^{-\frac32}\mathbb{E}_{k,\eta,N_{\max}+1}^{(0,5)}[f_0]    \lesssim\mathcal{M}^2\nu^{2\kappa}\langle t\rangle^{-2}.
\end{align*}
Similarly, we also deduce that
\begin{align*}
    &\int_0^t\sum_{k\neq 0}|k|^{2|\alpha|}\langle ks\rangle^{2|\omega|}|\mathcal{N}_{k,\textbf{initial}}(s)|^2\,\mathrm{d}s
    \lesssim\int_0^t\mathcal{M}^2\nu^{2\kappa}\langle s\rangle^{-2}\,\mathrm{d}s
    \lesssim\mathcal{M}^2\nu^{2\kappa}.
\end{align*}
Thus we conclude that
\begin{align}\label{Nonli-dens-ini}
    &\sum_{k\neq 0}|k|^{2|\alpha|}\langle kt\rangle^{2|\omega|}|\mathcal{N}_{k,\textbf{initial}}(t)|^2+\nu^\kappa\int_0^t\sum_{k\neq 0}|k|^{2|\alpha|}\langle ks\rangle^{2|\omega|}|\mathcal{N}_{k,\textbf{initial}}(s)|^2\,\mathrm{d}s\lesssim\mathcal{M}^2\nu^{2\kappa}.
\end{align}
\newline

\noindent $\bullet$~(Estimate of $\mathcal{N}_{k,\textbf{NL-electric}}(t)$):~Under the {\it a priori} assumption \eqref{priori-assu}, we can deduce that
\begin{align}\label{Nonli-dens-ele}
    &\sum_{k\neq 0}|k|^{2|\alpha|}\langle kt\rangle^{2|\omega|}|\mathcal{N}_{k,\textbf{NL-electric}}(t)|^2+\nu^\kappa\int_0^t\sum_{k\neq 0}|k|^{2|\alpha|}\langle ks\rangle^{2|\omega|}|\mathcal{N}_{k,\textbf{NL-electric}}(s)|^2\,\mathrm{d}s\nonumber\\
    \lesssim&\mathcal{M}\sum_{|\alpha|+|\omega|\leq N_{\max}}\left\{\sup_{0\leq \tau\leq t}\sum_{k\neq 0}|k|^{2|\alpha|}\langle k\tau\rangle^{2|\omega|}|\widehat{\rho}_k(\tau)|^2+\nu^\kappa\int_0^t\sum_{k\neq 0}|k|^{2|\alpha|}\langle ks\rangle^{2|\omega|}|\widehat{\rho}_k(s)|^2\,\mathrm{d}s\right\}.
\end{align}
Since the proof of \eqref{Nonli-dens-ini} is based on the norms \eqref{Energy-N} and \eqref{Dissi-N} and time decay estimates in Proposition \ref{linear-decay}, which are similar to the counterparts in  \cite{CLN-JAMS-2023}, it can be done in the same way as the part ``\textbf{Nonlinear interaction I}" of Theorem 8.1 in \cite{CLN-JAMS-2023}. We omit the proof for brevity.
\newline

\noindent $\bullet$~(Estimate of $\mathcal{N}_{k,\textbf{NL-collision}}(t)$):~We claim that
\begin{align}\label{colli-sum}
    &\sum_{|\alpha|+|\omega|\leq N_{max}}\sum_{k\neq 0}|k|^{2|\alpha|}\langle kt\rangle^{2|\omega|}|\mathcal{N}_{k,\textbf{NL-collision}}(t)|^2\nonumber\\
    &+\nu^\kappa\sum_{|\alpha|+|\omega|\leq N_{max}}\int_0^t\sum_{k\neq 0}|k|^{2|\alpha|}\langle ks\rangle^{2|\omega|}|\mathcal{N}_{k,\textbf{NL-collision}}(s)|^2\,\mathrm{d}s
    \lesssim\mathcal{M}^2\nu^{2\kappa}.
\end{align}
Once \eqref{colli-sum} is verified, we collect the estimates \eqref{Nonli-dens-ini}, \eqref{Nonli-dens-ele}, and \eqref{colli-sum} to obtain \eqref{rho-claim}.

Now we turn to the proof of \eqref{rho-claim}. For $|\alpha|+|\omega|\leq N_{max}$, we use \eqref{Pro-linear-decay-1} and Lemma \ref{Lemma GammaL} to have
\begin{align}\label{dens-colli}
   &\sum_{k\neq 0}|k|^{2|\alpha|}\langle kt\rangle^{2|\omega|}|\mathcal{N}_{k,\textbf{NL-collision}}(t)|^2\nonumber\\
   \lesssim&\sum_{k\neq 0}|k|^{2|\alpha|}\langle kt\rangle^{2|\omega|}\left|\int_{0}^{t}\int_{{\mathbb{R}}^3}S_k(t-\tau)\mathcal{F}_x\left[\nu\Gamma(f,f)\right]\mu^{\frac{1}{2}}\,\mathrm{d}v\mathrm{d}\tau\right|^2\nonumber\\
     \lesssim&\nu^{2}
   \sum_{k\neq 0}|k|^{2|\alpha|}\left|\int_{0}^{t}
\langle \nu^\kappa(t-\tau)\rangle^{-\frac32}  \left\{\mathbb{E}_{k,k\tau,|\omega|}^{(0,5)}[\mathcal{F}_x\left[\Gamma(f,f)\right](\tau,k,v)]\right\}^{\frac12}\mathrm{d}\tau\right|^2\nonumber\\
    \lesssim&\nu^{2}
   \sum_{k\neq 0}|k|^{2|\alpha|}\int_{0}^{t}
\mathbb{E}_{k,k\tau,|\omega|}^{(0,5)}\left[\mathcal{F}_x
\left[\Gamma(f,f)\right](\tau,k,v)\right]\mathrm{d}\tau\int_{0}^{t}
\langle \nu^\kappa(t-\tau)\rangle^{-3} \mathrm{d}\tau\nonumber\\
  \lesssim&\nu^{2-\kappa}\sum_{
|\bar{\omega}|\leq |\omega|,|\beta|\leq 1}\int_{0}^{t}\nu^{2\kappa|\beta|}
\|\langle v\rangle^{\frac{-5\gamma}{2s}}\partial^{\alpha}_x\partial^\beta_vY^{\bar{\omega}}
\Gamma(f,f)(\tau)\|^2\,\mathrm{d}\tau\nonumber\\
\lesssim&\nu^{2-\kappa}\sum_{\alpha'+\alpha''=\alpha
,|\bar{\omega}|+|\bar{\omega}'|\leq |\omega|\atop
|\beta'|+|\beta''|\leq |\beta|\leq 1}\int_0^t\nu^{2\kappa|\beta|}\left\|\left|\langle v\rangle^{\frac{-5\gamma}{2s}}\partial^{\alpha'}_x\partial^{\beta'}_vY^{\bar{\omega}}
f(\tau)\right|_{L^2}\left|\langle v\rangle^{\frac{-5\gamma}{2s}}\partial^{\alpha''}_x\partial^{\beta''}_vY^{\bar{\omega}'}f(\tau)
\right|_{H^{2s}}\right\|^2\,\mathrm{d}\tau\nonumber\\
=&\nu^{2-\kappa}\sum_{\alpha'+\alpha''= \alpha
\atop|\bar{\omega}|+|\bar{\omega}'|\leq |\omega|}\int_0^t\left\|\left|\langle v\rangle^{\frac{-5\gamma}{2s}}\partial^{\alpha'}_xY^{\bar{\omega}}f\right|_{L^2}\left|\langle v\rangle^{\frac{-5\gamma}{2s}}\partial^{\alpha''}_xY^{\bar{\omega}'}f(\tau)
\right|_{H^{2s}}\right\|^2\,\mathrm{d}\tau\nonumber\\
&+\nu^{2-\kappa}\sum_{\alpha'+\alpha''= \alpha\atop |\bar{\omega}|+|\bar{\omega}'|\leq |\omega|,
|\beta|= 1}\nu^{2\kappa|\beta|}\int_0^t\left\|\left|\langle v\rangle^{\frac{-5\gamma}{2s}}\partial^{\alpha'}_x\partial^{\beta}_vY^{\bar{\omega}}f(\tau)\right|_{L^2}
\left|\langle v\rangle^{\frac{-5\gamma}{2s}}\partial^{\alpha''}_xY^{\bar{\omega}'}f(\tau)\right|_{H^{2s}}
\right\|^2\,\mathrm{d}\tau\nonumber\\
&+\nu^{2-\kappa}\sum_{\alpha'+\alpha''= \alpha\atop |\bar{\omega}|+|\bar{\omega}'|\leq |\omega|,
|\beta|= 1}\nu^{2\kappa|\beta|}\int_0^t\left\|\left|\langle v\rangle^{\frac{-5\gamma}{2s}}\partial^{\alpha'}_xY^{\bar{\omega}}f(\tau)\right|_{L^2}
\left|\langle v\rangle^{\frac{-5\gamma}{2s}}\partial^{\alpha''}_x\partial^{\beta}_vY^{\bar{\omega}'}f(\tau)
\right|_{H^{2s}}
\right|^2_{L^2_x}\,\mathrm{d}\tau \nonumber\\
=:&\mathcal{I}_1+\mathcal{I}_2+\mathcal{I}_3.
\end{align}
For brevity, we only focus on the estimate of the third term $\mathcal{I}_3$ in \eqref{dens-colli}, since the other terms can be treated in a similar or simpler manner. Note that when $ |\alpha'|+|\bar{\omega}|\leq N_{max}/2$, we have $$N_{max}-|\alpha'|-|\bar{\omega}|\geq N_{max}/2>4,$$
as $N_{max}\geq 9$. On the other hand, when $ |\alpha'|+|\bar{\omega}|> N_{max}/2$,  it follows that $ |\alpha'|+|\bar{\omega}|\geq5$, and thus $$|\alpha''|+|\bar{\omega}'|+1+2\leq N_{max}-2.$$ Using the definitions of $\widetilde{\mathbb{E}}^{(0,\iota_0)}_{N}[f](t)$ and $\widetilde{\mathbb{D}}^{(0,\iota_0)}_{N}[f](t)$ in \eqref{Energy-N} and \eqref{Dissi-N}, together with the {\it a priori} assumption \eqref{priori-assu}, we can estimate $\mathcal{I}_3$ as
  \begin{align}\label{colli-I3}
\mathcal{I}_3\lesssim&\nu^{2-\kappa-2\kappa s}\sum_{|\alpha'|+|\bar{\omega}|\leq N_{max}/2
}\sum_{|\tilde{\alpha}|\leq2}\int_0^t\left\|\langle v\rangle^{\frac{-5\gamma}{2s}}\partial^{\alpha'+\tilde{\alpha}}_xY^{\bar{\omega}}f(\tau)\right\|^2
\nu^{2\kappa+2\kappa s}\left\|\langle v\rangle^{\frac{-5\gamma}{2s}}\partial^{\alpha''}_xY^{\bar{\omega}'}\nabla_vf(\tau)\right\|^2
_{H^{2s}}
\,\mathrm{d}\tau\nonumber\\
&+\nu^{2-\kappa-2\kappa s}\sum_{|\alpha'|+|\bar{\omega}|> N_{max}/2
}\sum_{|\tilde{\alpha}|\leq2}\int_0^t\left\|\langle v\rangle^{\frac{-5\gamma}{2s}}\partial^{\alpha'}_xY^{\bar{\omega}}f(\tau)\right\|^2
\nu^{2\kappa+2\kappa s}\left\|\langle v\rangle^{\frac{-5\gamma}{2s}}\partial^{\alpha''+\tilde{\alpha}}_xY^{\bar{\omega}'}
\nabla_vf(\tau)\right\|^2_{H^{2s}}
\,\mathrm{d}\tau\nonumber\\
\lesssim&\nu^{\kappa} \int_0^t \widetilde{\mathbb{E}}^{(0,\iota_0)}_{N_{\max}-2}[f(\tau)]
\widetilde{\mathbb{D}}^{(0,\iota_0)}_{N_{\max}}[f(\tau)]\,\mathrm{d}\tau+\nu^{\kappa} \int_0^t \widetilde{\mathbb{E}}^{(0,\iota_0)}_{N_{\max}}[f(\tau)]
\widetilde{\mathbb{D}}^{(0,\iota_0)}_{N_{\max}-2}[f(\tau)]\,\mathrm{d}\tau\,\nonumber\\
\lesssim&\mathcal{M}\nu^{2\kappa }\nu^\kappa\int_0^t\widetilde{\mathbb{D}}^{(0,\iota_0)}_{N_{\max}}[f(\tau)]\,\mathrm{d}\tau
+\mathcal{M}\nu^\kappa\int_0^t\widetilde{\mathbb{D}}^{(0,\iota_0)}_{N_{\max}-2}[f(\tau)]\,\mathrm{d}\tau\nonumber\\
\lesssim&\mathcal{M}^2\nu^{2\kappa},
\end{align}
where we used the interpolation estimate
\begin{align*}
&\nu^{1-\kappa+2s\kappa}\left\|\langle v\rangle^{\frac{-5\gamma}{2s}}\widetilde{f}(\tau)\right\|^2_{H^{2s}}\lesssim \left(\nu^{1-\kappa}\left\|\langle v\rangle^{\frac{-5\gamma}{2s}}\widetilde{f}(\tau)\right\|^2_{H^{s}}\right)^{1-s}\left(\nu^{1+\kappa}\left\|\langle v\rangle^{\frac{-5\gamma}{2s}}\nabla_v\widetilde{f}(\tau)\right\|^2_{H^{s}}\right)^{s}\\
&\qquad\lesssim \left(\nu^{1-\kappa}\left\|\langle v\rangle^{\frac{-5\gamma}{2s}-\gamma}\widetilde{f}(\tau)\right\|^2_{D}\right)^{1-s}\left(\nu^{1+\kappa}\left\|\langle v\rangle^{\frac{-5\gamma}{2s}-\gamma}\nabla_v\widetilde{f}(\tau)\right\|^2_{D}\right)^{s}
\\
&\qquad\lesssim \widetilde{\mathbb{D}}^{(0,\iota_0)}_{0}[f(\tau)],
\end{align*}
with $\widetilde{f} =\nu^{\kappa }\partial^{\alpha''}_xY^{\bar{\omega}'}
\nabla_vf.$
This is the very place we require that $\iota_0\geq N_{max}+9$, due to the following restriction:
\begin{align*}
  \left(\iota_0-N_{max}-2\right)\frac{-\gamma}{2s} \geq \frac{-5\gamma}{2s}-\gamma.
\end{align*}
Meanwhile, for $\mathcal{I}_1$ and $\mathcal{I}_2$, we similarly obtain
\begin{align}\label{colli-I12}
\mathcal{I}_1+\mathcal{I}_2\lesssim \mathcal{M}^2\nu^{2\kappa}.
\end{align}
Inserting \eqref{colli-I3} and \eqref{colli-I12} into \eqref{dens-colli} yields
\begin{align}\label{colli-infi}
  \sum_{|\alpha|+|\omega|\leq N_{max}}\sum_{k\neq 0}|k|^{2|\alpha|}\langle kt\rangle^{2|\omega|}|\mathcal{N}_{k,\textbf{NL-collision}}(t)|^2
    \lesssim\mathcal{M}^2\nu^{2\kappa}.
\end{align}
Similarly, we also have
\begin{align}\label{colli-inte}
\nu^\kappa\sum_{|\alpha|+|\omega|\leq N_{max}}\int_0^t\sum_{k\neq 0}|k|^{2|\alpha|}\langle ks\rangle^{2|\omega|}|\mathcal{N}_{k,\textbf{NL-collision}}(s)|^2\,\mathrm{d}s\lesssim\mathcal{M}^2\nu^{2\kappa}.
\end{align}
Combining \eqref{colli-infi} and \eqref{colli-inte}, we obtain \eqref{colli-sum}. This finishes the proof of Theorem \ref{dens-esti}.
\end{proof}

With the density estimates in  Theorem \ref{dens-esti}, we can obtain estimates on $\partial_t\phi(t)$  and $\nabla_x\phi(t)$.
\begin{lemma}\label{elec-esti}
Assume conditions in Theorem \ref{dens-esti} hold. Then for any $t\in [0, T]$, we have that
	\begin{eqnarray}\label{coro-elec}
		\sum_{|\alpha|+|\omega|\leq N_{max}-2,\atop|\alpha'|\leq 1}\|\partial^{\alpha}_xY^{\omega}\nabla_x\phi(t) \|_{L^{\infty}_x}+\sum_{|\alpha|+|\omega|\leq N_{max}-2,\atop|\alpha'|\leq 3}\|\partial^{\alpha+\alpha'}_xY^{\omega}\nabla_x\phi(t) \| \lesssim \langle t\rangle^{-2} \mathcal{M}\nu^{\kappa},
	\end{eqnarray}
and
\begin{eqnarray}\label{coro-phi}
		\|\partial_t\phi(t) \|_{L^{\infty}_x} \lesssim\sum_{|\alpha|+|\omega|\leq 3}\langle t\rangle^{-2}\left\| \partial^\alpha_xY^\omega f\right \|_{L^2_{x,v}}.
	\end{eqnarray}
\end{lemma}
\begin{proof} By the Poincar\'{e} inequality and the Sobolev inequalities, together with \eqref{Th-rho-1}, we obtain
\begin{align*}
  &\sum_{|\alpha|+|\omega|\leq N_{max}-2}\sum_{|\alpha'|\leq1}\|\partial^{\alpha+\alpha'}_xY^{\omega}\nabla_x\phi(t) \|_{L^{\infty}_x} \lesssim  \sum_{|\alpha|+|\omega|\leq N_{max}-2}\sum_{|\alpha'|\leq3}\|\partial^{\alpha+\alpha'}_xY^{\omega}\nabla_x\phi(t) \|
  \\
  &\qquad\lesssim\sum_{|\alpha|+|\omega|\leq N_{max}-2}\|\partial^{\alpha}_xY^{\omega}\nabla^4_x\phi(t) \|\lesssim\sum_{|\alpha|+|\omega|\leq N_{max}-2}\|\partial^{\alpha}_xY^{\omega}\nabla^2_x\rho_{\neq0}(t) \|\\
  &\qquad\lesssim \langle t\rangle^{-2}\sum_{|\alpha|+|\omega|\leq N_{max}}\|\partial^{\alpha}_xY^{\omega}\rho_{\neq0}(t) \|\lesssim \langle t\rangle^{-2} \mathcal{M}\nu^{\kappa}.
\end{align*}
In the same way, one can show
\begin{align*}
\sum_{|\alpha|+|\omega|\leq N_{max}-2,\atop|\alpha'|\leq 3}\|\partial^{\alpha+\alpha'}_xY^{\omega}\nabla_x\phi(t) \| \lesssim \langle t\rangle^{-2} \mathcal{M}\nu^{\kappa}.
\end{align*}
Then \eqref{coro-elec} follows.

To estimate \eqref{coro-phi}, we first note that
$$\int_{\T^3}\partial_t\phi\,\mathrm{d}x=0,\quad -\Delta_x\partial_t\phi=\nabla_x\cdot\int_{\mathbb{R}^3}v\mu^{\frac12} f\,\mathrm{d}v. $$
Applying the Poincar\'{e} and Sobolev inequalities then gives
\begin{align*}
\|\partial_t\phi(t) \|_{L^{\infty}_x}
\lesssim&  \sum_{|\alpha|=4}\left\|\partial^\alpha_x\partial_t\phi(t) \right \|_{L^2_x}\lesssim  \sum_{|\omega|=2,|\alpha|=2}t^{-2}\left\|Y^\omega\partial^\alpha_x\partial_t\phi(t) \right \|_{L^2_x}\\
\lesssim&\sum_{|\omega|=2}t^{-2}\left\|Y^\omega \nabla_x\int_{\mathbb{R}^3}v\mu^{\frac12} f\right \|_{L^2_x}\lesssim\sum_{|\alpha|+|\omega|\leq 3}t^{-2}\left\| \partial^\alpha_xY^\omega f\right \|_{L^2_{x,v}}.
\end{align*}
This ends the proof of Lemma \ref{elec-esti}.
\end{proof}

\section{Nonlinear energy estimates and global existence}\label{Nonlinear energy estimates}
In this section, we prove the global existence of a smooth solution
$
f(t,x,v)$ to the Cauchy problem \eqref{f} and \eqref{f-initial} by combining the a priori energy estimates with a continuation argument and the local existence theory. To derive the {\it a priori} energy estimates, we first establish the enhanced dissipation of the VPB system, and then perform weighted energy estimates to close the energy norms $\widetilde{\mathbb{E}}_{N}^{(q,\iota_0)}[f](t)$ for $N\leq N_{max}$. With these estimates at hand, the standard continuation argument, together with the local existence result, yields the desired global existence of a smooth solution to the Cauchy problem \eqref{f} and \eqref{f-initial}.

\subsection{Enhanced dissipation}
In this subsection, we will capture the enhanced dissipation of the VPB system \eqref{VPB}, which is essential to obtain the global existence of the VPB system in our weakly collision regime. As in \cite{CLN-JAMS-2023},  we achieve this aim by the energy method.

\begin{lemma}\label{elec-potenfield}

For $|\alpha|+|\beta|+|\omega|=N\leq N_{max}$, we denote $w_{\iota_0-N}=w^{(q)}_{\iota_0,N}$ as the weight function defined in \eqref{def-weight-1} for simplicity. Under the same assumptions in Theorem \ref{dens-esti}, the smooth solution $f(t,x,v)$ to the Cauchy problem \eqref{f} and \eqref{f-initial} satisfies
	
\begin{eqnarray}\label{Enhanced-eqn}
&&\frac{\mathrm{d}}{\mathrm{d}t}\int_{\mathbb{T}^3\times\mathbb{R}^3}
e^{\phi}\nu^{2\kappa|\beta|+\kappa}w^2_{\iota_0-N-1}\partial_{x_j}
\partial^\alpha_x\partial^\beta_v Y^\omega f\partial_{v_j}\partial^\alpha_x\partial^\beta_v Y^\omega f \,\mathrm{d}v\mathrm{d}x\nonumber\\
&&+\int_{\mathbb{T}^3\times\mathbb{R}^3}e^{\phi}
\nu^{2\kappa|\beta|+\kappa}w^2_{\iota_0-N-1}(\partial_{x_j}\partial^\alpha_x\partial^\beta_v Y^\omega f)^2\,\mathrm{d}v \mathrm{d}x\nonumber\\	
&\lesssim&\langle t\rangle^{-2}\mathcal{M}^{\frac{1}{2}}\nu^{\kappa}\widetilde{\mathbb{E}}_{N}[f](t) +\sum_{|\alpha|+|\omega|\leq N}\|\partial^\alpha_x Y^\omega\rho_{\neq0}\|\min\left\{\left(\widetilde{\mathbb{E}}^{(q,\iota_0)}_N[f](t)\right)^{1/2},
    \left(\widetilde{\mathbb{D}}^{(q,\iota_0)}_N[f](t)\right)^{1/2} \right\}\nonumber\\	&&+A_0^{-\frac{s}{2}}\nu^{\kappa}\widetilde{\mathbb{D}}_{N}^{(q,\iota_0)}[f](t)+\nu^{\kappa} \left(\widetilde{\mathbb{E}}_{N}^{(q,\iota_0)}[f](t)
\widetilde{\mathbb{D}}_{N_{\max}-2}^{(q,\iota_0)}[f](t)
\widetilde{\mathbb{D}}_{N}^{(q,\iota_0)}[f](t)\right)^{1/2}.
\end{eqnarray}
\end{lemma}

\begin{proof} For brevity, we focus primarily on the proof of \eqref{Enhanced-eqn} in the case $|\alpha|=|\beta|=|\omega|=0$.
Applying $\partial_{x_j}$ and $\partial_{v_j}$ to \eqref{f} yields
\begin{eqnarray}\label{enhanced-eqn-x}
	&&\partial_t\partial_{x_j}f+ v  \cdot\nabla_x\partial_{x_j}f+\nu\partial_{x_j}
\mathcal{L}f\nonumber\\
	&=&-\partial_{x_j}\nabla_x\phi\cdot v \mu^{\frac12}+\partial_{x_j}\{\nabla_x\phi\cdot\nabla_{ v  }f\}-\frac12 \partial_{x_j}\{\nabla_x\phi\cdot v f\}+\nu \partial_{x_j}{\Gamma}(f,f),
\end{eqnarray}
and
\begin{eqnarray}\label{enhanced-eqn-v}
	&&\partial_t\partial_{v_j}f+ v  \cdot\nabla_x\partial_{v_j}f+\partial_{x_j}f+\nu\partial_{v_j}\mathcal{L} f\nonumber\\
	&=&-\partial_{v_j}\{\nabla_x\phi\cdot v \mu^{\frac12}\}+\partial_{v_j}\{\nabla_x\phi\cdot\nabla_{ v  }f\}-\frac12 \partial_{v_j}\{\nabla_x\phi\cdot v f\}+\nu \partial_{v_j}{\Gamma}(f,f).
\end{eqnarray}
We multiply \eqref{enhanced-eqn-x} and \eqref{enhanced-eqn-v}  by $\partial_{v_j} f$ and $\partial_{x_j} f$ respectively  to have
\begin{eqnarray*}
&&\partial_t(\partial_{x_j} f\partial_{v_j} f)+v\cdot\nabla_x(\partial_{x_j} f\partial_{v_j} f)	+\nabla_x\phi\cdot v(\partial_{x_j} f\partial_{v_j} f)\nonumber\\
&&+(\partial_{x_j}f)^2+\nu(\partial_{x_j}f\partial_{v_j}\mathcal{L}f+\partial_{v_j}f\partial_{x_j}
\mathcal{L}f)\nonumber\\
&=&\nabla_x\phi\cdot\nabla_v(\partial_{x_j}f\partial_{v_j}f)+\partial_{x_j}\nabla_x\phi\cdot\nabla_v f\partial_{v_j}f-\frac{1}{2}\partial_{x_j}\nabla_x\phi\cdot v f\partial_{v_j}f-\frac12\partial_{x_j}\phi f\partial_{x_j}f\nonumber\\
&&-\partial_{x_j}\nabla_x\phi\cdot v\mu^{\frac12}\partial_{v_j}f-\partial_{v_j}\{\nabla_x\phi\cdot v\mu^{\frac12}\}\partial_{x_j}f+\nu\partial_{x_j}\Gamma(f,f)\partial_{v_j}f+\nu\partial_{v_j}\Gamma(f,f)\partial_{x_j}f.
\end{eqnarray*}
Multiplying the above equation by $\nu^\kappa e^\phi w^2_{\iota_0-1}$, we can further obtain
\begin{eqnarray}\label{xv-1}
	&&\frac{\mathrm{d}}{\mathrm{d}t}\int_{\mathbb{T}^3\times\mathbb{R}^3}\nu^\kappa e^\phi w^2_{\iota_0-1}\partial_{x_j} f\partial_{v_j} f \,\mathrm{d}v\mathrm{d}x+\int_{\mathbb{T}^3\times\mathbb{R}^3}\nu^\kappa e^\phi w^2_{\iota_0-1}(\partial_{x_j}f)^2\,\mathrm{d}v \mathrm{d}x\nonumber\\
	&=&\int_{\mathbb{T}^3\times\mathbb{R}^3}(\partial_t+v\cdot\nabla_x)(\nu^\kappa e^\phi w^2_{\iota_0-1})\partial_{x_j} f\partial_{v_j} f \,\mathrm{d}v\mathrm{d}x\nonumber\\
    &&-\int_{\mathbb{T}^3\times\mathbb{R}^3}\nu^\kappa e^\phi w^2_{\iota_0-1}\nabla_x\phi\cdot v(\partial_{x_j} f\partial_{v_j} f)\,\mathrm{d}v\mathrm{d}x\nonumber\\
	&&-\int_{\mathbb{T}^3\times\mathbb{R}^3}\nu^\kappa e^\phi w^2_{\iota_0-1}\nu(\partial_{x_j}f\partial_{v_j}\mathcal{L}f
+\partial_{v_j}f\partial_{x_j}\mathcal{L}f)\,\mathrm{d}v\mathrm{d}x\nonumber\\
	&&+\int_{\mathbb{T}^3\times\mathbb{R}^3}\nu^\kappa e^\phi w^2_{\iota_0-1}\nabla_x\phi\cdot\nabla_v(\partial_{x_j}f\partial_{v_j}f)\,\mathrm{d}v\mathrm{d}x\nonumber\\
	&&+\int_{\mathbb{T}^3\times\mathbb{R}^3}\nu^\kappa e^\phi w^2_{\iota_0-1}\left\{\partial_{x_j}\nabla_x\phi\cdot\nabla_v f\partial_{v_j}f-\frac{1}{2}\partial_{x_j}\nabla_x\phi\cdot v f\partial_{v_j}f-\frac12\partial_{x_j}\phi f\partial_{x_j}f\right\}\,\mathrm{d}v\mathrm{d}x\nonumber\\
	&&-\int_{\mathbb{T}^3\times\mathbb{R}^3}\nu^\kappa e^\phi w^2_{\iota_0-1}\{\partial_{x_j}\nabla_x\phi\cdot v\mu^{\frac12}\partial_{v_j}f+\partial_{v_j}\{\nabla_x\phi\cdot v\mu^{\frac12}\}\partial_{x_j}f\}\,\mathrm{d}v\mathrm{d}x\nonumber\\
	&&+\int_{\mathbb{T}^3\times\mathbb{R}^3}\nu^\kappa e^\phi w^2_{\iota_0-1}\{\nu\partial_{x_j}\Gamma(f,f)\partial_{v_j}f+\nu\partial_{v_j}
\Gamma(f,f)\partial_{x_j}f\}\,\mathrm{d}v\mathrm{d}x.
\end{eqnarray}
Now we estimate the terms in \eqref{xv-1} individually.
Using the {\it a priori} assumption \eqref{priori-assu} together with \eqref{coro-phi}, the first and second terms on the right-hand side of the above inequality can be estimated as
\begin{eqnarray*}
  &&\int_{\mathbb{T}^3\times\mathbb{R}^3}(\partial_t+v\cdot\nabla_x)(\nu^\kappa e^\phi w^2_{\iota_0-1})\partial_{x_j} f\partial_{v_j} f \,\mathrm{d}v\mathrm{d}x\nonumber\\
  &&-\int_{\mathbb{T}^3\times\mathbb{R}^3}\nu^\kappa e^\phi w^2_{\iota_0-1}\nabla_x\phi\cdot v(\partial_{x_j} f\partial_{v_j} f)\,\mathrm{d}v\mathrm{d}x\nonumber\\
  &=&\int_{\mathbb{T}^3\times\mathbb{R}^3}\partial_t\phi \nu^\kappa e^\phi w^2_{\iota_0-1}\partial_{x_j} f\partial_{v_j}f\,\mathrm{d}v\mathrm{d}x\lesssim\|\partial_t\phi\|_{L^\infty_x}
  \|w_{\iota_0-1}\partial_{x_j} f\|\nu^\kappa\|w_{\iota_0-1}\partial_{v_j}f\|\\
  &\lesssim& \langle t\rangle^{-2}\mathcal{M}^{\frac{1}{2}}\nu^{\kappa}\left(\|w_{\iota_0-1}\partial_{x_j} f\|^2+\nu^{2\kappa}\|w_{\iota_0-1}\partial_{v_j}f\|^2\right)\lesssim
\langle t\rangle^{-2}\mathcal{M}^{\frac{1}{2}}\nu^{\kappa}\widetilde{\mathbb{E}}^{(q,\iota_0)}_0[f](t).
\end{eqnarray*}
From \eqref{Gamma-noncut-1} in Lemma \ref{Gamma-noncut}, the third term can be bounded
\begin{eqnarray*}
     &&\nu^{1+\kappa} \|w_{\iota_0-1}\partial_{x_j} f\|_D\sum_{|\beta_j|\leq1} \|w_{\iota_0-1}\partial^{\beta_j}_{v_j} f\|_D\nonumber\\
     &&\qquad
     \lesssim\nu^{1+\kappa}\left(A_0^{-1}\nu^{\kappa-1}\widetilde{\mathbb{D}}_{0}^{(q,\iota_0)}
     [f](t)\right)^{1/2}\left(\nu^{-\kappa-1}
     \widetilde{\mathbb{D}}_{0}^{(q,\iota_0)}[f](t)\right)^{1/2}
      \lesssim
     A_0^{-\frac{1}{2}}\nu^{\kappa}\widetilde{\mathbb{D}}_{0}^{(q,\iota_0)}[f](t).\nonumber
\end{eqnarray*}
 Applying integration by parts with respect to $v$ in the fourth term, we can use the {\it a priori} assumption \eqref{priori-assu} together with the estimate \eqref{coro-elec} to control the fourth and the fifth terms  by
\begin{eqnarray*}
&&\left(\|\nabla_x\phi\|_{L^{\infty}}+\|\nabla^2_x\phi\|_{L^{\infty}}\right)
\nu^\kappa\left(\|w_{\iota_0-1}\partial_{x_j} f\|^2+\|w_{\iota_0-1}\partial_{v_j}f\|^2+\|w_{\iota_0}f\|^2\right)\\
&& \qquad\lesssim
\mathcal{M}^{\frac{1}{2}}\langle t\rangle^{-2}
\nu^{2\kappa}\left(\|w_{\iota_0-1}\partial_{x_j} f\|^2+\|w_{\iota_0-1}\partial_{v_j}f\|^2+\|w_{\iota_0}f\|^2\right)\\
&& \qquad\lesssim
\langle t\rangle^{-2}\mathcal{M}^{\frac{1}{2}}\widetilde{\mathbb{E}}^{(q,\iota_0)}_0[f](t),
\end{eqnarray*}
where we have used the fact $w_{\iota_0-1}|v|\leq w_{\iota_0}$ since $-\gamma>2s$.

By integrating by parts in $v$ and $x$,  the sixth term can be dominated by
\begin{align*}
    &\nu^\kappa\|\nabla_x\phi\|\left(\|\nabla_x\phi f\|+ \|\nabla_x f\|\right)\lesssim
    \langle t\rangle^{-2} \mathcal{M}\nu^{2\kappa}\left(\|\nabla_x\phi\|_{L^{\infty}}^2+\| f\|^2\right)+\nu^\kappa\|\rho_{\neq0}\|\|\nabla_x f\|\\
    &\qquad\lesssim \langle t\rangle^{-2} \mathcal{M}\nu^{2\kappa} \widetilde{\mathbb{E}}^{(q,\iota_0)}_0[f](t)+\nu^\kappa \|\rho_{\neq0}\|\min\left\{\left(\widetilde{\mathbb{E}}^{(q,\iota_0)}_0[f](t)\right)^{1/2}, \left(\widetilde{\mathbb{D}}^{(q,\iota_0)}_0[f](t)\right)^{1/2} \right\},
\end{align*}
where we have used \eqref{coro-elec}.

Applying \eqref{Gamma-noncut-1}, the Sobolev inequalities and the {\it a priori} assumption \eqref{priori-assu}, the last term can be bounded as
 \begin{align}\label{enhan-00}
 &\nu^{1+\kappa} \Big(\left|w_{\iota_0-1}f\right|_{L^2}\left|e^{{q}\langle v\rangle}\partial_{x_j} f\right|_{D}+\left|w_{\iota_0-1}\partial_{x_j}f\right|_{L^2}\left|e^{{q}\langle v\rangle} f\right|_{D} \nonumber\\
&+\left|w_{\iota_0-1}f\right|_{D}\left|e^{{q}\langle v\rangle}\partial_{x_j} f\right|_{L^2}+\left|w_{\iota_0-1}\partial_{x_j}f\right|_{D}\left|e^{{q}\langle v\rangle} f\right|_{L^2},\left|w_{\iota_0-1} \partial_{v_j}f\right|_{D} \Big)\nonumber \\
&+\nu^{1+\kappa}\sum_{|\beta_j|\leq1} \Big(\left|w_{\iota_0-1}f\right|_{L^2}\left|e^{{q}\langle v\rangle}\partial^{\beta_j}_{v_j} f\right|_{D}+\left|w_{\iota_0-1}\partial^{\beta_j}_{v_j}f\right|_{L^2}\left|e^{{q}\langle v\rangle} f\right|_{D} \nonumber\\
&+\left|w_{\iota_0-1}f\right|_{D}\left|e^{{q}\langle v\rangle}\partial^{\beta_j}_{v_j} f\right|_{L^2}+\left|w_{\iota_0-1}\partial^{\beta_j}_{v_j}f\right|_{D}\left|\langle v\rangle^{1-\iota_0}w_{\iota_0-1} f\right|_{L^2},\left|w_{\iota_0-1} \partial_{x_j}f\right|_{D} \Big) \nonumber\\
\lesssim &\nu^{1+\kappa} \sum_{|\alpha_j|\leq1} \left|w_{\iota_0-1} \partial_{v_j}f\right|_{D}\nonumber\\
&\times \Big[\|w_{\iota_0-1} \partial^{\alpha_j}_{x_j}f\|\left(\nu^{\kappa-1}\widetilde{\mathbb{D}}_{N_{\max}-2}^{(q,\iota_0)}[f](t)\right)^{1/2}+
\|w_{\iota_0-1} \partial^{\alpha_j}_{x_j}f\|_D
\left(\widetilde{\mathbb{E}}_{N_{\max}-2}^{(q,\iota_0)}[f](t)\right)^{1/2}\Big]\nonumber\\
&+\nu \sum_{|\beta_j|\leq1}\nu^{|\beta_j|\kappa} \left|w_{\iota_0-1} \partial_{x_j}f\right|_{D}\nonumber\\
&\times\Big[\|w_{\iota_0-1} \partial^{\beta_j}_{v_j}f\|\left(\nu^{\kappa-1}\widetilde{\mathbb{D}}_{N_{\max}-2}^{(q,\iota_0)}[f](t)\right)^{1/2}+
\|w_{\iota_0-1} \partial^{\beta_j}_{v_j}f\|_D
\left(\widetilde{\mathbb{E}}_{N_{\max}-2}^{(q,\iota_0)}[f](t)\right)^{1/2}\Big] \nonumber
\\
&\lesssim\nu^{\kappa} \left(\widetilde{\mathbb{E}}_{0}^{(q,\iota_0)}[f](t)
\widetilde{\mathbb{D}}_{N_{\max}-2}^{(q,\iota_0)}[f](t)
\widetilde{\mathbb{D}}_{0}^{(q,\iota_0)}[f](t)\right)^{1/2}.
\end{align}

Therefore, we insert the above estimates into \eqref{xv-1} to conclude that
\begin{eqnarray}
	&&\frac{\mathrm{d}}{\mathrm{d}t}\int_{\mathbb{T}^3\times\mathbb{R}^3}\nu^\kappa e^\phi w^2_{\iota_0-1}\partial_{x_j} f\partial_{v_j} f \,\mathrm{d}v\mathrm{d}x+\int_{\mathbb{T}^3\times\mathbb{R}^3}\nu^\kappa e^\phi w^2_{\iota_0-1}(\partial_{x_j}f)^2\,\mathrm{d}v \mathrm{d}x\nonumber\\
	&\lesssim&\langle t\rangle^{-2}\mathcal{M}^{\frac{1}{2}}\nu^{\kappa}
\widetilde{\mathbb{E}}^{(q,\iota_0)}_0[f](t) +\|\rho_{\neq0}\|\min\left\{\left(\widetilde{\mathbb{E}}^{(q,\iota_0)}_0[f](t)\right)^{1/2},
    \left(\widetilde{\mathbb{D}}^{(q,\iota_0)}_0[f](t)\right)^{1/2} \right\}\nonumber\\	&&+A_0^{-\frac{s}{2}}\nu^{\kappa}\widetilde{\mathbb{D}}^{(q,\iota_0)}_{0}[f](t)+\nu^{\kappa} \left(\widetilde{\mathbb{E}}_{0}^{(q,\iota_0)}[f](t)
\widetilde{\mathbb{D}}_{N_{\max}-2}^{(q,\iota_0)}[f](t)
\widetilde{\mathbb{D}}_{0}^{(q,\iota_0)}[f](t)\right)^{1/2}.\nonumber
\end{eqnarray}
The derivation of \eqref{Enhanced-eqn}  for the general case $\partial^\alpha_x\partial^\beta_v Y^\omega f$ with $|\alpha|,|\beta|,|\omega|\neq0$ can be carried out in a similar manner. For brevity, we only estimate the following term:
\begin{align*}
  \mathcal{I}_4= & \int_{\mathbb{T}^3\times\mathbb{R}^3}
e^{\phi}\nu^{2\kappa|\beta|+\kappa}w^2_{\iota_0-N-1}[\partial_tf+ v  \cdot\nabla_x,\partial_{x_j}
\partial^\alpha_x\partial^\beta_v Y^\omega] f\partial_{v_j}\partial^\alpha_x\partial^\beta_v Y^\omega f \,\mathrm{d}v\mathrm{d}x\\
=& -\sum_{|\beta'|= |\beta|-1
,\atop
|\beta_1|= 1}\int_{\mathbb{T}^3\times\mathbb{R}^3}
e^{\phi}\nu^{2\kappa|\beta|+\kappa}w^2_{\iota_0-N-1}\partial_{x_j}
\partial^{\alpha+\beta_1}_x\partial^{\beta'}_v Y^\omega f\partial_{v_j}\partial^\alpha_x\partial^\beta_v Y^\omega f \,\mathrm{d}v\mathrm{d}x,
\end{align*}
whose treatment differs from the Vlasov-Poisson-Landau system case in \cite{CLN-JAMS-2023}. Note that, as in \eqref{nabla-Diss},
\begin{align*}
  \nu^{2\kappa}\|w_{\iota_0-N-1}\partial_{v_j}\partial^\alpha_x\partial^\beta_v Y^\omega f\|^2 \lesssim &A_0^{-s}\left(
  A_0\nu^{1-\kappa}\|w_{\iota_0-N} \partial^\alpha_x\partial^\beta_v Y^\omega  f\|_D^2+\nu^{1+\kappa}\| w_{\iota_0-N-1}\nabla_v \partial^\alpha_x\partial^\beta_v Y^\omega f\|^2_D\right).
\end{align*}
It then follows that
 \begin{align*}
   |\mathcal{I}_4|\lesssim &\nu^{2\kappa|\beta|+\kappa}\sum_{|\beta'|= |\beta|-1
,\atop
|\beta_1|= 1}\left\|w_{\iota_0-N-1}\partial_{x_j}
\partial^{\alpha+\beta_1}_x\partial^{\beta'}_v Y^\omega f\right\|\left\|w_{\iota_0-N-1}\partial_{v_j}\partial^\alpha_x\partial^\beta_v Y^\omega f\right\|
\\
\lesssim &A_0^{-\frac{s}{2}}\nu^{2\kappa} \left(\widetilde{\mathbb{D}}_{N}^{(q,\iota_0)}[f](t)\right)^{1/2}
\left(\nu^{-2\kappa}\widetilde{\mathbb{D}}_{N}^{(q,\iota_0)}[f](t)\right)^{1/2}
\lesssim A_0^{-\frac{s}{2}}\nu^{\kappa} \widetilde{\mathbb{D}}_{N}^{(q,\iota_0)}[f](t).
 \end{align*}
This completes the proof of Lemma \ref{elec-potenfield}.

\end{proof}

\subsection{Weighted energy estimates}
In this subsection, we derive weighted energy estimates for the unique smooth solution $ f(t,x,v) $ to the Cauchy problem \eqref{f} and \eqref{f-initial} on the time interval $ [0, T]$ for any $T>0$, under the {\it a priori} assumption \eqref{priori-assu}.

As a preliminary step, we first establish the no-weight energy estimates for $Y^{\omega}f(t,x,v)$ with $|\omega|\leq N_{max}$.

\begin{lemma}\label{lemma-nonweight-Y}
Suppose the same conditions in Theorem \ref{dens-esti} hold, there exists a proper positive constant $\delta$ such that
\begin{eqnarray}\label{basic-l2}
&&\frac{1}{2}\frac{\mathrm{d}}{\mathrm{d}t}\left\|f\right\|^2+\delta\nu\left\|{\bf P}^{\bot}f\right\|^2_{D}\nonumber\\
&\lesssim& \left\|\rho_{\neq0}\right\|\|\langle v\rangle^{\frac{\gamma}{2s}}f\|+\langle t\rangle^{-2}\mathcal{M}\nu^{\kappa}\left(\mathcal{M}\nu^{1+2\kappa}+\|f\|^2\right)+
\mathcal{M}^{\frac{1}{2}}\nu^{1+\kappa}\left(\left\|\nabla_xf\right\|^2+\left\|{\bf P}^{\bot}f\right\|^2_{D}\right),
\end{eqnarray} 
and for $1\leq |\omega|\leq N$,
    \begin{eqnarray}\label{G-omega-noweight-end}
&&\frac{\mathrm{d}}{\mathrm{d}t}\left\|Y^\omega f(t)\right\|^2+\delta\nu\left\|{\bf P}^{\bot} Y^\omega f\right\|^2_{D}\nonumber\\
&\lesssim& 
\nu^{\kappa}\left[A^{-1}_0\eta\widetilde{\mathbb{D}}^{(q,\iota_0)}_N[f](t)
   +\left(A^{-1}_0C({\eta})+A^{-\frac{s}{2}}_0\right)\left(\widetilde{\mathbb{D}}^{(q,\iota_0)}_{N-1}[f](t)
   +\widetilde{\mathbb{D}}^{(q,\iota_0)}_0[f](t)\right)\right]\nonumber\\
   &&+\mathcal{M}^{\frac12}\langle t\rangle^{-2}\widetilde{\mathbb{E}}^{(q,\iota_0)}_N[f](t)+C(\eta) \left(\mathcal{M}^2\nu^{1+2\kappa}\langle t\rangle^{-4}+\|{\bf P}^{\bot}f\|^2_D\right)\nonumber\\
	&&+A_0^{-\frac{1}{2}}\sum_{|\alpha|+|\omega|\leq N}\|\partial^\alpha_xY^\omega \rho_{\neq0}\| \min\left\{\left(\widetilde{\mathbb{E}}^{(q,\iota_0)}_N[f](t)\right)^{1/2},
\left(\widetilde{\mathbb{D}}^{(q,\iota_0)}_N[f](t)\right)^{1/2}\right\}\nonumber\\
           &&+\nu^{\kappa} \left(\widetilde{\mathbb{E}}_{N}^{(q,\iota_0)}[f](t)
\widetilde{\mathbb{D}}_{N_{\max}-2}^{(q,\iota_0)}[f](t)
\widetilde{\mathbb{D}}_{N}^{(q,\iota_0)}[f](t)\right)^{1/2}.
\end{eqnarray}
\end{lemma}
\begin{proof} We first prove \eqref{basic-l2}. To this aim,  we first denote $\bar{h}=\int_{\mathbb{T}^3}h\,\mathrm{d}x$ for any proper function $h$. Note that
\begin{align*}
   &\left|\left(\nabla_x\phi\cdot v\mu^{\frac{1}{2}}, e^{\phi}f \right)\right|\lesssim \|\nabla_x\phi\|\|\langle v\rangle^{\frac{\gamma}{2s}}f\|\lesssim \left\|\rho_{\neq0}\right\|\|\langle v\rangle^{\frac{\gamma}{2s}}f\|, \quad \mbox{and}
\end{align*}
\begin{align}\label{abc}
   & \bar{a}=\bar{b}_j(1\leq j\leq3)=0, \quad |\bar{c}|=\frac{1}{6}\|\nabla_x\phi\|^2\lesssim \langle t\rangle^{-4} \mathcal{M}^2\nu^{2\kappa}
\end{align}
by initial conditions of the conservation laws \eqref{cons-laws} in Theorem \ref{Main-Th.} and the estimate \eqref{coro-elec}.  
Then we further use the {\it a priori} assumption \eqref{priori-assu}, \eqref{coro-phi}, \eqref{abc}, \eqref{Lemma L_1}, and \eqref{Lemma L_3} to obtain
\begin{eqnarray*}
&&\frac{1}{2}\frac{\mathrm{d}}{\mathrm{d}t}\left\|e^{\frac{\phi}2} f\right\|^2+\delta\nu\left\|e^{\frac{\phi}{2}}{\bf P}^{\bot}f\right\|^2_{D}\nonumber\\
&\lesssim& \|\partial_t \phi\|_{L^{\infty}}\left\|e^{\frac{\phi}2} f\right\|^2+\left|\left(\nabla_x\phi\cdot v\mu^{\frac{1}{2}}, e^{\phi}f \right)\right|
+\nu\|f\|_{L^{\infty}_xL^2_v}\|f\|_D\|{\bf P}^{\bot}f\|_D\nonumber\\
&\lesssim& \langle t\rangle^{-2}\mathcal{M}\nu^{\kappa}\|f\|^2+\left\|\rho_{\neq0}\right\|\|\langle v\rangle^{\frac{\gamma}{2s}}f\|+
\mathcal{M}^{\frac{1}{2}}\nu^{1+\kappa}\left(|\bar{c}|^2+\left\|\nabla_x{\bf P}f\right\|^2+\left\|{\bf P}^{\bot}f\right\|^2_{D}\right)\nonumber\\
&\lesssim& \left\|\rho_{\neq0}\right\|\|\langle v\rangle^{\frac{\gamma}{2s}}f\|+\langle t\rangle^{-2}\mathcal{M}\nu^{\kappa}\left(\mathcal{M}\nu^{1+2\kappa}+\|f\|^2\right)+
\mathcal{M}^{\frac{1}{2}}\nu^{1+\kappa}\left(\left\|\nabla_xf\right\|^2+\left\|{\bf P}^{\bot}f\right\|^2_{D}\right).
\end{eqnarray*} 
This implies \eqref{basic-l2}.

Now we come to prove \eqref{G-omega-noweight-end}. For $1\leq|\omega|=N\leq N_{max}$, applying $Y^\omega$ into \eqref{f} and multiplying the result equality by $Y^\omega f$, one can deduce by using \eqref{LemmaL-Y-omega} that
\begin{eqnarray}\label{G-omega-noweight}
&&\frac{1}{2}\sum_{|\omega|=N}\frac{\mathrm{d}}{\mathrm{d}t}\left\|Y^\omega f(t)\right\|^2+\delta\nu\sum_{|\omega|=N}\left\|{\bf P}^{\bot} Y^\omega f\right\|^2_{D}\nonumber\\
&\lesssim& 
C_\eta\nu\sum_{|\omega'|\leq |\omega|-1}\|Y^{\omega'} f\|_{D}^2+\sum_{|\omega|=N}\left(Y^\omega(\nabla_x\phi\cdot v \mu^{\frac12}),Y^\omega f\right) \nonumber\\
&&+\sum_{|\omega|=N}\left(Y^\omega(\nabla_x\phi\cdot v f),Y^\omega f\right) +\sum_{|\omega|=N}\left(Y^\omega(\nabla_x\phi\cdot\nabla_v f),Y^\omega f\right) \nonumber\\
&&+\sum_{|\omega|=N}\nu\left(Y^\omega\Gamma(f,f),Y^\omega f\right).
\end{eqnarray}
Here we used the fact that
\[[\partial_t+v\cdot\nabla_x, Y^\omega]=0.\]
By interpolation method, the first term on the right-hand side of \eqref{G-omega-noweight} can be estimated as
\begin{align}\label{N-1omega}
C_\eta\nu\sum_{|\omega'|\leq |\omega|-1}\|Y^{\omega'} f\|_{D}^2\lesssim& C_{\tilde{\eta}}\nu\|f\|_{D}^2+\tilde{\eta}\nu\sum_{|\omega'|=N}\|Y^{\omega'} f\|_{D}^2\\
\lesssim& C_{\tilde{\eta}}\nu\|{\bf P}^{\bot}f\|_{D}^2+\tilde{\eta}\nu\sum_{|\omega|=N}\|{\bf P}^{\bot}Y^{\omega} f\|_{D}^2+
C_{\tilde{\eta}}\nu\|{\bf P}f\|^2+\tilde{\eta}\nu\sum_{|\omega|=N}\|{\bf P}Y^{\omega} f\|^2,\nonumber
\end{align}
where $\tilde{\eta}$ is a constant much smaller than $\eta$.

To estimate the last two terms in \eqref{N-1omega}, we decompose them into as $(\cdot)_{\neq0}+(\cdot)_{=0}$, respectively. Then, for the $(\cdot)_{\neq0}$ part of ${\bf P}f$ and ${\bf P}Y^{\omega} f$, we apply the Poincaré inequality to obtain
\begin{align}\label{nozero-Y}
   & C(\tilde{\eta})\nu\|\left({\bf P}f\right)_{\neq0}\|^2+\tilde{\eta}\nu\sum_{|\omega|=N}\left\|\left({\bf P}Y^{\omega} f\right)_{\neq0}\right\|^2
\lesssim C_{\tilde{\eta}}\nu\|\nabla_x{\bf P}f\|^2+\tilde{\eta}\nu\sum_{|\omega|=N}\|\nabla_x{\bf P}Y^{\omega} f\|^2.
\end{align}
For the $(\cdot)_{\neq0}$ part of ${\bf P}f$ and ${\bf P}Y^{\omega} f$, we first use
the definition of ${\bf P}f$ and \eqref{abc} to have
\begin{align*}
&(2\pi)^3\left( {\bf P} Y^{\omega} f\right)_{=0}
=
{\bf P}\left(\int_{\mathbb{T}^3}  Y^{\omega} f\,\mathrm{d}x\right)={\bf P}\left(\partial_v^{\omega}\int_{\mathbb{T}^3} f\,\mathrm{d}x\right)\\
&\quad= \int_{\mathbb{R}^3}\partial_v^{\omega}\left[\int_{\mathbb{T}^3} f\,\mathrm{d}x\right] \left(1,v, |v|^2-3\right)\mu^{\frac12}\,\mathrm{d}v\cdot\left(1,v,  \frac{|v|^2-3}{6}\right)\mu^{\frac12}\\
&\quad=(-1)^{|\omega|}\int_{\mathbb{R}^3}\left[\bar{c}+\int_{\mathbb{T}^3}{\bf P}^{\bot} f\,\mathrm{d}x\right] \partial_v^{\omega}\left[\left(1,v, |v|^2-3\right)\mu^{\frac12}\right]\,\mathrm{d}v\cdot\left(1,v, \frac{|v|^2-3}{6}\right)\mu^{\frac12}.
\end{align*}
This implies 
\begin{align}\label{zero-Y}
&C_{\tilde{\eta}}\nu\left\|\left( {\bf P} f\right)_{=0}\right\|^2+\tilde{\eta}\nu\sum_{|\omega|=N}\left\|\left({\bf P} Y^{\omega} f\right)_{=0}\right\|^2\nonumber\\
&\quad\lesssim C_{\tilde{\eta}}\nu\left(\|\bar{c}\|^2+\|{\bf P}^{\bot}f\|^2_D\right)\lesssim C_{\tilde{\eta}}\nu\left(\langle t\rangle^{-4} \mathcal{M}^2\nu^{2\kappa}+\|{\bf P}^{\bot}f\|^2_D\right).
\end{align}
We use  \eqref{nozero-Y} and \eqref{zero-Y} to obtain that
\begin{align*}
&C_{\tilde{\eta}}\nu\|{\bf P}f\|^2+\tilde{\eta}\nu\sum_{|\omega|=N}\|{\bf P}Y^{\omega} f\|^2\\
&\hspace{1cm}\lesssim C_{\tilde{\eta}}\nu\|\nabla_x{\bf P}f\|^2+\tilde{\eta}\nu\sum_{|\omega|=N}\|\nabla_x{\bf P}Y^{\omega} f\|^2+C_{\tilde{\eta}}\nu\left(\langle t\rangle^{-4} \mathcal{M}^2\nu^{2\kappa}+\|{\bf P}^{\bot}f\|^2_D\right)\\
&\hspace{1cm}\lesssim A^{-1}_0\nu\left(C_{\tilde{\eta}}\widetilde{\mathbb{D}}^{(q,\iota_0)}_N[f](t)
+\tilde{\eta}\widetilde{\mathbb{D}}^{(q,\iota_0)}_0[f](t)\right)+C_{\tilde{\eta}}\nu\left(\mathcal{M}^2\nu^{1+2\kappa}\langle t\rangle^{-4}+\|{\bf P}^{\bot}f\|^2_D\right).
\end{align*}
Combining the above estimate and \eqref{N-1omega},  the second term on the right-hand side of \eqref{G-omega-noweight} can be bounded by 
\begin{align*}
\tilde{\eta}\nu\sum_{|\omega|=N}\|{\bf P}^{\bot}Y^{\omega} f\|_{D}^2+C({\tilde{\eta}})\nu\left(\mathcal{M}^2\nu^{1+2\kappa}\langle t\rangle^{-4}+\|{\bf P}^{\bot}f\|^2_D\right)+A^{-1}_0\nu\left(C({\tilde{\eta}})\widetilde{\mathbb{D}}^{(q,\iota_0)}_N[f](t)
+\tilde{\eta}\widetilde{\mathbb{D}}^{(q,\iota_0)}_0[f](t)\right).
\end{align*}

For the  last three terms in \eqref{G-omega-noweight}, since their estimation is much easier than that for corresponding terms in \eqref{G-alpha-beta}, we omit the details here and give the following estimate
\begin{align*}
&\sum_{|\omega|=N}\left(Y^\omega(\nabla_x\phi\cdot v f),Y^\omega f\right)+\sum_{|\omega|=N}\left(Y^\omega(\nabla_x\phi\cdot\nabla_v f),Y^\omega f\right)  \nonumber\\ &
+\nu\sum_{|\omega|=N}\left(Y^\omega\Gamma(f,f),Y^\omega f\right)\nonumber\\
\lesssim
   &\mathcal{M}^{\frac12}\langle t\rangle^{-2}\widetilde{\mathbb{E}}^{(q,\iota_0)}_N[f](t)+C(\eta) \mathcal{M}^2\nu^{1+2\kappa}\langle t\rangle^{-4}\nonumber\\
	&+A_0^{-\frac{1}{2}}\sum_{|\alpha|+|\omega|\leq N}\|\partial^\alpha_xY^\omega \rho_{\neq0}\| \min\left\{\left(\widetilde{\mathbb{E}}^{(q,\iota_0)}_N[f](t)\right)^{1/2},
\left(\widetilde{\mathbb{D}}^{(q,\iota_0)}_N[f](t)\right)^{1/2}\right\}\nonumber\\
            &+\nu^{\kappa} \left(\widetilde{\mathbb{E}}_{N}^{(q,\iota_0)}[f](t)
\widetilde{\mathbb{D}}_{N_{\max}-2}^{(q,\iota_0)}[f](t)
\widetilde{\mathbb{D}}_{N}^{(q,\iota_0)}[f](t)\right)^{1/2}.
\end{align*}
Consequently, we denote $\tilde{\eta}$ as new $\eta$, which is sufficiently small, and collect the estimates above in \eqref{G-omega-noweight} to have
\eqref{G-omega-noweight-end}.
\end{proof}

The following theorem provides the total {\it a priori} energy estimates for the classical solution of \eqref{f} and \eqref{f-initial}.

\begin{theorem}\label{Th-basic-estimate} Suppose the same conditions in Theorem \ref{dens-esti} hold. For $ t\in [0, T_0]$, the solution $ f(t,x,v) $ to the Cauchy problem \eqref{f}-\eqref{f-initial} satisfies
	\begin{eqnarray}\label{priori-ener}
		\sup_{ 0\leq t\leq T}\widetilde{\mathbb{E}}^{(q,\iota_0)}_N[f](t)
		+\nu^\kappa\int_{0}^{T}\widetilde{\mathbb{D}}^{(q,\iota_0)}_N[f](t)dt\lesssim \mathcal{M}^2\nu^{2\kappa}\min\{\nu^{-\kappa},\langle t\rangle\}^{\max\{0,N-N_{max}+2\}}.
	\end{eqnarray}
\end{theorem}

\begin{proof} 
According to the definition of $\widetilde{\mathbb{E}}^{(q,\iota_0)}_N[f](t)$ in \eqref{Energy-N}, we will estimate the following norms $$\sum_{|\alpha'|\leq1,
|\beta'|\leq2}\left\|\nu^{\kappa(|\beta|+|\beta'|)}w_{\iota_0-N-|\alpha'|-|\beta'|}\partial^{\alpha+\alpha'}_x\partial^{\beta+\beta'}_vY^\omega f\right\|^2,\quad |\alpha|+|\beta|+|\omega|=N\leq N_{max}.$$

Since the whole proof is quite long, we divide it into three steps.

\noindent\underline{{\it Step 1. Estimation of $\left\|\nu^{\kappa|\beta|}e^{\frac{\phi}2}w_{\iota_0-N}\partial^\alpha_x\partial^\beta_vY^\omega f(t)\right\|^2$:}} 
Applying $\partial^\alpha_x\partial^\beta_vY^\omega$ to \eqref{f} and multiplying the resulting equation by $e^{\phi}\nu^{2\kappa|\beta|}w^2_{\iota_0-N}\partial^\alpha_x\partial^\beta_vY^\omega f$ gives
\begin{eqnarray}\label{G-alpha-beta}
&&\frac{1}{2}\frac{\mathrm{d}}{\mathrm{d}t}\left\|\nu^{\kappa|\beta|}e^{\frac{\phi}2}w_{\iota_0-N}\partial^\alpha_x\partial^\beta_vY^\omega f(t)\right\|^2+\nu\left\|\nu^{\kappa|\beta|}e^{\frac{\phi}2}w_{\iota_0-N}\partial^\alpha_x\partial^\beta_vY^\omega f\right\|^2_{D}\nonumber\\
&\lesssim& \Big(\eta \nu\sum_{|\beta'|\leq|\beta|,|\bar{\omega}|\leq|\omega|}
+C(\eta) \nu\sum_{|\beta'|\leq|\beta|,|\bar{\omega}|\leq|\omega|\atop
|\beta'|+|\bar{\omega}|\leq|\beta|+|\omega|-1}\Big)
\left\|\nu^{\kappa|\beta|}e^{\frac{\phi}2}w_{\iota_0-N}\partial^\alpha_x\partial^{\beta'}_v
Y^{\bar{\omega}} f\right\|^2_{D}\nonumber\\
&&+C(\eta) \nu\sum_{|\bar{\omega}|\leq|\omega|}\left\|\chi_{\{| v|\leq2C_{\eta}\}}\partial^\alpha_x Y^{\bar{\omega}} f\right\|^2+\|\partial_t\phi\|_{L^\infty_x}\left\|\nu^{\kappa|\beta|}e^{\frac{\phi}2}w_{\iota_0-N}\partial^\alpha_x\partial^\beta_vY^\omega f\right\|^2\nonumber\\
&&+\left(\partial^\alpha_x\partial^\beta_vY^\omega(\nabla_x\phi\cdot v \mu^{\frac12}),e^{\phi}\nu^{2\kappa|\beta|}w^2_{\iota_0-N}\partial^\alpha_x\partial^\beta_vY^\omega f\right) \nonumber\\
&&+\sum_{|\beta'|=1}\left(\partial^{\alpha+\beta'}_x\partial^{\beta-\beta'}_vY^\omega f, e^{\phi}\nu^{2\kappa|\beta|}
w^2_{\iota_0-N}\partial^\alpha_x\partial^\beta_vY^\omega f\right)\nonumber\\
&&
+\sum_{1\leq|\alpha_1|+|\omega_1|\leq |\alpha|+|\beta|}\left(v_i\partial^{\alpha_1+e_i}_x Y^{\omega_1}\phi
\partial^{\alpha-\alpha_1}_x\partial^{\beta}_vY^{\omega-\omega_1}f,e^{\phi}\nu^{2\kappa|\beta|}w^2_{\iota_0-N}\partial^\alpha_x\partial^\beta_vY^\omega f\right)\nonumber\\
&&
+\sum_{1\leq|\alpha_1|+|\omega_1|\leq |\alpha|+|\omega|}\left(\partial^{\alpha_1+e_i}_x Y^{\omega_1}\phi
	\partial^{\alpha-\alpha_1}_x\partial^{\beta-e_i}_vY^{\omega-\omega_1}f,
e^{\phi}\nu^{2\kappa|\beta|}w^2_{\iota_0-N}\partial^\alpha_x\partial^\beta_vY^\omega f\right)\nonumber\\
&&
+\sum_{|\alpha_1|+|\omega_1|\leq |\alpha|+|\omega|}\left(\partial^{\alpha_1+e_i}_x Y^{\omega_1}\phi
	\partial^{\alpha-\alpha_1}_x\partial^{\beta+e_i}_vY^{\omega-\omega_1}f,
e^{\phi}\nu^{2\kappa|\beta|}w^2_{\iota_0-N}\partial^\alpha_x\partial^\beta_vY^\omega f\right)\nonumber\\
&&+\nu\left(\partial^\alpha_x\partial^\beta_vY^\omega\Gamma(f,f),e^{\phi}\nu^{2\kappa|\beta|}
w^2_{\iota_0-N}\partial^\alpha_x\partial^\beta_vY^\omega f\right).
\end{eqnarray}
Here we used \eqref{wL-lin} for arbitrary $\eta>0$, the cancellation technique in \cite{Guo-JAMS-2012}:
$$-\frac{1}{2}v\cdot\nabla_x\left(e^{\phi}w^2_{\iota_0-N}\partial^\alpha_x\partial^\beta_vY^\omega f\nu^{2\kappa|\beta|}\right)+\frac{1}{2}v\cdot\nabla_x\phi e^{\phi}w^2_{\iota_0-N}\partial^\alpha_x\partial^\beta_vY^\omega f\nu^{2\kappa|\beta|}=0,$$
and the following fact due to commuting property $[\partial_t+v\cdot\nabla_x,Y^\omega]=0$: 
\begin{eqnarray*}
[\partial^\alpha_x\partial^\beta_vY^\omega,\partial_t+v\cdot\nabla_x]=\sum_{ |\beta'|=1}\partial^{\alpha+\beta'}_x\partial^{\beta-\beta'}_vY^\omega.
\end{eqnarray*}
Now we estimate the terms in \eqref{G-alpha-beta} one by one.

The first term in the right hand side of \eqref{G-alpha-beta} can be dominated by
\begin{align*}
   & A^{-1}_0\nu^{\kappa}\left(\eta\widetilde{\mathbb{D}}^{(q,\iota_0)}_N[f](t)
   +C_{\eta}\widetilde{\mathbb{D}}^{(q,\iota_0)}_{N-1}[f](t)\right).
\end{align*}
For the second term, we decompose it as $(\cdot)_{\neq0}+(\cdot)_{=0}$ two parts.
For the part $(\cdot)_{\neq0}$, we apply the Poincaré inequality to obtain
\begin{align*}
   & C(\eta) \nu\sum_{|\bar{\omega}|\leq|\omega|}\left\|\left(\chi_{\{| v|\leq2C_{\eta}\}}\partial^\alpha_x Y^{\bar{\omega}} f\right)_{\neq0}\right\|^2\nonumber\\
   &\qquad\lesssim C(\eta) \nu\sum_{|\bar{\omega}|\leq|\omega|}\left\|\chi_{\{| v|\leq2C_{\eta}\}}\nabla_x\partial^\alpha_x Y^{\bar{\omega}} f\right\|^2\lesssim A^{-1}_0C_{\eta}\nu\widetilde{\mathbb{D}}^{(q,\iota_0)}_N[f](t).
\end{align*}
For the other part $(\cdot)_{=0}$, it doesn't vanish only when $|\alpha|=0$. Then, as the derivation of \eqref{zero-Y}, we can obtain 
\begin{align*}
   & C(\eta) \nu\sum_{|\bar{\omega}|\leq|\omega|}\left\|\left(\chi_{\{| v|\leq2C_{\eta}\}} Y^{\bar{\omega}} f\right)_{=0}\right\|^2\nonumber\\
      &\qquad  \leq C(\eta) \nu\sum_{|\bar{\omega}|\leq|\omega|}\left\|{\bf P}^{\bot} \left(Y^{\bar{\omega}}f\right)\right\|^2_D+C(\eta) \nu\left(\|\bar{c}\|^2+\|{\bf P}^{\bot}f\|^2_D\right)\nonumber\\
    &\qquad
   \leq C({\eta})\nu \sum_{|\bar{\omega}|\leq|\omega|}\left\|{\bf P}^{\bot} \left(Y^{\bar{\omega}}f\right)\right\|^2_D+C(\eta) \mathcal{M}^2\nu^{1+2\kappa}\langle t\rangle^{-4}.\nonumber
\end{align*}
Therefore, the upper bound of the second term is
\begin{align*}
A^{-1}_0C({\eta})\nu\widetilde{\mathbb{D}}^{(q,\iota_0)}_{N}[f](t)+C({\eta})\nu \sum_{|\bar{\omega}|\leq|\omega|}\left\|{\bf P}^{\bot} \left(Y^{\bar{\omega}}f\right)\right\|^2_D+C(\eta) \mathcal{M}^2\nu^{1+2\kappa}\langle t\rangle^{-4}.
\end{align*}
Due to \eqref{coro-phi} and the {\it a priori} assumption \eqref{priori-assu}, the upper bound of the third term is 
\begin{align*}
    \mathcal{M}^{\frac12}\langle t\rangle^{-2}\left\|\nu^{\kappa|\beta|}e^{\frac{\phi}2}w_{\iota_0-N}\partial^\alpha_x\partial^\beta_vY^\omega f\right\|^2\lesssim \mathcal{M}^{\frac12}\langle t\rangle^{-2}\widetilde{\mathbb{E}}^{(q,\iota_0)}_N[f](t).
\end{align*}
For the fourth term, we integrate by parts with respect to $v$ and apply the Poincaré inequality to bound it by
\begin{align*}
\|\partial^\alpha_xY^\omega \nabla_x\phi\|\|\partial^\alpha_xY^\omega f\|\lesssim
A_0^{-\frac{1}{2}}\sum_{|\alpha|+|\omega|\leq N}\|\partial^\alpha_xY^\omega \rho_{\neq0}\| \left(\widetilde{\mathbb{E}}^{(q,\iota_0)}_N[f](t)\right)^{1/2}.
\end{align*}
On the other hand, we can also integrate by parts with respect to $v,x$ and bound it with the $\widetilde{\mathbb{D}}^{(q,\iota_0)}_N[f](t)$ norm as
\begin{align*}
   & \left(\partial^\alpha_x\partial^\beta_vY^\omega(\nabla_x\phi\cdot v \mu^{\frac12}),\nu^{2\kappa|\beta|}e^{\phi}w^2_{\iota_0-N}\partial^\alpha_x\partial^\beta_vY^\omega f\right) \\
   =&(-1)^{|\beta|+1}\left(\partial^\alpha_x Y^{\omega}\phi , \nu^{2\kappa|\beta|}e^{\phi}\partial^{2\beta}_v( v w^2_{\iota_0-N}\mu^{\frac12})\cdot\nabla_x\partial^\alpha_xY^\omega f\right)\\
   &+(-1)^{|\beta|+1}\left(\partial^\alpha_x Y^{\omega}\phi ,\nabla_x\phi\cdot \partial^{2\beta}_v(v w^2_{\iota_0-N}\mu^{\frac12})\nu^{2\kappa|\beta|}e^{\phi}\partial^\alpha_xY^\omega f\right)\\
   \lesssim& \|\partial^\alpha_x Y^{\omega}\phi\| \left(\|\nabla_x\partial^\alpha_xY^\omega f\|+ \|\nabla_x\phi\|_{\infty} \|\partial^\alpha_xY^\omega f\|\right)\\
   \lesssim& A_0^{-\frac{1}{2}}\sum_{|\alpha|+|\omega|\leq N}\|\partial^\alpha_xY^\omega \rho_{\neq0}\| \left(\widetilde{\mathbb{D}}^{(q,\iota_0)}_N[f](t)\right)^{1/2}+\langle t\rangle^{-2} \mathcal{M}\nu^{\kappa}\widetilde{\mathbb{E}}^{(q,\iota_0)}_N[f](t).
\end{align*}
Then we have the following upper bound for the fourth term:
\begin{align*}
A_0^{-\frac{1}{2}}\sum_{|\alpha|+|\omega|\leq N}\|\partial^\alpha_xY^\omega \rho_{\neq0}\| \min\left\{\left(\widetilde{\mathbb{E}}^{(q,\iota_0)}_N[f](t)\right)^{1/2},
\left(\widetilde{\mathbb{D}}^{(q,\iota_0)}_N[f](t)\right)^{1/2}\right\}+\langle t\rangle^{-2} \mathcal{M}\nu^{\kappa}\widetilde{\mathbb{E}}^{(q,\iota_0)}_N[f](t).
\end{align*}

For the fifth term, noting that
\begin{align}\label{beta-DD}
  &\left\|\nu^{\kappa|\beta|}w_{\iota_0-N}\partial^{\beta'}_{v}
  \partial^\alpha_x\partial^{\beta-\beta'}_v Y^\omega f\right\|^2\lesssim A_0^{-s}\Big(
  A_0
  \nu^{1-\kappa}\left\|\nu^{2\kappa(|\beta|-1)}w_{\iota_0-N+1} \partial^\alpha_x\partial^{\beta-\beta'}_v Y^\omega  f\right\|_D^2\nonumber\\
  &\qquad +\nu^{1+\kappa}\left\| \nu^{\kappa(|\beta|-1)}w_{\iota_0-N}\nabla_v \partial^\alpha_x\partial^{\beta-\beta'}_v Y^\omega f\right\|^2_D\Big)
\end{align}
for $|\beta'|=1$,
we have
\begin{eqnarray}	&&\sum_{|\beta'|=1}\left(w_{\iota_0-N}\partial^{\alpha+\beta'}_x\partial^{\beta-\beta'}_vY^\omega f, e^{\phi}\nu^{2\kappa|\beta|}
w^2_{\iota_0-N}\partial^\alpha_x\partial^\beta_vY^\omega f\right)\nonumber\\
	&=&\nu^\kappa\sum_{|\beta'|=1}\left(\nu^{\kappa(|\beta|-1)}w_{\iota_0-N}\partial^{\alpha+\beta'}_x
\partial^{\beta-\beta'}_vY^\omega f, e^{\phi}\nu^{\kappa|\beta|}
w_{\iota_0-N}\partial^\alpha_x\partial^\beta_vY^\omega f\right)\nonumber\\
	&\lesssim&\nu^{\kappa}\sum_{|\beta'|=1}\left\|\nu^{\kappa(|\beta|-1)}w_{\iota_0-N}\partial^{\alpha+\beta'}_x
\partial^{\beta-\beta'}_vY^\omega f\right\| \left\|\nu^{\kappa|\beta|}w_{\iota_0-N}\partial^{\beta'}_{v}\partial^\alpha_x
\partial^{\beta-\beta'}_v Y^\omega f\right\|  \nonumber\\
		&\lesssim& A_0^{-\frac{s}{2}}\nu^{\kappa} \widetilde{\mathbb{D}}^{(q,\iota_0)}_{N-1}[f](t).\nonumber
\end{eqnarray}

For the eighth term, its estimation is most difficult due to the appearance of one order extra velocity derivative, especially when $|\alpha_1|+|\omega_1|$ is large. In fact, in \cite{CLN-JAMS-2023}, this term is controlled by $\langle t\rangle^{-2} \mathcal{M}^{1/2} \widetilde{\mathbb{E}}^{(q,\iota_0)}_N[f](t)$ with the help of the time decay of $\partial^{\alpha_1+e_i}_x Y^{\omega_1}\phi$ when $|\alpha_1|+|\omega_1|\leq 4$. While for the case $|\alpha_1|+|\omega_1|\geq 5$, thanks to the $H^1_v$ dissipation of the linear Landau collision operator, this term was dominated by 
$$\sum_{|\alpha|+|\omega|\leq N}\|\partial^\alpha_xY^\omega \rho_{\neq0}\| \min\left\{\left(\widetilde{\mathbb{E}}^{(q,\iota_0)}_N[f](t)\right)^{1/2},
\left(\widetilde{\mathbb{D}}^{(q,\iota_0)}_N[f](t)\right)^{1/2}\right\}$$ via integration by parts and thus absorbing the extra velocity derivative by the dissipation norm $\widetilde{\mathbb{D}}^{(q,\iota_0)}_N[f](t)$. Due to the $H^s_v$ dissipation of the linear Boltzmann collision operator, this method is not applicable to our situation, especially for the corresponding estimate of $\partial^\alpha_x\partial^{\beta+\beta'}_vY^\omega$ with $|\beta'|=2$ in Step 2. 

To overcome this difficulty, instead of the above method in \cite{CLN-JAMS-2023}, we  will make use of the accurate time decay of $\nabla_v\phi$ in \eqref{coro-elec} and  the good time decay of norms $\widetilde{\mathbb{E}}^{(q,\iota_0)}_{\bar{N}}[f](t)$ for $\bar{N}$ small enough by the {\it a priori} assumption \eqref{priori-assu} when $|\alpha_1|+|\omega_1|\leq N_{max}-2$ and $|\alpha_1|+|\omega_1|\geq N_{max}-1$, respectively. More explicitly, we can use \eqref{priori-assu},  \eqref{Th-rho-1}, and \eqref{coro-elec} to dominate it by
\begin{eqnarray}\label{new1}
&&\sum_{|\alpha_1|+|\omega_1|\leq N_{max}-2}\|\partial^{\alpha_1+e_i}_x Y^{\omega_1}\phi\|_{\infty}\|\nu^{\kappa|\beta|}e^{\frac{\phi}{2}}
w_{\iota_0-N}\partial^{\alpha-\alpha_1}_x\partial^{\beta+e_i}_vY^{\omega-\omega_1}f\| \|\nu^{\kappa|\beta|}e^{\frac{\phi}{2}}w_{\iota_0-N}\partial^\alpha_x\partial^\beta_vY^\omega f\|\nonumber\\
&& + \sum_{|\alpha_1|+|\omega_1|\geq N_{max}-1}\|\partial^{\alpha_1+e_i}_x Y^{\omega_1}\phi\|\nonumber\\
&&\times\left\|\nu^{\kappa|\beta|}e^{\frac{\phi}{2}}
w_{\iota_0-N}
\partial^{\alpha-\alpha_1}_x\partial^{\beta+e_i}_vY^{\omega-\omega_1}f\right\|_{L^{\infty}_xL^2_v} \left\|\nu^{\kappa|\beta|}e^{\frac{\phi}{2}}w_{\iota_0-N}
\partial^\alpha_x\partial^{\beta}_vY^\omega f\right\|\nonumber\\
&\lesssim& \langle t\rangle^{-2} \mathcal{M} \widetilde{\mathbb{E}}^{(q,\iota_0)}_N[f](t)
 + \sum_{|\alpha_1|+|\omega_1|\geq N_{max}-1}\|\partial^{\alpha_1}_x Y^{\omega_1}\rho_{\neq0}\|_{L^3_x}\\
&&\times\sum_{|\alpha_2|+|\beta|+|\omega_2|\leq1}\left\|\nu^{\kappa|\beta|}e^{\frac{\phi}{2}}
w_{\iota_0-|\alpha_2|-|\beta|-|\omega_2|-2}
\partial^{\alpha_2}_x\partial^{\beta+e_i}_vY^{\omega_2}f\right\|_{L^6_xL^2_v} \left\|\nu^{\kappa|\beta|}e^{\frac{\phi}{2}}w_{\iota_0-N}
\partial^\alpha_x\partial^{\beta}_vY^\omega f\right\|\nonumber\\
&\lesssim& \langle t\rangle^{-2} \mathcal{M} \widetilde{\mathbb{E}}^{(q,\iota_0)}_N[f](t)+\mathcal{M}\nu^{\kappa} \sum_{|\alpha_2|+|\beta|+|\omega_2|\leq1\atop
|\alpha'|=|\beta'|=1}\left\|\nu^{\kappa|\beta|}e^{\frac{\phi}{2}}
w_{\iota_0-N}
\partial^{\alpha_2+\alpha'}_x\partial^{\beta+\beta'}_vY^{\omega_2}f\right\| \left\|\nu^{\kappa|\beta|}e^{\frac{\phi}{2}}w_{\iota_0-N}
\partial^\alpha_x\partial^{\beta}_vY^\omega f\right\|\nonumber\\
&\lesssim& \langle t\rangle^{-2} \mathcal{M} \widetilde{\mathbb{E}}^{(q,\iota_0)}_N[f](t)+\langle t\rangle^{-4} \mathcal{M}^{\frac{1}{2}}\left(\widetilde{\mathbb{E}}^{(q,\iota_0)}_N[f](t)
+\mathcal{M}^2\nu^{2\kappa}\right)\nonumber\\
&\lesssim& \langle t\rangle^{-2}  \left(\mathcal{M}^{\frac{1}{2}} \widetilde{\mathbb{E}}^{(q,\iota_0)}_N[f](t)+\mathcal{M}^2\nu^{2\kappa}\right).\nonumber
\end{eqnarray}
The sixth term and seventh term can be estimated in the similar and simpler way and have the same upper bound as the eighth term.

As in the estimation of \eqref{enhan-00}, the last term in \eqref{G-alpha-beta} is bounded by
\begin{align*}
&\nu\sum_{\alpha_1+\alpha_2=\alpha\atop\beta_1+\beta_2\leq\beta,\omega_1+\omega_2\leq \omega }\left[
\left(\left|\nu^{\kappa|\beta_1|}e^{\frac{\phi}{2}}w_{\iota_0-N}
\partial^{\alpha_1}_x\partial^{\beta_1}_vY^{\omega_1} f\right|_{L^2}\left|\nu^{\kappa|\beta_2|}e^{q\langle v\rangle}
\partial^{\alpha_2}_x\partial^{\beta_2}_vY^{\omega_2} f\right|_{D}\right.\right.\\
&\left.\left.+\left|\nu^{\kappa|\beta_1|}e^{\frac{\phi}{2}}w_{\iota_0-N}
\partial^{\alpha_1}_x\partial^{\beta_1}_vY^{\omega_1} f\right|_{D}\left|\nu^{\kappa|\beta_2|}e^{q\langle v\rangle}
\partial^{\alpha_2}_x\partial^{\beta_2}_vY^{\omega_2} f\right|_{L^2},\left|\nu^{\kappa|\beta|}e^{\frac{\phi}{2}}w_{\iota_0-N}
\partial^\alpha_x\partial^{\beta}_vY^\omega f\right|_D\right)\right]\\
\lesssim&\nu\sum_{\alpha_1+\alpha_2=\alpha}\sum_{\beta_1+\beta_2\leq\beta,\omega_1+\omega_2\leq \omega\atop |\alpha_2|+|\beta_2|+|\omega_2|\leq N_{max}-4}
\left(\left\|\nu^{\kappa|\beta_1|}e^{\frac{\phi}{2}}w_{\iota_0-N}
\partial^{\alpha_1}_x\partial^{\beta_1}_vY^{\omega_1} f\right\|\left\||\nu^{\kappa|\beta_2|}e^{\frac{\phi}{2}}
\partial^{\alpha_2}_x\partial^{\beta_2}_vY^{\omega_2} f|_{D}\right\|_{L^{\infty}_x}\right.\\
&\left.+\left\|\nu^{\kappa|\beta_1|}e^{\frac{\phi}{2}}w_{\iota_0-N}
\partial^{\alpha_1}_x\partial^{\beta_1}_vY^{\omega_1} f\right\|_{D}\left\||\nu^{\kappa|\beta_2|}e^{\frac{\phi}{2}}e^{q\langle v\rangle}
\partial^{\alpha_2}_x\partial^{\beta_2}_vY^{\omega_2} f\right\|_{L^{\infty}_xL^2_v},\left\|\nu^{\kappa|\beta|}e^{\frac{\phi}{2}}w_{\iota_0-N}
\partial^\alpha_x\partial^{\beta}_vY^\omega f\right\|_D\right)\\
&+\nu\sum_{\alpha_1+\alpha_2=\alpha}\sum_{\beta_1+\beta_2\leq\beta,\omega_1+\omega_2\leq \omega\atop |\alpha_2|+|\beta_2|+|\omega_2|\geq N_{max}-3}
\left(\left\|\nu^{\kappa|\beta_1|}e^{\frac{\phi}{2}}w_{\iota_0-N}
\partial^{\alpha_1}_x\partial^{\beta_1}_vY^{\omega_1} f\right\|_{L^{\infty}_xL^2_v}\left\||\nu^{\kappa|\beta_2|}e^{\frac{\phi}{2}}
\partial^{\alpha_2}_x\partial^{\beta_2}_vY^{\omega_2} f\right\|_{D}\right.\\
&\left.+\left\||\nu^{\kappa|\beta_1|}e^{\frac{\phi}{2}}w_{\iota_0-N}
\partial^{\alpha_1}_x\partial^{\beta_1}_vY^{\omega_1} f|_{D}\right\|_{L^{\infty}_x}\left\||\nu^{\kappa|\beta_2|}e^{\frac{\phi}{2}}e^{q\langle v\rangle}
\partial^{\alpha_2}_x\partial^{\beta_2}_vY^{\omega_2} f\right\|,\left\|\nu^{\kappa|\beta|}e^{\frac{\phi}{2}}w_{\iota_0-N}
\partial^\alpha_x\partial^{\beta}_vY^\omega f\right\|_D\right)\\
\lesssim& \nu^{\kappa} \left[\left(\widetilde{\mathbb{E}}^{(q,\iota_0)}_N[f](t)
\widetilde{\mathbb{D}}^{(q,\iota_0)}_{N_{max}-2}[f](t)\right)^{1/2}+
\left(\widetilde{\mathbb{D}}^{(q,\iota_0)}_N[f](t)
\widetilde{\mathbb{E}}^{(q,\iota_0)}_{N_{max}-2}[f](t)\right)^{1/2}\right]
\left(\widetilde{\mathbb{D}}^{(q,\iota_0)}_N[f](t)\right)^{1/2}
\\
\lesssim& \nu^{\kappa} \left(\widetilde{\mathbb{E}}_{N}^{(q,\iota_0)}[f](t)
\widetilde{\mathbb{D}}_{N_{\max}-2}^{(q,\iota_0)}[f](t)
\widetilde{\mathbb{D}}_{N}^{(q,\iota_0)}[f](t)\right)^{1/2}. 
\end{align*}

Therefore, inserting the above estimates  into \eqref{G-alpha-beta}, we arrive at
\begin{eqnarray}\label{G-1}
	&&\frac{1}{2}\frac{\mathrm{d}}{\mathrm{d}t}\left\|\nu^{\kappa|\beta|}e^{\frac{\phi}2}w_{\iota_0-N}\partial^\alpha_x\partial^\beta_vY^\omega f\right\|^2+\nu\left\|\nu^{\kappa|\beta|}e^{\frac{\phi}2}w_{\iota_0-N}\partial^\alpha_x\partial^\beta_vY^\omega f\right\|^2_{D}\nonumber\\ 
	&\lesssim&\nu^{\kappa}\left[A^{-1}_0\eta\widetilde{\mathbb{D}}^{(q,\iota_0)}_N[f](t)
   +\left(A^{-1}_0C_{\eta}+A^{-\frac{s}{2}}_0\right)\widetilde{\mathbb{D}}^{(q,\iota_0)}_{N-1}[f](t)
   \right]\nonumber\\
   &&+\mathcal{M}^{\frac12}\langle t\rangle^{-2}\widetilde{\mathbb{E}}^{(q,\iota_0)}_N[f](t)+C_{\eta} \nu\sum_{|\bar{\omega}|\leq N}\left\|{\bf P}^{\bot} \left(Y^{\bar{\omega}}f\right)\right\|^2_D+C(\eta) \mathcal{M}^2\nu^{1+2\kappa}\langle t\rangle^{-4}\nonumber\\
	&&+A_0^{-\frac{1}{2}}\sum_{|\alpha|+|\omega|\leq N}\|\partial^\alpha_xY^\omega \rho_{\neq0}\| \min\left\{\left(\widetilde{\mathbb{E}}^{(q,\iota_0)}_N[f](t)\right)^{1/2},
\left(\widetilde{\mathbb{D}}^{(q,\iota_0)}_N[f](t)\right)^{1/2}\right\}\nonumber\\
            &&+\nu^{\kappa} \left(\widetilde{\mathbb{E}}_{N}^{(q,\iota_0)}[f](t)
\widetilde{\mathbb{D}}_{N_{\max}-2}^{(q,\iota_0)}[f](t)
\widetilde{\mathbb{D}}_{N}^{(q,\iota_0)}[f](t)\right)^{1/2}.
\end{eqnarray}

\noindent\underline{{\it Step 2. Estimation of $\sum_{|\alpha'|\leq1,|\beta'|\leq2\atop
|\alpha'|+|\beta'|\geq1}\left\|\nu^{\kappa(|\beta|+|\beta'|)}w_{\iota_0-N-|\alpha'|-|\beta'|}\partial^{\alpha+\alpha'}_x\partial^{\beta+\beta'}_vY^\omega f\right\|^2$:}} For the case with extra spatial derivatives or velocity derivatives, we can also obtain inequalities as in \eqref{G-alpha-beta}.  The estimation of terms in these inequalities are quite similar as that in Step 1 except for the transport terms, which correspond
 to  the fifth term in the right side of \eqref{G-alpha-beta}. Therefore, we only focus on these transport terms, whose estimation is quite different from that in \cite{CLN-JAMS-2023} when $|\beta'|=2$,  and omit the details for other terms.  According to the definition of the norms $\widetilde{\mathbb{D}}^{(q,\iota_0)}_N[f](t)$ for $N\leq N_{max}$, we can obtain
\begin{eqnarray*}	&&\sum_{|\alpha'|=|\beta'|=1}\left(\partial^{\alpha+\alpha'+\beta'}_x
\partial^{\beta-\beta'}_vY^\omega f, e^{\phi}\nu^{2\kappa|\beta|}
w^2_{\iota_0-N-1}\partial^{\alpha+\alpha'}_x\partial^{\beta}_vY^\omega f\right)\nonumber\\
	&=&\nu^\kappa\sum_{|\alpha'|=|\beta'|=1}\left(\nu^{\kappa(|\beta|-1)}w_{\iota_0-N-1}
\partial^{\alpha'}_x\partial^{\alpha+\beta'}_x
\partial^{\beta-\beta'}_vY^\omega f, e^{\phi}\nu^{\kappa|\beta|}
w_{\iota_0-N-1}\partial^{\beta'}_v\partial^{\alpha+\alpha'}_x\partial^{\beta-\beta'}_vY^\omega f\right)\nonumber\\
	&\lesssim&\nu^{\kappa}\sum_{|\alpha'|=|\beta'|=1}\left\|\nu^{\kappa(|\beta|-1)}w_{\iota_0-N-1}\partial^{\alpha'}_x\partial^{\alpha+\beta'}_x
\partial^{\beta-\beta'}_vY^\omega f\right\| \left\|\nu^{\kappa|\beta|}w_{\iota_0-N-1}\partial^{\beta'}_v\partial^{\alpha+\alpha'}_x
\partial^{\beta-\beta'}_v Y^\omega f\right\|  \nonumber\\
		&\lesssim& \nu^{\kappa} \sum_{|\alpha''|=|\alpha|+1\atop
|\beta''|=|\beta|-1} \left(\widetilde{\mathbb{D}}^{(q)}_{\alpha'',\beta'',\omega}[f](t) \right)^{\frac{1}{2}} \left(A_0^{-s}\widetilde{\mathbb{D}}^{(q)}_{\alpha'',\beta'',\omega}[f](t) \right)^{\frac{1}{2}}\nonumber\\
&\lesssim&   A_0^{-\frac{s}{2}} \nu^{\kappa} \sum_{|\alpha''|=|\alpha|+1\atop
|\beta''|=|\beta|-1} \widetilde{\mathbb{D}}^{(q)}_{\alpha'',\beta'',\omega}[f](t) ,
\end{eqnarray*}

\begin{eqnarray*}	&&\sum_{|\beta_1|=|\beta'|=1}\left(\partial^{\alpha+\beta'}_x
\partial^{\beta+\beta_1-\beta'}_vY^\omega f, e^{\phi}\nu^{2\kappa(|\beta|+1)}
w^2_{\iota_0-N-1}\partial^{\alpha}_x\partial^{\beta+\beta_1}_vY^\omega f\right)\nonumber\\
	&=&\nu^\kappa\sum_{|\beta_1|=|\beta'|=1}\left(\nu^{\kappa|\beta|}w_{\iota_0-N-1}
\partial^{\beta'}_x\partial^{\alpha}_x
\partial^{\beta+\beta_1-\beta'}_vY^\omega f, e^{\phi}\nu^{\kappa(|\beta|+1)}
w_{\iota_0-N-1}\partial^{\beta_1}_v\partial^{\alpha}_x\partial^{\beta}_vY^\omega f\right)\nonumber\\
	&\lesssim&\nu^{\kappa}\sum_{|\beta_1|=|\beta'|=1}\left\|\nu^{\kappa|\beta|}w_{\iota_0-N-1}
\partial^{\beta'}_x\partial^{\alpha}_x
\partial^{\beta+\beta_1-\beta'}_vY^\omega f\right\| \left\|\nu^{\kappa(|\beta|+1)}
w_{\iota_0-N-1}\partial^{\beta_1}_v\partial^{\alpha}_x\partial^{\beta}_vY^\omega f\right\|  \nonumber\\
		&\lesssim& \nu^{\kappa} \sum_{|\alpha''|=|\alpha|+1\atop
|\beta''|=|\beta|-1} \left(A_0^{-s}\widetilde{\mathbb{D}}^{(q)}_{\alpha'',\beta'',\omega}[f](t) \right)^{\frac{1}{2}} \left(A_0^{-s}\widetilde{\mathbb{D}}_{N}^{(q,\iota_0)}[f](t) \right)^{\frac{1}{2}}\lesssim   A_0^{-s} \nu^{\kappa} \widetilde{\mathbb{D}}_{N}^{(q,\iota_0)}[f](t) ,
\end{eqnarray*}
 and 
\begin{eqnarray*}	&&\sum_{|\beta'|=1,|\tilde{\beta}|=2}\left(w_{\iota_0-N-2}\partial^{\alpha+\beta'}_x
\partial^{\beta+\tilde{\beta}-\beta'}_vY^\omega f, e^{\phi}\nu^{2\kappa(|\beta|+2)}
w^2_{\iota_0-N-2}\partial^\alpha_x\partial^{\beta+\tilde{\beta}}_vY^\omega f\right)\nonumber\\	&=&\nu^\kappa\sum_{|\beta'|=1,|\tilde{\beta}|=2}\left(\nu^{\kappa(|\beta|+1)}w_{\iota_0-N-2}\partial^{\alpha+\beta'}_x
\partial^{\beta+\tilde{\beta}-\beta'}_vY^\omega f, e^{\phi}\nu^{\kappa(|\beta|+2)}
w_{\iota_0-N-2}\partial^\alpha_x\partial^{\beta+\tilde{\beta}}_vY^\omega f\right)\nonumber\\
	&\lesssim&\nu^{\kappa}\sum_{|\beta'|=1,|\tilde{\beta}|=2}\left\|\nu^{\kappa(|\beta|+1)}w_{\iota_0-N-2}
\partial^{\alpha+\beta'}_x
\partial^{\beta+\tilde{\beta}-\beta'}_vY^\omega f\right\| \left\|\nu^{\kappa(|\beta|+2)}w_{\iota_0-N-2}\partial^{\beta'}_{v}\partial^\alpha_x\partial^{\beta+\tilde{\beta}-\beta'}_v Y^\omega f\right\|  \nonumber\\
		&\lesssim& \nu^{\kappa} A_0^{\frac{1-s}{2}}\sum_{|\alpha''|=|\alpha|+1\atop
|\beta''|=|\beta|-1} \left(\widetilde{\mathbb{D}}^{(q)}_{\alpha'',\beta'',\omega}[f](t) \right)^{\frac{1}{2}} A_0^{\frac{1}{2}}\sum_{|\beta''|=|\beta|-1} \left(\widetilde{\mathbb{D}}^{(q)}_{\alpha,\beta'',\omega}[f](t) \right)^{\frac{s}{2}}\left(\widetilde{\mathbb{D}}^{(q)}_{\alpha,\beta,\omega}[f](t) \right)^{\frac{1-s}{2}}\nonumber\\
&\lesssim& \nu^{\kappa} A_0^{\frac{1}{2}}\Big[\sum_{|\alpha''|=|\alpha|+1\atop
|\beta''|=|\beta|-1} \left(\widetilde{\mathbb{D}}^{(q)}_{\alpha'',\beta'',\omega}[f](t) \right)^{\frac{1}{2}} \sum_{|\beta''|=|\beta|-1} \left(\widetilde{\mathbb{D}}^{(q)}_{\alpha,\beta'',\omega}[f](t) \right)^{\frac{s}{2}}\Big]\left(A_0\widetilde{\mathbb{D}}^{(q)}_{\alpha,\beta,\omega}[f](t) \right)^{\frac{1-s}{2}}\nonumber\\
&\lesssim& \eta A_0\nu^{\kappa}\widetilde{\mathbb{D}}^{(q)}_{\alpha,\beta,\omega}[f](t) +A_0^{\frac{1}{1+s}}\nu^{\kappa}\bigg( \sum_{|\alpha''|=|\alpha|+1\atop
|\beta''|=|\beta|-1} \widetilde{\mathbb{D}}^{(q)}_{\alpha'',\beta'',\omega}[f](t) + \widetilde{\mathbb{D}}^{(q,\iota_0)}_{N-1}[f](t)\bigg)
\end{eqnarray*}
for any small constant $\eta>0$.
Therefore, applying $\partial_{x_i}\partial^\alpha_x\partial^\beta_vY^\omega, \partial_{v_i}\partial^\alpha_x\partial^\beta_vY^\omega$, and $ \partial^\alpha_x\partial^{\beta+\beta'}_vY^\omega$ with $|\beta'|=2$ to \eqref{f}, and multiplying the resulting equations by 
\begin{align*}
&\nu^{2\kappa|\beta|}e^{\phi}w_{\iota_0-N-1}^2\partial_{x_i}\partial^\alpha_x\partial^\beta_vY^\omega f, \quad \nu^{2\kappa(|\beta|+1)}e^{\phi}w_{\iota_0-N-1}^2\partial_{v_i}
\partial^\alpha_x\partial^\beta_vY^\omega f, \quad \mbox{and}\\ &\nu^{2\kappa(|\beta|+2)}e^{\phi}w_{\iota_0-N-2}^2\partial^\alpha_x\partial^{\beta+\beta'}_vY^\omega f,
\end{align*}
respectively, we can obtain
\begin{eqnarray}\label{G-2}
	&&\frac{1}{2}\frac{\mathrm{d}}{\mathrm{d}t}\left\|\nu^{\kappa|\beta|}e^{\frac{\phi}2}w_{\iota_0-N-1}\partial_{x_i}\partial^\alpha_x\partial^\beta_vY^\omega f\right\|^2+\nu\left\|\nu^{\kappa|\beta|}e^{\frac{\phi}2}w_{\iota_0-N-1}\partial_{x_i}\partial^\alpha_x\partial^\beta_vY^\omega f\right\|^2_{D}\nonumber\\
&\lesssim&\nu^{\kappa}A^{-1}_0\left(\eta\widetilde{\mathbb{D}}^{(q,\iota_0)}_N[f](t)
   +C_{\eta}
   \widetilde{\mathbb{D}}^{(q,\iota_0)}_{N-1}[f](t)\right)
   \nonumber\\
   &&+ A_0^{-\frac{s}{2}} \nu^{\kappa} \sum_{|\alpha''|=|\alpha|+1\atop
|\beta''|=|\beta|-1} \widetilde{\mathbb{D}}^{(q)}_{\alpha'',\beta'',\omega}[f](t)
   +\mathcal{M}^{\frac12}\langle t\rangle^{-2}\widetilde{\mathbb{E}}^{(q,\iota_0)}_N[f](t) \nonumber\\
	&&+A_0^{-\frac{1}{2}}\sum_{|\alpha|+|\omega|\leq N}\|\partial^\alpha_xY^\omega \rho_{\neq0}\| \min\left\{\left(\widetilde{\mathbb{E}}^{(q,\iota_0)}_N[f](t)\right)^{1/2},
\left(\widetilde{\mathbb{D}}^{(q,\iota_0)}_N[f](t)\right)^{1/2}\right\}\nonumber\\
            &&+\nu^{\kappa} \left(\widetilde{\mathbb{E}}_{N}^{(q,\iota_0)}[f](t)
\widetilde{\mathbb{D}}_{N_{\max}-2}^{(q,\iota_0)}[f](t)
\widetilde{\mathbb{D}}_{N}^{(q,\iota_0)}[f](t)\right)^{1/2},
\end{eqnarray}
\begin{eqnarray}\label{G-3}
	&&\frac{1}{2}\frac{\mathrm{d}}{\mathrm{d}t}\left\|\nu^{\kappa(|\beta|+1)}e^{\frac{\phi}2}w_{\iota_0-N-1}\partial_{v_i}\partial^\alpha_x\partial^\beta_vY^\omega f\right\|^2\nonumber\\
   &&+\nu\left\|\nu^{\kappa(|\beta|+1)}e^{\frac{\phi}2}w_{\iota_0-N-1}\partial_{v_i}\partial^\alpha_x\partial^\beta_vY^\omega f(\tau)\right\|^2_{D}\nonumber\\
&\lesssim&\nu^{\kappa}\left[\left(\eta+A_0^{-s}\right)\widetilde{\mathbb{D}}^{(q,\iota_0)}_N[f](t)
   +C_{\eta}\widetilde{\mathbb{D}}^{(q,\iota_0)}_{N-1}[f](t)
   \right]\nonumber\\
   &&+\mathcal{M}^{\frac12}\langle t\rangle^{-2}\widetilde{\mathbb{E}}^{(q,\iota_0)}_N[f](t)+C_{\eta} \nu\sum_{|\bar{\omega}|\leq N}\left\|{\bf P}^{\bot} \left(Y^{\bar{\omega}}f\right)\right\|^2_D+C(\eta) \mathcal{M}^2\nu^{1+2\kappa}\langle t\rangle^{-4}\nonumber\\
	&&+A_0^{-\frac{1}{2}}\sum_{|\alpha|+|\omega|\leq N}\|\partial^\alpha_xY^\omega \rho_{\neq0}\| \min\left\{\left(\widetilde{\mathbb{E}}^{(q,\iota_0)}_N[f](t)\right)^{1/2},
\left(\widetilde{\mathbb{D}}^{(q,\iota_0)}_N[f](t)\right)^{1/2}\right\}\nonumber\\
            &&+\nu^{\kappa} \left(\widetilde{\mathbb{E}}_{N}^{(q,\iota_0)}[f](t)
\widetilde{\mathbb{D}}_{N_{\max}-2}^{(q,\iota_0)}[f](t)
\widetilde{\mathbb{D}}_{N}^{(q,\iota_0)}[f](t)\right)^{1/2},
\end{eqnarray}
and
\begin{eqnarray}\label{G-4-max}
	&&\sum_{|\beta'|=2}\frac{1}{2}\frac{\mathrm{d}}{\mathrm{d}t}
\left\|\nu^{\kappa(|\beta|+2)}e^{\frac{\phi}2}w_{\iota_0-N-2}\partial^\alpha_x
\partial^{\beta+\beta'}_vY^\omega f(t)\right\|^2\nonumber\\
	&&+\nu\sum_{|\beta'|=2}\left\|\nu^{\kappa(|\beta|+2)}e^{\frac{\phi}2}w_{\iota_0-N-2}
\partial^\alpha_x
\partial^{\beta+\beta'}_vY^\omega f\right\|^2_{D}\nonumber\\
	&\lesssim&\nu^{\kappa}\left[A_0\eta\widetilde{\mathbb{D}}^{(q,\iota_0)}_N[f](t)
   +A_0C_{\eta}
   \widetilde{\mathbb{D}}^{(q,\iota_0)}_{N-1}[f](t)\right]
   \nonumber\\
   &&+A^{\frac{1}{2}}_0\nu^{\kappa} \sum_{|\alpha'|=|\alpha|+1,|\beta'|=|\beta|-1
   }\widetilde{\mathbb{D}}^{(q)}_{\alpha',\beta',\omega}[f](t)+\mathcal{M}^{\frac12}\langle t\rangle^{-2}\widetilde{\mathbb{E}}^{(q,\iota_0)}_N[f](t) \nonumber\\
   &&+C_{\eta} \nu\sum_{|\bar{\omega}|\leq N}\left\|{\bf P}^{\bot} \left(Y^{\bar{\omega}}f\right)\right\|^2_D+C(\eta) \mathcal{M}^2\nu^{1+2\kappa}\langle t\rangle^{-4}\nonumber\\
	&&+A_0^{-\frac{1}{2}}\sum_{|\alpha|+|\omega|\leq N}\|\partial^\alpha_xY^\omega \rho_{\neq0}\| \min\left\{\left(\widetilde{\mathbb{E}}^{(q,\iota_0)}_N[f](t)\right)^{1/2},
\left(\widetilde{\mathbb{D}}^{(q,\iota_0)}_N[f](t)\right)^{1/2}\right\}\nonumber\\
            &&+\nu^{\kappa} \left(\widetilde{\mathbb{E}}_{N}^{(q,\iota_0)}[f](t)
\widetilde{\mathbb{D}}_{N_{\max}-2}^{(q,\iota_0)}[f](t)
\widetilde{\mathbb{D}}_{N}^{(q,\iota_0)}[f](t)\right)^{1/2}.
\end{eqnarray}

\noindent\underline{{\it Step 3. Estimation of $\widetilde{\mathbb{E}}_{N}^{(q,\iota_0)}[f](t)$:}} 
Finally, we estimate $\widetilde{\mathbb{E}}_{N}^{(q,\iota_0)}[f](t)$ based on the estimates in Step 1, Step 2, and Lemma \ref{lemma-nonweight-Y}. We make a proper linear combination of \eqref{Enhanced-eqn}, \eqref{basic-l2}, \eqref{G-omega-noweight-end}, \eqref{G-1}, \eqref{G-2}, \eqref{G-3}, and \eqref{G-4-max} for all indexes $(\alpha,\beta,\omega)$ with $|\alpha|+|\beta|+|\omega|=N\leq N_{max}$, apply an induction on $N, |\beta|$, and integrate the resultant with respect to time $t$ to obtain
\begin{eqnarray}\label{G-end0}
		&&\widetilde{\mathbb{E}}^{(q,\iota_0)}_N[f](t)		+\nu^\kappa\int_{0}^{t}\widetilde{\mathbb{D}}^{(q,\iota_0)}_N[f](\tau)\,\mathrm{d}\tau\nonumber\\
        &\lesssim&\widetilde{\mathbb{E}}^{(q,\iota_0)}_N[f](0)+\int_0^t \langle \tau\rangle^{-2} \widetilde{\mathbb{E}}^{(q,\iota_0)}_N[f](\tau)\,\mathrm{d}\tau\nonumber\\
&&+\nu^\kappa\int_0^t \left(\widetilde{\mathbb{E}}_{N}^{(q,\iota_0)}[f](\tau)
\widetilde{\mathbb{D}}_{N_{\max}-2}^{(q,\iota_0)}[f](\tau)
\widetilde{\mathbb{D}}_{N}^{(q,\iota_0)}[f](\tau)\right)^{1/2}\,\mathrm{d}\tau\nonumber\\
&&+\int^t_0\sum_{|\alpha|+|\omega|\leq N}\|\partial^\alpha_xY^\omega \rho_{\neq0}(\tau)\| \min\left\{\left(\widetilde{\mathbb{E}}^{(q,\iota_0)}_N[f](\tau)\right)^{1/2},
\left(\widetilde{\mathbb{D}}^{(q,\iota_0)}_N[f](\tau)\right)^{1/2}\right\}
\,\mathrm{d}\tau.
	\end{eqnarray}
Note that for any small $\eta>0$,
\begin{align*}
   &\nu^\kappa\int_0^t \left(\widetilde{\mathbb{E}}_{N}^{(q,\iota_0)}[f](\tau)
\widetilde{\mathbb{D}}_{N_{\max}-2}^{(q,\iota_0)}[f](\tau)
\widetilde{\mathbb{D}}_{N}^{(q,\iota_0)}[f](\tau)\right)^{1/2}\,\mathrm{d}\tau\\
&\qquad \lesssim \eta \nu^\kappa\int_{0}^{t}\widetilde{\mathbb{D}}^{(q,\iota_0)}_N[f](\tau)\,\mathrm{d}\tau
+C(\eta)\nu^\kappa\sup_{0\leq \tau\leq t}\widetilde{\mathbb{E}}_{N}^{(q,\iota_0)}[f](\tau)
\int_{0}^{t}\widetilde{\mathbb{D}}^{(q,\iota_0)}_{N_{\max}-2}[f](\tau)\,\mathrm{d}\tau\\
&\qquad \lesssim \eta \nu^\kappa\int_{0}^{t}\widetilde{\mathbb{D}}^{(q,\iota_0)}_N[f](\tau)\,\mathrm{d}\tau
+C(\eta)\mathcal{M}\nu^{2\kappa}\sup_{0\leq \tau\leq t}\widetilde{\mathbb{E}}_{N}^{(q,\iota_0)}[f](\tau)
\end{align*}
by the {\it a priori} assumption \eqref{priori-assu}.

Now we estimate the last term in \eqref{G-end0}. For $N\leq N_{max}-1$, choosing small positive constant $\eta$, we apply the poincaré inequality and use \eqref{priori-assu}, \eqref{Th-rho-1} to bound it by
\begin{align*}
&\int^t_0\sum_{|\alpha|+|\omega|\leq N}\|\nabla_x^{N_{max}-N}\partial^\alpha_xY^\omega \rho_{\neq0}(\tau)\| 
\left(\widetilde{\mathbb{E}}^{(q,\iota_0)}_N[f](\tau)\right)^{1/2}
\,\mathrm{d}\tau\\
&\qquad\lesssim  \sup_{0\leq \tau\leq t}\left(\widetilde{\mathbb{E}}_{N}^{(q,\iota_0)}[f](\tau)\right)^{1/2}
 \int^t_0\langle \tau\rangle^{N-N_{max}}\sum_{|\alpha|+|\bar{\omega}|\leq N_{max}}\|\partial^\alpha_xY^{\bar{\omega}} \rho_{\neq0}(\tau)\|
\,\mathrm{d}\tau\\
&\qquad\lesssim \eta \sup_{0\leq \tau\leq t}\widetilde{\mathbb{E}}_{N}^{(q,\iota_0)}[f](\tau) +
C(\eta)  \mathcal{M}^{2}\nu^{2\kappa} \left(\int_{0}^{t}\langle \tau\rangle^{N-N_{max}}\,\mathrm{d}\tau\right)^2\\
&\qquad\lesssim \eta \sup_{0\leq \tau\leq t}\widetilde{\mathbb{E}}_{N}^{(q,\iota_0)}[f](\tau) +
C(\eta)  \mathcal{M}^{2}\nu^{2\kappa} \left(\chi_{\{N\leq N_{max}-2\}}+ [\ln \langle t\rangle ]^2 1\chi_{\{N= N_{max}-1\}}\right).
\end{align*}
For $N= N_{\max}$, we can once again apply the Cauchy inequalities together with \eqref{priori-assu} and \eqref{Th-rho-1} to bound it by
    \begin{eqnarray*}
     &&\min\left\{ \eta \sup_{0\leq \tau\leq t}\widetilde{\mathbb{E}}_{N}^{(q,\iota_0)}[f](\tau)
     +C(\eta)\mathcal{M}^2\nu^{2\kappa}\langle t\rangle^{2}, \int_0^t \sum_{|\alpha|+|\omega|= N_{max}}\left\|\partial^\alpha_xY^{\bar{\omega}} \rho_{\neq0}(\tau)\right\|\left(\widetilde{\mathbb{D}}^{(q,\iota_0)}_N[f](\tau)
     \right)^{1/2}\,\mathrm{d}\tau\right\}\nonumber\\
     &\lesssim&\min\left\{\eta \sup_{0\leq \tau\leq t}\widetilde{\mathbb{E}}_{N}^{(q,\iota_0)}[f](\tau)
     +C(\eta)\mathcal{M}^2\nu^{2\kappa}\langle t\rangle^{2},\right.\\
       &&\qquad\qquad\left.\eta \nu^\kappa\int_{0}^{t}\widetilde{\mathbb{D}}^{(q,\iota_0)}_N[f](\tau)\,\mathrm{d}\tau
+C(\eta)\nu^{-\kappa}\int_0^t\sum_{|\alpha|+|\omega|= N_{max}}\left\|\partial^\alpha_xY^{\bar{\omega}} \rho_{\neq0}(\tau)\right\|^2\, \mathrm{d}\tau \right\}\nonumber\\
        &\lesssim&\min\left\{\eta \sup_{0\leq \tau\leq t}\widetilde{\mathbb{E}}_{N}^{(q,\iota_0)}[f](\tau)
     +C(\eta)\mathcal{M}^2\nu^{2\kappa}\langle t\rangle^{2}, \eta \nu^\kappa\int_{0}^{t}\widetilde{\mathbb{D}}^{(q,\iota_0)}_N[f](\tau)\,\mathrm{d}\tau
+C(\eta)\mathcal{M}^2\right\}.
    \end{eqnarray*}
We insert the above estimates in \eqref{G-end0}, apply the Grönwall inequality and use the smallness of $\eta, \nu$ to derive \eqref{priori-ener}. This completes the proof of Theorem \ref{Th-basic-estimate}.
\end{proof}

\subsection{Global existence}
Theorem \ref{Th-basic-estimate} together with Lemma \ref{elec-esti} yields the a priori bound
\begin{align}
\sup_{ 0\leq t\leq T_0}&\Bigg\{\max\{\nu^{-\kappa},\langle t\rangle\}^{\min\{0,-N+N_{max}-2\}}\left\{\widetilde{\mathbb{E}}^{(q,\iota_0)}_N[f](t)
+\nu^\kappa\int_{0}^{t}\widetilde{\mathbb{D}}^{(q,\iota_0)}_N[f](\tau)\mathrm{d}\tau\right\}\nonumber\\
	&+\langle t\rangle^{4}\Big\{\|\phi(t)\|_{W^{5,\infty}_x}^2+\sum_{|\alpha|+|\omega|\leq 4}\|\partial^\alpha_xY^\omega \nabla_x\phi(t)\|^2_{L^\infty_x}\Big\}\Bigg\}\leq \mathcal{M}^2\nu^{2\kappa},\label{apri-es}
\end{align}
which closes the energy estimates.

On the other hand, the local existence and uniqueness of solutions to the Cauchy problem \eqref{f} and \eqref{f-initial} can be established, following a suitable adaptation of the arguments in \cite{DLYZ-KRM2013}. Combining this with the a priori estimate \eqref{apri-es}, the standard continuation argument extends the local solution globally in time. We therefore arrive at the following result:
\begin{theorem}\label{glob-exis}
  Under the assumptions of Theorem \ref{Main-Th.}, the Cauchy problem \eqref{f} and \eqref{f-initial} admits a unique global smooth solution $f(t,x,v)$. Moreover, the estimates in \eqref{Th-rho-1}, Lemma \ref{elec-esti}, and Theorem \ref{Th-basic-estimate} remain valid uniformly for all $t\in [0,\infty)$.
\end{theorem}

\section{The stretched exponential decay}\label{The stretched exponential decay}
To complete the proof of Theorem \ref{Main-Th.}, it remains to establish the stretched exponential decay of the solution $f(t,x,v)$ to the Cauchy problem \eqref{f}–\eqref{f-initial}, as well as the corresponding decay of its density function $\rho(t,x)$. We derive these time-decay estimates in this section.

To this end, let $\bar{q}=\tfrac{q}{2}$, fix a small positive constant $\delta$, and let $T\in (0,\infty)$ be given. we introduce another {\it a priori} assumptions:
\begin{align}\label{priori-assu0}
    \sup_{0\leq t<T}e^{\delta(\nu^\kappa  t
    )^{\frac{s}{s-\gamma}}}\widetilde{\mathbb{E}}^{(\bar{q},\iota_0)}_{1,1,0,0}[f](t)
    +\nu^\kappa\int_0^{T}e^{\delta(\nu^\kappa  t
    )^{\frac{s}{s-\gamma}}}\widetilde{\mathbb{D}}^{(\bar{q},\iota_0)}_{1,1,0,0}[f](\tau)\,\mathrm{d}\tau\leq \mathcal{M}\nu^{2\kappa},
\end{align}
and
\begin{align}\label{priori-assu1}
    \sup_{0\leq t<T}e^{\delta(\nu  t
    )^{\frac{1}{1-\gamma-2s}}}\widetilde{\mathbb{E}}^{(\bar{q},\iota_0)}_{1,1,0,0}[f](t)+\nu^\kappa\int_0^{T}e^{\delta(\nu  t
    )^{\frac{1}{1-\gamma-2s}}}\widetilde{\mathbb{D}}^{(\bar{q},\iota_0)}_{1,1,0,0}[f](\tau)
    \,\mathrm{d}\tau\leq \mathcal{M}\nu^{2\kappa}.
\end{align}
\subsection{Nonlinear density estimates}
In this subsection, we will prove the  stretched exponential decay of the density $\rho(t,x,v)$ as follows.
\begin{theorem}\label{exp-dens}
Under conditions in Theorem \ref{Main-Th.},  we further assume that the above {\it a priori} assumptions \eqref{priori-assu0} and \eqref{priori-assu1} hold for the unique solution $f(t,x,v)$ to the Cauchy problem \eqref{f}-\eqref{f-initial} for $t\in [0, T_0)$. Then we have
\begin{itemize}
    \item For any $p\in[2,+\infty]$, $\rho_{\neq 0}$ satisfies
    \begin{align}\label{density-decay-1}
        \sum_{|\alpha|\leq 1}\left(\left\|e^{\frac{\delta}2(\nu^\kappa  t
    )^{\frac{s}{s-\gamma}}}\partial^\alpha_x\rho_{\neq 0}\right\|^2_{L^p_t([0,T_0); L^2_x)}+\left\|e^{\frac{\delta}2(\nu  t
    )^{\frac{1}{1-\gamma-2s}}}\partial^\alpha_x\rho_{\neq 0}\right\|^2_{L^p_t([0,T_0); L^2_x)}\right)\lesssim \mathcal{M}^2\nu^{2\kappa}.
    \end{align}
    \item $\rho_{\neq 0}$ admits a decomposition $\rho_{\neq 0}=\rho_{\neq 0}^{(1)}+\rho_{\neq 0}^{(2)}$ such that
     \begin{align}\label{density-decay-3}
        \sum_{|\alpha|\leq 1}\left(\left\|e^{\frac{\delta}2(\nu^\kappa  t
    )^{\frac{s}{s-\gamma}}}\partial^\alpha_x\rho^{(1)}_{\neq 0}\right\|^2_{L^1_t([0,T_0); L^2_x)}+\left\|e^{\frac{\delta}2(\nu  t
    )^{\frac{1}{1-\gamma-2s}}}\partial^\alpha_x\rho^{(1)}_{\neq 0}\right\|^2_{L^1_t([0,T_0); L^2_x)}\right)\lesssim \mathcal{M}^2\nu^{2\kappa}, \quad \mbox{and}
    \end{align}
             \begin{align}\label{density-decay-4}
        \sum_{|\alpha|\leq 1}\left(\left\|e^{\frac{\delta}2(\nu^\kappa  t
    )^{\frac{s}{s-\gamma}}}\partial^\alpha_x\rho^{(2)}_{\neq 0}\right\|^2_{L^2_t([0,T_0); L^2_x)}+\left\|e^{\frac{\delta}2(\nu  t
    )^{\frac{1}{1-\gamma-2s}}}\partial^\alpha_x\rho^{(2)}_{\neq 0}\right\|^2_{L^2_t([0,T_0); L^2_x)}\right)\lesssim \mathcal{M}^2\nu.
    \end{align}
Here  $\rho_{\neq 0}^{(1)}=\sum_{k\neq0}\widehat{\rho}_k^{(1)}e^{ik\cdot x}$ with
 \begin{eqnarray*}
	\widehat{\rho}_k^{(1)}(t)=\mathcal{N}_{k,\textbf{initial}}(t)
+\mathcal{N}_{k,\textbf{NL-electric}}(t)+\int_{0}^{t}G_k(t-\tau)
[\mathcal{N}_{k,\textbf{initial}}(\tau)+\mathcal{N}_{k,\textbf{NL-electric}}(\tau)]\,
\mathrm{d}\tau,
\end{eqnarray*}
and $\rho_{\neq 0}^{(2)}=\sum_{k\neq0}\widehat{\rho}_k^{(2)}e^{ik\cdot x}$, 
 \begin{eqnarray*}
	\widehat{\rho}_k^{(2)}(t)=\mathcal{N}_{k,\textbf{NL-collision}}(t)
+\int_{0}^{t}G_k(t-\tau)\mathcal{N}_{k,\textbf{NL-collision}}(\tau)\,\mathrm{d}\tau.
\end{eqnarray*}
Note that $\mathcal{N}_{k,\textbf{initial}}(t),\mathcal{N}_{k,\textbf{NL-electric}}(t)$, and $\mathcal{N}_{k,\textbf{NL-collision}}(t)$ have been defined in \eqref{def-N-k-int-NL}.
\end{itemize}
\end{theorem}

\begin{proof} The proof of \eqref{density-decay-1}, \eqref{density-decay-3}, and \eqref{density-decay-4} with different stretched exponential weights is similar, we focus on these estimates with the stretched exponential weight $e^{\frac{\delta}2(\nu^\kappa  t
    )^{\frac{s}{s-\gamma}}}$ for brevity.
Then, as the proof of Theorem $11.1$ in \cite{CLN-JAMS-2023}, our proof can be reduced to establishing the validity of the following two claims:
\begin{align}\label{claim-1}
 &\sup_{0\leq\tau\leq t}\sum_{k\neq 0}|k|^2e^{\delta(\nu^\kappa  \tau
)^{\frac{s}{s-\gamma}}}|\mathcal{N}_{k,\textbf{initial}}(\tau)|^2+\sup_{0\leq\tau\leq t}\sum_{k\neq 0}|k|^2e^{\delta(\nu^\kappa  \tau
)^{\frac{s}{s-\gamma}}}|\mathcal{N}_{k,\textbf{NL-electric}}(\tau)|^2\nonumber\\
&+\left\{\int_0^t\left[\sum_{k\neq 0}|k|^2e^{\delta(\nu^\kappa  \tau
)^{\frac{s}{s-\gamma}}}\left(|\mathcal{N}_{k,\textbf{initial}}(\tau)|^2+|\mathcal{N}_{k,\textbf{NL-electric}}(\tau)|^2\right)\right]^{\frac12}\mathrm{d}\tau\right\}^2\nonumber\\
\lesssim& \mathcal{M}^2\nu^{2\kappa}
+\mathcal{M} \zeta (t),
\end{align}
and
\begin{align}\label{claim-2}
 &\sup_{0\leq\tau\leq t}\sum_{k\neq 0}|k|^2e^{\delta(\nu^\kappa  \tau
)^{\frac{s}{s-\gamma}}}|\mathcal{N}_{k,\textbf{NL-collision}}(\tau)|^2\nonumber\\
&
+\int_0^t\sum_{k\neq 0}|k|^2e^{\delta(\nu^\kappa  \tau
)^{\frac{s}{s-\gamma}}}|\mathcal{N}_{k,\textbf{NL-collision}}(\tau)|^2\mathrm{d}\tau
\lesssim \mathcal{M}^2\nu,
\end{align}
where
\begin{align*}
    \zeta(t)=&\sup_{0\leq\tau\leq t}\sum_{k\neq 0}e^{\delta(\nu^\kappa  t
)^{\frac{s}{s-\gamma}}}|k|^2|\widehat{\rho}_k^{(1)}(t)|^2+\left(\int_0^t\bigg[\sum_{k\neq 0}e^{\delta(\nu^\kappa  \tau
)^{\frac{s}{s-\gamma}}}|k|^2|\widehat{\rho}_k^{(1)}(\tau)|^2\bigg]^{\frac12}\mathrm{d}\tau\right)^2\nonumber\\
&+\nu^{2\kappa-1}\sup_{0\leq\tau\leq t}\sum_{k\neq 0}e^{\delta(\nu^\kappa  t
)^{\frac{s}{s-\gamma}}}|k|^2|\widehat{\rho}_k^{(2)}(t)|^2+\nu^{2\kappa-1}\int_0^t\sum_{k\neq 0}e^{\delta(\nu^\kappa  \tau
)^{\frac{s}{s-\gamma}}}|k|^2|\widehat{\rho}_k^{(2)}(\tau)|^2\mathrm{d}\tau.
\end{align*}
Since the proof of \eqref{claim-1}, which corresponds to \cite[$(11.9)$ in Proposition $11.2$]{CLN-JAMS-2023}, is quite similar to the Vlasov-Poisson-Landau system case treated there, we only prove \eqref{claim-2} for simplicity.

For $\delta_0$ given in \eqref{Pro-linear-decay-2}, we choose $\delta>0$ sufficiently small such that $\delta \leq \frac{\delta_0}{2}$, which yields
\begin{align}\label{delt-delt}
\delta(\nu^\kappa  \tau
)^{\frac{s}{s-\gamma}}\leq \delta(\nu^\kappa\tau_1
)^{\frac{s}{s-\gamma}}+\frac{\delta_0}{2}  (\nu^\kappa  (\tau-\tau_1)
)^{\frac{s}{s-\gamma}},\qquad 0\leq \tau_1\leq \tau.
\end{align}
Then we apply \eqref{Pro-linear-decay-2} to have
\begin{align*}
&\sum_{k\neq 0}|k|^2e^{\delta(\nu^\kappa  \tau
)^{\frac{s}{s-\gamma}}}|\mathcal{N}_{k,\textbf{NL-collision}}(\tau)|^2\nonumber\\
\lesssim  &\sum_{k\neq 0}|k|^2e^{\delta(\nu^\kappa  \tau
)^{\frac{s}{s-\gamma}}}\left(\int_{0}^{\tau}\int_{{\mathbb{R}}^3}
S_k(\tau-\tau_1)\mathcal{F}_x\left[\nu\Gamma(f,f)(\tau_1)\right]\mu^{\frac{1}{2}}\,\mathrm{d}v\mathrm{d}\tau_1\right)^2\nonumber\\
\lesssim  &\sum_{k\neq 0}\nu^2e^{\delta(\nu^\kappa  \tau
)^{\frac{s}{s-\gamma}}}\left(\int_{0}^{\tau}e^{-\delta_0(\nu^\kappa  (\tau-\tau_1  )
)^{\frac{s}{s-\gamma}}}\sum_{|\alpha|=1,|\beta|\leq 1}\nu^{\kappa|\beta|}\left|\langle v\rangle^{\frac{-\gamma}{2s}}e^{\bar{q}\langle v\rangle}\mathcal{F}_x\left[\partial^\alpha_x\partial^\beta_v\Gamma(f,f)(\tau_1)
\right]\right|_{L^2}\,\mathrm{d}\tau_1\right)^2\nonumber\\
\lesssim  &\sum_{k\neq 0}\nu^2e^{\delta(\nu^\kappa  \tau
)^{\frac{s}{s-\gamma}}}\int_{0}^{\tau}e^{-\delta_0(\nu^\kappa  (\tau-\tau_1 )
)^{\frac{s}{s-\gamma}}}\,\mathrm{d}\tau_1\nonumber\\
&\times\int_{0}^{\tau}e^{-\delta_0(\nu^\kappa  (\tau-\tau_1 )
)^{\frac{s}{s-\gamma}}}\sum_{|\alpha|=1,|\beta|\leq 1}\nu^{2\kappa|\beta|}\left|\langle v\rangle^{\frac{-\gamma}{2s}}e^{\bar{q}\langle v\rangle}\mathcal{F}_x\left[\partial^\alpha_x\partial^\beta_v\Gamma(f,f)\right]
\right|^2_{L^2}\,
\mathrm{d}\tau_1\nonumber\\
\lesssim  &\sum_{k\neq 0}\nu^{2-\kappa}\int_{0}^{\tau}e^{\delta(\nu^\kappa  \tau_1
)^{\frac{s}{s-\gamma}}}e^{-\frac{\delta_0}{2}
(\nu^\kappa  (\tau-\tau_1  )   )^{\frac{s}{s-\gamma}}}
\sum_{|\alpha|=1,|\beta|\leq 1}\nu^{2\kappa|\beta|}\left|\langle v\rangle^{\frac{-\gamma}{2s}}e^{\bar{q}\langle v\rangle}\mathcal{F}_x\left[\partial^\alpha_x\partial^\beta_v\Gamma(f,f)(\tau_1)\right]
\right|^2_{L^2}\,\mathrm{d}\tau_1.
\end{align*}
Note that from Lemma \ref{Lemma GammaL},
\begin{align*}
&\sum_{k\neq 0}\sum_{|\alpha|=1,|\beta|\leq 1}\left|\langle v\rangle^{\frac{-\gamma}{2s}}e^{\bar{q}\langle v\rangle}\mathcal{F}_x\left[\partial^\alpha_x\partial^\beta_v\Gamma(f,f)(\tau_1)\right]
\right|^2_{L^2}\\
&\qquad \leq\sum_{|\alpha|=1,|\beta|\leq 1\atop |\beta_1|\leq |\beta|}\nu^{2\kappa|\beta|}\left\|\langle v\rangle^{\frac{-\gamma}{2s}}e^{\bar{q}\langle v\rangle}\partial^\alpha_x\partial^{\beta_1}_vf(\tau_1)\right\|^2
\left\||\langle v\rangle^{\frac{-\gamma}{2s}}e^{\bar{q}\langle v\rangle}\partial^{\beta-\beta_1}_vf(\tau_1)|_{H^{2s}}\right\|_{L^{\infty}_x}^2\\
&\qquad +\sum_{|\alpha|=1,|\beta|\leq 1\atop |\beta_1|\leq |\beta|}\nu^{2\kappa|\beta|}\left\||\langle v\rangle^{\frac{-\gamma}{2s}}e^{\bar{q}\langle v\rangle}\partial^{\beta_1}_vf(\tau_1)|_{L^2}\right\|_{L^{6}_x}^2
\left\||\langle v\rangle^{\frac{-\gamma}{2s}}e^{\bar{q}\langle v\rangle}\partial^\alpha_x\partial^{\beta-\beta_1}_vf(\tau_1)
|_{H^{2s}}\right\|_{L^{3}_x}^2\\
&\qquad \lesssim \nu^{\kappa-1-2\kappa s}\left(\widetilde{\mathbb{E}}^{(\bar{q},\iota_0)}_{1,1,0,0}[f](\tau_1) \widetilde{\mathbb{D}}^{(\bar{q},\iota_0)}_{N_{\max}-2}[f](\tau_1)
+\widetilde{\mathbb{E}}^{(\bar{q},\iota_0)}_{N_{\max}-2}[f](\tau_1)
\widetilde{\mathbb{D}}^{(\bar{q},\iota_0)}_{1,1,0,0}[f](\tau_1)\right),
\end{align*}
Here we have used the condition $\gamma + 2s < 0$ and the following interpolation estimates:
\begin{align*}
&\nu^{1-\kappa+2\kappa s}\left\|\langle v\rangle^{\frac{-\gamma}{2s}}e^{\bar{q}\langle v\rangle}\widetilde{f}(\tau)\right\|^2_{H^{2s}}\nonumber\\
\lesssim& \left(\nu^{1-\kappa}\left\|\langle v\rangle^{\frac{-\gamma}{2s}}e^{\bar{q}\langle v\rangle}\widetilde{f}(\tau)\right\|^2_{H^{s}}\right)^{1-s}\left(\nu^{1+\kappa}\left\|\langle v\rangle^{\frac{-\gamma}{2s}}e^{\bar{q}\langle v\rangle}\nabla_v\widetilde{f}(\tau)\right\|^2_{H^{s}}\right)^{s}\\
\lesssim& \left(\nu^{1-\kappa}\left\|\langle v\rangle^{\frac{-\gamma}{2s}-\gamma}e^{\bar{q}\langle v\rangle}\widetilde{f}(\tau)\right\|^2_{D}\right)^{1-s}\left(\nu^{1+\kappa}\left\|\langle v\rangle^{\frac{-\gamma}{2s}-\gamma}e^{\bar{q}\langle v\rangle}\nabla_v\widetilde{f}(\tau)\right\|^2_{D}\right)^{s}
\\
\lesssim& \widetilde{\mathbb{D}}^{(\bar{q},\iota_0)}_{N_{max}-2}[f(\tau)],
\end{align*}
where $\widetilde{f} = \nu^{\kappa|\beta-\beta_1|}\partial_x^{\alpha''} \partial_v^{\beta-\beta_1} f$ with $|\alpha''| \leq 2$.

Consequently, we can further obtain
\begin{align}
\label{claim-2-1}
&\sum_{k\neq 0}|k|^2e^{\delta(\nu^\kappa  \tau
)^{\frac{s}{s-\gamma}}}|\mathcal{N}_{k,\textbf{NL-collision}}(\tau)|^2\nonumber\\
\lesssim  &\nu^{1-2\kappa s}\int_{0}^{\tau}e^{-\frac{\delta_0}{2}(\nu^\kappa  (\tau-\tau_1
)^{\frac{s}{s-\gamma}}}\widetilde{\mathbb{E}}^{(\bar{q},\iota_0)}_{N_{\max}-2}[f](\tau_1)
e^{\delta(\nu^\kappa  \tau_1
)^{\frac{s}{s-\gamma}}}\widetilde{\mathbb{D}}^{(\bar{q},\iota_0)}_{1,1,0,0}[f](\tau_1)
\mathrm{d}\tau_1\nonumber\\
&+\nu^{1-2\kappa s}\int_{0}^{\tau}e^{-\frac{\delta_0}{2}(\nu^\kappa  (\tau-\tau_1
)^{\frac{s}{s-\gamma}}}\widetilde{\mathbb{D}}^{(\bar{q},\iota_0)}_{N_{\max}-2}[f](\tau_1)e^{\delta(\nu^\kappa  \tau_1
)^{\frac{s}{s-\gamma}}}\widetilde{\mathbb{E}}^{(\bar{q},\iota_0)}_{1,1,0,0}[f](\tau_1)
\mathrm{d}\tau_1\nonumber\\
\lesssim  &\nu^{1-2\kappa s}\sup_{\tau_1\in[0,\tau]}\widetilde{\mathbb{E}}^{(\bar{q},\iota_0)}_{N_{\max}-2}[f](\tau_1)\int_{0}^{\tau}e^{\delta(\nu^\kappa  \tau_1
)^{\frac{s}{s-\gamma}}}\widetilde{\mathbb{D}}^{(\bar{q},\iota_0)}_{1,1,0,0}[f](\tau_1)
\mathrm{d}\tau_1\nonumber\\
&+\nu^{1-2\kappa s}\sup_{\tau_1\in[0,\tau]}\left\{e^{\delta(\nu^\kappa  \tau_1
)^{\frac{s}{s-\gamma}}}\widetilde{\mathbb{E}}^{(\bar{q},\iota_0)}_{1,1,0,0}[f](\tau_1)\right\}
\int_{0}^{\tau}\widetilde{\mathbb{D}}^{(\bar{q},\iota_0)}_{N_{\max}-2}[f](\tau_1)\,
\mathrm{d}\tau_1\nonumber\\
\lesssim  &\mathcal{M}\nu\int_{0}^{\tau}e^{\delta(\nu^\kappa  \tau_1
)^{\frac{s}{s-\gamma}}}\widetilde{\mathbb{D}}^{(\bar{q},\iota_0)}_{1,1,0,0}[f](\tau_1)
\mathrm{d}\tau_1+\mathcal{M}\nu^{1-\kappa}\sup_{\tau_1\in[0,\tau]}\left\{e^{\delta(\nu^\kappa  \tau_1
)^{\frac{s}{s-\gamma}}}\widetilde{\mathbb{E}}^{(\bar{q},\iota_0)}_{1,1,0,0}[f](\tau_1)\right\}\nonumber\\
\lesssim&\mathcal{M}^2\nu^{1+\kappa}.
\end{align}
For the $L^1$ norm estimate, we use \eqref{delt-delt} and proceed similarly to obtain
\begin{align}\label{claim-2-2}
&\int_0^t\sum_{k\neq 0}|k|^2e^{\delta(\nu^\kappa  \tau
)^{\frac{s}{s-\gamma}}}|\mathcal{N}_{k,\textbf{NL-collision}}(\tau)|^2\,\mathrm{d}\tau\nonumber\\
\lesssim&\nu^{1-2\kappa}\int_0^t\int_{0}^{\tau}e^{-\frac{\delta_0}{2}(\nu^\kappa  (\tau-\tau_1
)^{\frac{s}{s-\gamma}}}\widetilde{\mathbb{E}}^{(\bar{q},\iota_0)}_{N_{\max}-2}[f](\tau_1)
e^{\delta(\nu^\kappa  \tau_1
)^{\frac{s}{s-\gamma}}}\widetilde{\mathbb{D}}^{(\bar{q},\iota_0)}_{1,1,0,0}[f](\tau_1)
\,\mathrm{d}\tau_1\mathrm{d}\tau\nonumber\\
&+\nu^{1-2\kappa}\int_0^t\int_{0}^{\tau}e^{-\frac{\delta_0}{2}(\nu^\kappa  (\tau-\tau_1
)^{\frac{s}{s-\gamma}}}\widetilde{\mathbb{D}}^{(\bar{q},\iota_0)}_{N_{\max}-2}[f](\tau_1)
e^{\delta(\nu^\kappa  \tau_1
)^{\frac{s}{s-\gamma}}}\widetilde{\mathbb{E}}^{(\bar{q},\iota_0)}_{1,1,0,0}[f](\tau_1)\,
\mathrm{d}\tau_1\mathrm{d}\tau\nonumber\\
\lesssim  &\nu^{1-3\kappa}\sup_{\tau_1\in[0,\tau]}\widetilde{\mathbb{E}}^{(\bar{q},\iota_0)}_{N_{\max}-2}[f](\tau_1)\int_{0}^{\tau}e^{\delta(\nu^\kappa  \tau_1
)^{\frac{s}{s-\gamma}}}\widetilde{\mathbb{D}}^{(\bar{q},\iota_0)}_{1,1,0,0}[f](\tau_1)\,
\mathrm{d}\tau_1\nonumber\\
&+\nu^{1-3\kappa}\sup_{\tau_1\in[0,\tau]}\left\{e^{\delta(\nu^\kappa  \tau_1
)^{\frac{s}{s-\gamma}}}\widetilde{\mathbb{E}}^{(\bar{q},\iota_0)}_{1,1,0,0}[f](\tau_1)\right\}
\int_{0}^{\tau}\widetilde{\mathbb{D}}^{(\bar{q},\iota_0)}_{N_{\max}-2}[f](\tau_1)\,
\mathrm{d}\tau_1\nonumber\\
\lesssim  &\mathcal{M}\nu^{1-\kappa}\int_{0}^{\tau}e^{\delta(\nu^\kappa  \tau_1
)^{\frac{s}{s-\gamma}}}\widetilde{\mathbb{D}}^{(\bar{q},\iota_0)}_{1,1,0,0}[f](\tau_1)\,
\mathrm{d}\tau_1+\mathcal{M}\nu^{1-2\kappa}\sup_{\tau_1\in[0,\tau]}
\left\{e^{\delta(\nu^\kappa  \tau_1
)^{\frac{s}{s-\gamma}}}\widetilde{\mathbb{E}}^{(\bar{q},\iota_0)}_{1,1,0,0}[f](\tau_1)\right\}\nonumber\\
\lesssim&\mathcal{M}^2\nu.
\end{align}
Combing \eqref{claim-2-1} and \eqref{claim-2-2} yields \eqref{claim-2}. This completes the proof of Theorem \eqref{exp-dens}.
 \end{proof}
\subsection{Nonlinear energy estimates}
In this subsection, we will prove the nonlinear energy estimates with exponential time decay.
Unlike the angular cutoff Boltzmann operator and the Landau operator, according to relevant proofs in \cite{DLYZ-KRM2013}, due to the complexity of the non-cutoff Boltzmann operator, only exponential weighted estimates of first order in velocity have been established so far. Therefore, in our subsequent estimates, the velocity variable in the exponential weight is merely of first order.

\begin{theorem}\label{Th-nonlinear-energy-exponential}
    Under the conditions in Theorem \ref{Main-Th.}, for $\bar{q}=\frac{q}{2}$ and small constant $\delta$,  we have
  \begin{equation}\label{Th-nonlinear-energy-exponential-1}
      \sup_{0\leq \tau <+\infty}\left\{e^{\delta(\nu  \tau
    )^{\frac{1}{1-\gamma-2s}}}\widetilde{\mathbb{E}}^{(\bar{q},\iota_0)}_{0,0,0,0}[f](\tau)\right\}
    +\nu^\kappa\int_0^\infty e^{\delta(\nu  \tau
    )^{\frac{1}{1-\gamma-2s}}}\widetilde{\mathbb{D}}^{(\bar{q},\iota_0)}_{0,0,0,0}[f](\tau)\mathrm{d}\tau \lesssim\mathcal{M}^2\nu^{1-\kappa},
  \end{equation}
  \begin{equation}\label{Th-nonlinear-energy-exponential-2}
      \sup_{0\leq \tau <+\infty}\left\{e^{\delta(\nu  \tau
    )^{\frac{1}{1-\gamma-2s}}}\widetilde{\mathbb{E}}^{(\bar{q},\iota_0)}_{1,1,0,0}[f](\tau)\right\}+\nu^\kappa\int_0^\infty e^{\delta(\nu  \tau
    )^{\frac{1}{1-\gamma-2s}}}\widetilde{\mathbb{D}}^{(\bar{q},\iota_0)}_{1,1,0,0}[f](\tau)\mathrm{d}\tau \lesssim\mathcal{M}^2\nu^{1-\kappa},
  \end{equation}
  and
   \begin{equation}\label{Th-nonlinear-energy-exponential-3}
      \sup_{0\leq \tau <+\infty}\left\{e^{\delta(\nu^\kappa  \tau
)^{\frac{s}{s-\gamma}}}\widetilde{\mathbb{E}}^{(\bar{q},\iota_0)}_{1,1,0,0}[f](\tau)\right\}+\nu^\kappa\int_0^\infty e^{\delta(\nu^\kappa  \tau
)^{\frac{s}{s-\gamma}}}\widetilde{\mathbb{D}}^{(\bar{q},\iota_0)}_{1,1,0,0}[f](\tau)\mathrm{d}\tau \lesssim\mathcal{M}^2\nu^{1-\kappa}.
  \end{equation}
\end{theorem}
\begin{proof} For brevity, we only prove the enhanced decay estimate \eqref{Th-nonlinear-energy-exponential-3} since the proof of \eqref{Th-nonlinear-energy-exponential-1} and \eqref{Th-nonlinear-energy-exponential-2} can be proceeded in the same way.

By an argument analogous to the proof of Theorem \ref{Th-basic-estimate}, we obtain the following inequality:
\begin{align}\label{11ED}
 &\frac{\mathrm{d}}{\mathrm{d}t}\widetilde{\mathbb{E}}^{(\bar{q},\iota_0)}_{1,1,0,0}[f](t)
 +\nu^\kappa \widetilde{\mathbb{D}}^{(\bar{q},\iota_0)}_{1,1,0,0}[f](t)\nonumber\\
 \lesssim& \langle t\rangle^{-2} \widetilde{\mathbb{E}}^{(\bar{q},\iota_0)}_{1,1,0,0}[f](t)
 +\nu^{\kappa} \left(\widetilde{\mathbb{E}}^{(\bar{q},\iota_0)}_{1,1,0,0}[f](t)
 \widetilde{\mathbb{D}}^{(\bar{q},\iota_0)}_{N_{max}-2}[f](t)
 \widetilde{\mathbb{D}}^{(\bar{q},\iota_0)}_{1,1,0,0}[f](t)\right)^{\frac12}\nonumber\\
 &+ \sum_{|\alpha|\leq 1}\left\|\partial^\alpha \rho_{\neq0}(t)\right\| \min\left\{\widetilde{\mathbb{E}}^{(\bar{q},\iota_0)}_{1,1,0,0}[f](t),
 \widetilde{\mathbb{D}}^{(\bar{q},\iota_0)}_{1,1,0,0}[f](t)\right\}.
\end{align}
Recalling the definition of $\widetilde{\mathbb{E}}^{(q,\iota_0)}_{1,1,0,0}[f](t)$ and $\widetilde{\mathbb{D}}^{(q,\iota_0)}_{1,1,0,0}[f](t)$ with $q$ replaced by $\bar{q}$, we define functions $g$ and $h$ as
\begin{align}
 &\int_{\mathbb{R}^3}g^2(t,v)\,\mathrm{d}v\equiv \widetilde{\mathbb{E}}^{(\bar{q},\iota_0)}_{1,1,0,0}[f](t)\nonumber\\
	&\qquad=
	\sum_{|\alpha|=1}\Big[A_0\sum_{|\alpha'|\leq 1}\left\|e^{\frac12\phi} \langle v\rangle^{\frac{-\gamma}{2s}(\iota_0-1-|\alpha'|)}\partial^{\alpha+\alpha'}_x f\right\|^2\nonumber\\
	&\qquad\quad+\nu^\kappa\int_{\mathbb{R}^3\times \mathbb{R}^3}e^{\phi}\langle v\rangle^{\frac{-\gamma}{2s}(\iota_0-2)}\nabla_x\partial^{\alpha}_x Y^\omega f\cdot \langle v\rangle^{\frac{-\gamma}{2s}(\iota_0-2)}\nabla_v\partial^{\alpha}_x  f\,\mathrm{d}v\mathrm{d}x\nonumber\\
	&\qquad\quad
+\nu^{2\kappa}\sum_{|\beta'|= 1}\left\|e^{\frac12\phi} \langle v\rangle^{\frac{-\gamma}{2s}(\iota_0-2)}\partial^{\alpha}_x\partial^{\beta'}_v Y^\omega f\right\|^2\nonumber\\
&\qquad\quad+A_0^{-1}\sum_{|\beta'|=2}\nu^{4\kappa}\left\|e^{\frac12\phi}\langle v\rangle^{\frac{-\gamma}{2s}(\iota_0-3)}\partial^{\alpha}_x\partial^{\beta+\beta'}_v  f\right\|^2\Big]\nonumber,\\
&\int_{\mathbb{R}^3}h^2(t,v)\,\mathrm{d}v\equiv \widetilde{\mathbb{D}}^{(\bar{q},\iota_0)}_{1,1,0,0}[f](t).\nonumber
\end{align}
By utilizing enhanced dissipation and the Poincar\'{e} inequality, we can obtain that
\begin{align*}
    \nu^{\kappa}\int_{\mathbb{R}^3}\langle v\rangle^{\frac{\gamma}{s}}g^2(t,v)\,\mathrm{d}v\lesssim \nu^\kappa \widetilde{\mathbb{D}}^{(\bar{q},\iota_0)}_{1,1,0,0}[f](t).
\end{align*}
Thus we can estimate \eqref{11ED} as
\begin{align}\label{11gh}
   & \frac{\mathrm{d}}{\mathrm{d}t}\int_{\mathbb{R}^3}g^2(t,v)\,\mathrm{d}v
   +\nu^\kappa\int_{\mathbb{R}^3}\langle v\rangle^{\frac{\gamma}{s}}g^2(t,v)\,\mathrm{d}v+\nu^\kappa\int_{\mathbb{R}^3}h^2(t,v)
   \,\mathrm{d}v\nonumber\\
\lesssim&\nu^{\kappa} \left(\widetilde{\mathbb{E}}^{(\bar{q},\iota_0)}_{1,1,0,0}[f](t)
 \widetilde{\mathbb{D}}^{(\bar{q},\iota_0)}_{N_{max}-2}[f](t)
 \widetilde{\mathbb{D}}^{(\bar{q},\iota_0)}_{1,1,0,0}[f](t)\right)^{\frac12}+\sum_{|\alpha|= 1}\|\partial^\alpha \rho_{\neq0}(t)\|\min\left\{\widetilde{\mathbb{E}}^{(\bar{q},\iota_0)}_{1,1,0,0}
 [f](t),\widetilde{\mathbb{D}}^{(\bar{q},\iota_0)}_{1,1,0,0}[f](t)\right\}.
\end{align}
According to Lemma \ref{decay-exponent}, it suffices to estimate the time integral of the two terms on the right-hand side of \eqref{11gh}. For the first term, using \eqref{priori-ener} and \eqref{priori-assu0}, an upper bound for the corresponding integral is given by
\begin{align*}
&\eta\nu^{\kappa}\int_0^t\widetilde{\mathbb{D}}^{(\bar{q},\iota_0)}_{1,1,0,0}[f](\tau)
e^{\delta(\nu^\kappa  \tau
)^{\frac{s}{s-\gamma}}}\,\mathrm{d}\tau+C(\eta) \nu^{\kappa}\sup_{\tau\in[0,t]}\left\{e^{\delta(\nu^\kappa  \tau
)^{\frac{s}{s-\gamma}}}\widetilde{\mathbb{E}}^{(\bar{q},\iota_0)}_{1,1,0,0}[f](\tau)\right\} \int_0^t\widetilde{\mathbb{D}}^{(\bar{q},\iota_0)}_{N_{max}-2}[f](\tau)
\,\mathrm{d}\tau\\
&\lesssim \eta\nu^{\kappa}\int_0^t\widetilde{\mathbb{D}}^{(\bar{q},\iota_0)}_{1,1,0,0}[f](\tau)
e^{\delta(\nu^\kappa  \tau
)^{\frac{s}{s-\gamma}}}\,\mathrm{d}\tau+C(\eta) \mathcal{M}^{2}\nu^{4\kappa}.
\end{align*}
For the second term, we use Theorem \ref{exp-dens} to obtain
\begin{align*}
  & \int_0^t\min\left\{\widetilde{\mathbb{E}}^{(\bar{q},\iota_0)}_{1,1,0,0}[f](\tau),
\widetilde{\mathbb{D}}^{(\bar{q},\iota_0)}_{1,1,0,0}[f](\tau)\right\}^{\frac12}
\sum_{|\alpha|= 1}\left\|\partial^{\alpha} \rho_{\neq0}(\tau)\right\|
e^{\delta(\nu^\kappa  \tau )^{\frac{s}{s-\gamma}}}\,\mathrm{d}\tau\nonumber\\
\lesssim&\int_0^t\sum_{|\alpha|= 1}\left\|\partial^\alpha \rho^{(1)}_{\neq0}(t)\right\|
e^{\delta(\nu^\kappa \tau
)^{\frac{s}{s-\gamma}}}\left\{\widetilde{\mathbb{E}}^{(\bar{q},\iota_0)}_{1,1,0,0}[f](t)\right\}
^{\frac12}\,\mathrm{d}\tau\\
&+\int_0^t
\left\{\widetilde{\mathbb{D}}^{(\bar{q},\iota_0)}_{1,1,0,0}[f](t)\right\}^{\frac12}
\sum_{|\alpha|= 1}\left\|\partial^\alpha \rho^{(2)}_{\neq0}(\tau)\right\|
 e^{\delta(\nu^\kappa  \tau
)^{\frac{s}{s-\gamma}}}\,\mathrm{d}\tau\nonumber\\
\lesssim&\eta \sup_{\tau\in[0,t]}\left\{e^{\delta(\nu^\kappa  \tau
)^{\frac{s}{s-\gamma}}}\widetilde{\mathbb{E}}^{(\bar{q},\iota_0)}_{1,1,0,0}[f](\tau)\right\}
+C_\eta \left\{\int_0^t\sum_{|\alpha|= 1}\left\|\partial^\alpha \rho^{(1)}_{\neq0}(\tau)\right\| e^{\frac\delta2(\nu^\kappa  \tau
)^{\frac{s}{s-\gamma}}}\,\mathrm{d}\tau\right\}^2\nonumber\\
&+\eta\nu^{\kappa}\int_0^t\widetilde{\mathbb{D}}^{(\bar{q},\iota_0)}_{1,1,0,0}[f](\tau)
e^{\delta(\nu^\kappa  \tau
)^{\frac{s}{s-\gamma}}}\,\mathrm{d}\tau+C_\eta\nu^{-\kappa}\int_0^t\sum_{|\alpha|= 1}\left\|\partial^\alpha \rho^{(2)}_{\neq0}(\tau)\right\|^2 e^{\delta(\nu^\kappa  \tau
)^{\frac{s}{s-\gamma}}}\,\mathrm{d}\tau\nonumber\\
\lesssim&\eta \sup_{\tau\in[0,t]}\left\{e^{\delta(\nu^\kappa  \tau
)^{\frac{s}{s-\gamma}}}\widetilde{\mathbb{E}}^{(\bar{q},\iota_0)}_{1,1,0,0}[f](\tau)\right\}
+C_\eta\mathcal{M}^2\nu^{2\kappa}\nonumber\\
&+\eta\nu^{\kappa}\int_0^t\widetilde{\mathbb{D}}^{(q,\iota_0)}_{1,1,0,0}[f](\tau)
e^{\delta(\nu^\kappa  \tau
)^{\frac{s}{s-\gamma}}}\,\mathrm{d}\tau+C_\eta\mathcal{M}^2\nu^{1-\kappa}\nonumber\\
\lesssim&\eta \sup_{\tau\in[0,t]}\left\{e^{\delta(\nu^\kappa  \tau
)^{\frac{s}{s-\gamma}}}\widetilde{\mathbb{E}}^{(\bar{q},\iota_0)}_{1,1,0,0}[f](\tau)\right\}+\eta\nu^{\kappa}\int_0^T\widetilde{\mathbb{D}}^{(q,\iota_0)}_{1,1,0,0}[f](t)e^{\delta(\nu^\kappa  t
)^{\frac{s}{s-\gamma}}}\,\mathrm{d}\tau+C_\eta\mathcal{M}^2\nu^{1-\kappa},
\end{align*}
where we have used the fact that $2\kappa\geq1- \kappa$.

Substituting the above estimates into \eqref{11gh} and applying Lemma \ref{decay-exponent}, we obtain
\begin{align*}
&\widetilde{\mathbb{E}}^{(\bar{q},\iota_0)}_{1,1,0,0}[f](t)
+\nu^{\kappa}\int_0^t\widetilde{\mathbb{D}}^{(\bar{q},\iota_0)}_{1,1,0,0}[f](\tau)
e^{\delta(\nu^\kappa  \tau
)^{\frac{s}{s-\gamma}}}\,\mathrm{d}\tau    \\
&\qquad \lesssim C_\eta\mathcal{M}^2\nu^{1-\kappa}+ \eta\left(\sup_{\tau\in[0,t]}\left\{e^{\delta(\nu^\kappa  \tau
)^{\frac{s}{s-\gamma}}}\widetilde{\mathbb{E}}^{(\bar{q},\iota_0)}_{1,1,0,0}[f](\tau)\right\}
+\nu^{\kappa}\int_0^t\widetilde{\mathbb{D}}^{(\bar{q},\iota_0)}_{1,1,0,0}[f](\tau)
e^{\delta(\nu^\kappa  \tau
)^{\frac{s}{s-\gamma}}}\,\mathrm{d}\tau\right),
\end{align*}
where the smallness of $\nu$ is used. Choosing $\eta$ sufficiently small then yields \eqref{Th-nonlinear-energy-exponential-3}.
This ends the proof of \eqref{Th-nonlinear-energy-exponential}.
\end{proof}

\subsection{The proof of Theorem 1.1}
Finally, we are in a position to complete the proof of Theorem \ref{Main-Th.}.
\begin{proof}[The proof of Theorem \ref{Main-Th.}]

In fact, \eqref{main1} and \eqref{main2} follow from our global existence result Theorem \ref{glob-exis}.  \eqref{main3} and \eqref{main10} follow from interpolation inequalities based on \eqref{main2} and estimates in Theorem \ref{exp-dens}. For the Landau damping estimate \eqref{main11}, it is a direct result by the interpolation between \eqref{Th-rho-1} in Theorem \ref{dens-esti} and \eqref{density-decay-1} via Parseval's theorem:
\begin{align*}
\left|\widehat{\rho}_k\right|(t)\lesssim& [k(1+t)]^{1-N_{max}}\sum_{|\alpha|+|\omega|\leq N_{max}-1}\|\partial^\alpha_xY^\omega \rho_{\neq 0}(t) \|\\
\lesssim&  [k(1+t)]^{1-N_{max}}\sum_{|\alpha|+|\omega|\leq N_{max}}\rho_{\neq 0}(t) \|^{\frac{1}{N_{max}}}\Big(\sum_{|\alpha|+|\omega|\leq N_{max}}\|\partial^\alpha_xY^\omega \rho_{\neq 0}(t) \|\Big)^{\frac{N_{max}-1}{N_{max}}}\\
\lesssim& \mathcal{M}\nu^{\kappa}\langle [k(1+t)]^{1-N_{max}}\min\left\{e^{-\delta_N(\nu^\kappa t)^{\frac{s}{s-\gamma}}},e^{-\delta_N(\nu t)^{\frac{1}{1-\gamma-2s}}}\right\},
\end{align*}
where $\delta_N$ is a proper constant less than $\delta$ in \eqref{density-decay-1}.
Therefore the proof of Theorem \ref{Main-Th.}
is finished.
\end{proof}


\section{Appendix}\label{Appendix}
\setcounter{equation}{0}
In this appendix, we collect some significant estimates used in the previous sections.


We begin by recalling some fundamental results on the weighted energy type estimates for the linearized collision operator $\mathcal{L}$ and the nonlinear term $\Gamma$ for non-cutoff cases, the proofs can be found in \cite{AMUXY-JFA-2012,DLYZ-KRM2013,FLLZ-SCM-2018,LLXZ-JLMS-2024}. Throughout this appendix, we denote \[w_{\ell,q}=\langle v\rangle^{\ell}e^{q\langle v\rangle}.\]

\begin{lemma}\label{L-noncut}
	Let $\ell\in \mathbb{R}$, $\eta>0$, $0<s<1$ and $\max\{-3,-\frac32-2s\}<\gamma<0$.
	\begin{itemize}
		\item [(i).] It holds that
		\begin{equation}\label{Lemma L_1}
			\langle\mathcal{L}g,g\rangle\gtrsim|{\bf{P}}^{\bot}g|_{L^2_D}^2.
		\end{equation}
        \item[(ii).] It holds that
        \begin{equation}\label{LemmaL-Y-omega}
            \langle Y^\omega\mathcal{L}g,Y^\omega g\rangle\gtrsim |{\bf P}^{\bot} Y^{\omega}g|_{L^2_D}^2
            -C_\eta\sum_{|\omega'|\leq |\omega|-1}|Y^{\omega'} g|_{L^2_D}^2.
        \end{equation}
		\item [(iii).] It holds that
		\begin{equation}\label{Lemma L_2}
			\left\langle w_{\ell,q}^{2}\mathcal{L}g, g\right\rangle\gtrsim |w_{\ell,q} g|_D^2-C|\chi_{\{| v|\leq C\}}g|^2_{L^2},
		\end{equation}
where $\chi_{\{| v|\leq C\}}$ is a smooth cutoff function defined as
\begin{align}
\chi_{\{| v|\leq C\}}=\left\{\begin{array}{rll}
&0,\ |v|\geq C+1,\\[2mm]
&1,\ |v|\leq C.\end{array}\right.\notag
\end{align}
		For $|\beta|\geq 1$, one has
		\begin{equation}\label{Lemma L_3}
			\left\langle w_{\ell,q}^{2}\partial _\beta \mathcal{L}g,\partial_\beta g\right\rangle\gtrsim \left|w_{\ell,q}\partial_\beta g\right|_{D}^2-C\sum_{\beta'<\beta}\left|w_{\ell,q}\partial_{\beta'} g\right|_{D}^2 -C|\chi_{\{| v|\leq C\}}g|^2_{L^2}.
		\end{equation}
        \item [(iv)] Moreover, for $\ell\geq 0$, $q\in [0, \infty)$, $|\beta|+|\omega|\geq1$, and some small constant $\eta>0$, it holds that
		\begin{align}\label{wL-lin}
			&\big(\partial^{\beta}_vY^{\omega}\mathcal{L}g,
 \langle v\rangle^{2\ell}\partial^{\beta}_vY^{\omega}g\big)\nonumber\\
			\geq& \big| \langle v\rangle^{\ell}\partial^{\beta}_vY^{\omega}g\big|_{ D}^2-\eta\sum_{\beta'\leq\beta}\sum_{\omega'\leq \omega}\big| \langle v\rangle^{\ell}\partial^{\beta'}_vY^{\omega'}g\big|_{ D}^2\nonumber\\
&-C(\eta)\sum_{\beta'\leq\beta,\omega'\leq \omega,\atop
|\beta'|+|\omega'|\leq|\beta|+|\omega|-1}\big| \langle v\rangle^{\ell}\partial^{\beta'}_vY^{\omega'}g\big|_{ D}^2-C(\eta)\sum_{\omega'\leq \omega}\big| \chi_{\{| v|\leq2C_{\eta}\}}Y^{\omega'}g\big|^2_{L^2}.
		\end{align}
	\end{itemize}
\end{lemma}
\begin{proof}
	\eqref{Lemma L_1}has been shown in \cite{AMUXY-JFA-2012}.
   \eqref{LemmaL-Y-omega} can be obtained from the commutator estimate (3.7)-(3.8) in \cite{BCD-PLMS-2024}.
	The relevant coercive estimate \eqref{Lemma L_2} and \eqref{Lemma L_3} with exponential weights can be found in \cite{DLYZ-KRM2013}. \eqref{wL-lin} can be obtained deduced by combing \eqref{Lemma L_2} and \eqref{Lemma L_3} and we omit its proof for brevity.
	
\end{proof}
\begin{lemma}\label{Gamma-noncut}\quad
	For all $0<s<1$, ${q}>0$ and  $\ell\geq 0$,
\begin{itemize}
	\item [i)] It holds that
	\begin{eqnarray}\label{Gamma-noncut-1}
		&&|\langle \partial^\alpha_x\partial^\beta_v Y^\omega\Gamma(f,g), w_{\ell,q}\partial^\alpha_x\partial^\beta_v Y^\omega h\rangle|\nonumber\\
		&\lesssim&\sum\bigg\{\left|w_{\ell,q}\partial^{\alpha_1}_x\partial^{\beta_1}_v Y^{\omega_1}f\right|_{L^2_
			{\frac\gamma2+s}}
		\left|\partial^{\alpha_2}_x\partial^{\beta_2}_v Y^{\omega_2}g\right|_{L^2_D}+\left|\partial^{\alpha_2}_x\partial^{\beta_2}_v Y^{\omega_2}g\right|_{L^2_{\frac\gamma2+s}}
		\left|w_{\ell,q}\partial^{\alpha_1}_x\partial^{\beta_1}_v Y^{\omega_1}f\right|_{L^2_D}
		\nonumber\\
		&&+\min\left\{\left|w_{\ell,q}\partial^{\alpha_1}_x\partial^{\beta_1}_v Y^{\omega_1}f\right|_{L^2_v}
		\left|\partial^{\alpha_2}_x\partial^{\beta_2}_v Y^{\omega_2}g\right|_{L^2_{\frac\gamma2+s}},\left|\partial^{\alpha_2}_x\partial^{\beta_2}_v Y^{\omega_2}g\right|_{L^2_v}
		\left|w_{\ell,q}\partial^{\alpha_1}_x\partial^{\beta_1}_v Y^{\omega_1}f\right|_{L^2_{\frac\gamma2+s}}\right\}
		\nonumber\\
		&&+\left|e^{{q}\langle v\rangle}\partial^{\alpha_1}_x\partial^{\beta_1}_v Y^{\omega_1}g\right|_{L^2_v}
		\left|w_{\ell,q}\partial^{\alpha_2}_x\partial^{\beta_2}_v Y^{\omega_2}f\right|_{L^2_{\frac\gamma2+s}}\bigg\}\left|w_{\ell,q}\partial^{\alpha}_x\partial^{\beta}_v Y^{\omega}h\right|_{L^2_D},
	\end{eqnarray}
	where the summation $\sum$ is taken over $\alpha_1+\alpha_2\leq \alpha$, $\beta_1+\beta_2\leq\beta$ and $\omega_1+\omega_2\leq\omega$.
	
	\item [ii)]	Furthermore, we have
	\begin{eqnarray}\label{Gamma-noncut-2}
		\left|\langle{\Gamma}(f,g), h\rangle\right|&\lesssim& \left\{|f|_{L^2_{\frac\gamma2+s}}|g|_{L^2_D}+|g|_{L^2_{\frac\gamma2+s}}|f|_{L^2_D}\right.\nonumber\\
		&&\left.+\min\left\{|f|_{L^2_v}|g|_{L^2_{\frac\gamma2+s}},|g|_{L^2}|f|_{L^2_{\frac\gamma2+s}}\right\}\right\}|h|_{L^2_D}.
	\end{eqnarray}

\end{itemize}	
\end{lemma}
\begin{proof}
When $\omega =0$, \eqref{Gamma-noncut-1} is shown as in \cite{FLLZ-SCM-2018} for the case of weak angular singularity $0<s<\frac12$ and in \cite{LLXZ-JLMS-2024} for the case of strong angular singularity $\frac12\leq s<1 $. When $\omega\neq 0$, by a slight modification, we can obtain \eqref{Gamma-noncut-1} and omit its proof for brevity. \eqref{Gamma-noncut-2} can be seen in \cite{AMUXY-JFA-2012}.
\end{proof}

\begin{lemma}\label{Lemma GammaL}(cf. \cite{Strain-KRM-2012})
For $0<s<1, \ell>0$, it holds that
\begin{equation*}\label{Lemma GammaL_1}
\left|\langle v\rangle^\ell\partial^\beta_v Y^\omega\Gamma(f,f)\right|_{L^2_v}\lesssim
\sum_{\beta_1+\beta_2\leq\beta,\atop\omega_1+\omega_2\leq \omega }\left|\langle v\rangle^\ell\partial^{\beta_1}Y^{\omega_1} f\right|_{L^2_{\gamma+2s}}
\left|\langle v\rangle^\ell\partial^{\beta_2}Y^{\omega_2} f\right|_{H^{2s}_{\gamma+2s}}.
\end{equation*}
\end{lemma}

We now state the following velocity weighted Sobolev interpolation inequality, which can be regarded as a modified version of the one in \cite[pp.27]{BCD-PLMS-2024}.
\begin{lemma}\label{ws-ine-lem}
For $\ell\leq0$, it holds that for $f\in C_0^\infty(\R^3)$
\begin{align}\label{ws-ine}
    |\lag v\rag^\ell\lag v\rag^{\frac{\ga}{2s}}f|_{L^2}\lesssim |\lag v\rag^\ell\lag v\rag^{\frac{\ga}{2}}f|^s_{\dot{H}^s}|\lag v\rag^\ell\langle v\rangle^{\frac{(1+s)\gamma}{2s}}\nabla_v f|^{1-s}_{\dot{H}^s}.
\end{align}
\end{lemma}
\begin{proof}
    We use the argument of dyadic decomposition of $\R^3$. Frist, we write
\begin{align}
    \lag v\rag^\ell\lag v\rag^{\frac{\ga}{2s}}f\sim{\FI}_{|v|\leq 1} 2^{-(\ell+\frac{\ga}{2s})}f+\sum\limits_{j=0}^{\infty}{\FI}_{2^j\leq|v|\leq 2^{j+1}} 2^{j(\ell+\frac{\ga}{2s})}f.\notag
\end{align}
Then, applying the fractional Gagliardo–Nirenberg interpolation inequality inequality on each dyadic shell yields
\begin{align}
\left|{\FI}_{|v|\leq 1}2^{-(\ell+\frac{\ga}{2s})}f\right|\leq \left|{\FI}_{|v|\leq 1}2^{-(\ell+\frac{\ga}{2})}f\right|_{\dot{H}^s}^s\left|{\FI}_{|v|\leq 1}2^{-(\ell+\frac{(1+s)\ga}{2s})}\na_vf\right|_{\dot{H}^s}^{1-s}.\notag
\end{align}
and
\begin{align}
\left|{\FI}_{2^j\leq|v|\leq 2^{j+1}}2^{j(\ell+\frac{\ga}{2s})}f\right|\leq \left|2^{j(\ell+\frac{\ga}{2})}f\right|_{\dot{H}^s}^s\left|{\FI}_{2^j\leq|v|\leq 2^{j+1}}2^{j(\ell+\frac{(1+s)\ga}{2s})}\na_vf\right|_{\dot{H}^s}^{1-s}.\notag
\end{align}
Summing over $j=0,1,2,\cdots$, then gives \eqref{ws-ine}.
This ends the proof of Lemma \ref{ws-ine-lem}.

\end{proof}

We conclude with two Strain-Guo type lemmas. First, we state the algebraic decay lemma from \cite{CLN-JAMS-2023}, and then we give a modified version that establishes stretched-exponential decay.

 \begin{lemma}\label{decay-algeb}
Let $g: [0, \infty)\times \mathbb{R}^3\rightarrow \mathbb{R}$ be a smooth function.
    Suppose that for given positive constants $m, \mathfrak{C}$, and $ \bar{\delta}$, the following holds:
    \begin{itemize}
        \item There is a uniform bound of the algebraically weighted estimate:
            \begin{equation}
            \sup_{0\leq t<\infty}\int_{\mathbb{R}^3}\langle v\rangle^{4m}g^2(t,v)\,\mathrm{d}v\leq \mathfrak{C}.\notag
        \end{equation}
        \item For $t\in [0, T)$,
         \begin{align}
\frac{\mathrm{d}}{\mathrm{d}t}\int_{\mathbb{R}^3}g^2(t,v)\,\mathrm{d}v+\bar{\delta}\int_{\mathbb{R}^3}\langle v\rangle^{-m}g^2(t,v)\,\mathrm{d}v\leq \left(\frac{243\pi}{2}+1\right)\mathfrak{C}\langle \bar{\delta}t\rangle^{-3}.\notag
        \end{align}
           \end{itemize}
   Then we have
    \begin{align}
        \int_{\mathbb{R}^3}g^2(t,v)\,\mathrm{d}v \leq C_{q_0,m} \mathfrak{C}.\notag
    \end{align}
\end{lemma}

For the reader’s convenience, we provide a modified proof of the following lemma, which corresponds to Lemma A.1 in \cite{CLN-JAMS-2023}, but with the second-order exponential weight function replaced by a first-order one.
\begin{lemma}\label{decay-exponent}
Let  $g: [0, T)\times \mathbb{R}^3\rightarrow \mathbb{R}$ be smooth for some given constant $T\in (0, \infty)$.
    Suppose that for given positive constants $q_0, \mathfrak{C}, \delta_0, \delta_1$, and $\tilde{q}\in (0, \frac{q_0}{2})$, the following holds:
    \begin{itemize}
        \item There is a uniform bound of the exponentially weighted estimate:
            \begin{equation}
            \sup_{0\leq t\leq T}\int_{\mathbb{R}^3}e^{q_0\langle v\rangle}g^2(t,v)\,\mathrm{d}v\leq \mathfrak{C}.\notag
        \end{equation}
        \item For proper functions $h:[0, T)\times \mathbb{R}^3\rightarrow \mathbb{R}$ and $\mathfrak{R}: [0, T)\rightarrow \mathbb{R}$ satisfying
         \begin{align}\label{decay-exponent-2}
\frac{\mathrm{d}}{\mathrm{d}t}\int_{\mathbb{R}^3}g^2(t,v)\,\mathrm{d}v+\delta_0\int_{\mathbb{R}^3}\langle v\rangle^{-m}g^2(t,v)\,\mathrm{d}v+\delta_1\int_{\mathbb{R}^3}h^2(t,v)\,\mathrm{d}v\leq \mathfrak{R}(t)
        \end{align}
        with
        \begin{align}
         \int_0^T e^{\tilde{q}\langle\delta_0 t\rangle^{\frac{1}{1+m}}}\mathfrak{R}(t)\,\mathrm{d}t\leq \mathfrak{C},\notag
        \end{align}
    \end{itemize}
    there exists a constant $C_{q_0,m}$ such that
    \begin{align}\label{decay-exponent-4}
        \sup_{0\leq t\leq T}e^{\tilde{q}(\delta_0 t)^{\frac{1}{1+m}}}\int_{\mathbb{R}^3}g^2(t,v)\,\mathrm{d}v+\delta_1\int_0^Te^{\tilde{q}\langle\delta_0 t\rangle^{\frac{1}{1+m}}}\int_{\mathbb{R}^3}h^2(t,v)\,\mathrm{d}v\mathrm{d}t\leq C_{q_0,m} \mathfrak{C}.
    \end{align}
\end{lemma}
\begin{proof}
    We first multiply \eqref{decay-exponent-2} by $e^{\tilde{q}(\delta_0 t)^{\frac{1}{1+m}}}$ to have
    \begin{align}\label{lemma-exponent-pr1}
     &\frac{\mathrm{d}}{\mathrm{d}t}\int_{\mathbb{R}^3}e^{\tilde{q}(\delta_0 t)^{\frac{1}{1+m}}}g^2(t,v)\,\mathrm{d}v+\delta_0e^{\tilde{q}(\delta_0 t)^{\frac{1}{1+m}}}\int_{\mathbb{R}^3}\langle v\rangle^{-m}g^2(t,v)\,\mathrm{d}v+\delta_1e^{\tilde{q}(\delta_0 t)^{\frac{1}{1+m}}}\int_{\mathbb{R}^3}h^2(t,v)d\,\mathrm{d}v\nonumber\\
     \lesssim& e^{\tilde{q}(\delta_0 t)^{\frac{1}{1+m}}}\mathfrak{R}(t)+e^{\tilde{q}(\delta_0 t)^{\frac{1}{1+m}}}\tilde{q}(\delta_0t)^{-\frac{m}{1+m}}\int_{\mathbb{R}^3}g^2(t,v)\,\mathrm{d}v.
    \end{align}
Splitting $\mathbb{R}^3_v$ into two regions $\langle v\rangle\leq (\delta_0t)^{\frac1{1+m}}$ and $\langle v\rangle\geq (\delta_0t)^{\frac1{1+m}}$ in the last term of \eqref{lemma-exponent-pr1}, we have
\begin{align}
  &e^{\tilde{q}(\delta_0 t)^{\frac{1}{1+m}}}\tilde{q}
  (\delta_0t)^{-\frac{m}{1+m}}\int_{\mathbb{R}^3}g^2(t,v)\,\mathrm{d}v\nonumber\\
  \lesssim&e^{\tilde{q}(\delta_0 t)^{\frac{1}{1+m}}}\tilde{q}(\delta_0t)^{-\frac{m}{1+m}}\int_{\langle v\rangle\leq (\delta_0t)^{\frac1{1+m}}}g^2(t,v)\,\mathrm{d} v+e^{\tilde{q}(\delta_0 t)^{\frac{1}{1+m}}}\tilde{q}(\delta_0t)^{-\frac{m}{1+m}}\int_{\langle v\rangle\geq (\delta_0t)^{\frac1{1+m}}}g^2(t,v)\,\mathrm{d} v\nonumber\\
\lesssim&\tilde{q}e^{\tilde{q}\langle t\rangle^{\frac{1}{1+m}}}\int_{\mathbb{R}^3} \langle v\rangle^{-m} g^2(t,v)\,\mathrm{d} v+e^{\tilde{q}(\delta_0 t)^{\frac{1}{1+m}}}\tilde{q}(\delta_0t)^{-\frac{m}{1+m}}\int_{\langle v\rangle\geq (\delta_0t)^{\frac1{1+m}}}g^2(t,v)\,\mathrm{d} v\nonumber\\
\lesssim&\tilde{q}e^{\tilde{q}\langle t\rangle^{\frac{1}{1+m}}}\int_{\mathbb{R}^3} \langle v\rangle^{-m} g^2(t,v)\,\mathrm{d} v+e^{\tilde{q}(\delta_0 t)^{\frac{1}{1+m}}}\tilde{q}(\delta_0t)^{-\frac{m}{1+m}}\int_{\mathbb{R}^3}e^{-q_0(\delta_0t)^{-\frac1{1+m}}}e^{q_0\langle v\rangle}g^2(t,v)\,\mathrm{d} v\nonumber\\
\lesssim&\tilde{q}e^{\tilde{q}\langle t\rangle^{\frac{1}{1+m}}}\int_{\mathbb{R}^3} \langle v\rangle^{-m} g^2(t,v)\,\mathrm{d} v+\tilde{q}(\delta_0t)^{-\frac{m}{1+m}}e^{-\frac{q_0}2(\delta_0t)^{-\frac1{1+m}}}\int_{\mathbb{R}^3}e^{q_0\langle v\rangle}g^2(t,v)\,\mathrm{d} v.\nonumber
\end{align}
Substituting the above inequality into \eqref{lemma-exponent-pr1},  one can deduce \eqref{decay-exponent-4} as follows:
\begin{align}
&\sup_{0\leq t\leq T}e^{\tilde{q}(\delta_0 t)^{\frac{1}{1+m}}}\int_{\mathbb{R}^3}g^2(t,v)\,\mathrm{d}v+\delta_1\int_0^Te^{\tilde{q}\langle\delta_0 t\rangle^{\frac{1}{1+m}}}\int_{\mathbb{R}^3}h^2(t,v)\,\mathrm{d}v\mathrm{d}t\nonumber\\
\leq&\int_{\mathbb{R}^3}g^2(0,v)\,\mathrm{d}v+\int_0^T e^{\tilde{q}(\delta_0 t)^{\frac{1}{1+m}}}\mathfrak{R}(t)\,\mathrm{d}t\nonumber\\
&+\int_0^T\tilde{q}(\delta_0t)^{-\frac{m}{1+m}}e^{-\frac{q_0}2(\delta_0t)^{-\frac1{1+m}}}\int_{\mathbb{R}^3}e^{q_0\langle v\rangle}g^2(t,v)\,\mathrm{d} v\mathrm{d}t\nonumber\\
\leq &2\mathfrak{C}+C_{q_0,m}\mathfrak{C}\nonumber
\leq C_{q_0,m}\mathfrak{C}.
\end{align}
Thus the proof of Lemma \ref{decay-exponent} is complete.
\end{proof}

\bigbreak
\begin{center}
{\bf Acknowledgment}
\end{center}
The research of YJL was supported by the National Natural Science Foundation of China (Project No.~12171176). The research of SQL was supported by the grant from the National Natural Science Foundation of China (Project No.~12325107). The work of QHX was supported by the National Natural Science Foundation of China (Project No.~12271506) and the National Key Research and Development Program of
China (Project No.~2020YFA0714200). The research of HJZ was supported by three grant from National Natural Science Foundation of China  (Project No.~ 12221001, 12371225 and 12571242 ).

\medskip
\noindent\textbf{Data Availability Statement:}
Data sharing is not applicable to this article as no datasets were generated or analysed during the current study.

\noindent\textbf{Conflict of Interest:}
The authors declare that they have no conflict of interest.

\end{document}